\documentclass[reqno]{amsart}
  \usepackage{mathabx}
\usepackage{cite}
\usepackage{color}

\usepackage{geometry}
\usepackage[linktocpage]{hyperref}
\usepackage{appendix}
\allowdisplaybreaks

\theoremstyle{plain}
\newtheorem{thm}{Theorem}[section]
\newtheorem{cor}[thm]{Corollary}
\newtheorem{lem}[thm]{Lemma}
\newtheorem{prop}[thm]{Proposition}

\theoremstyle{definition}

\theoremstyle{remark}
\newtheorem{rem}{Remark}[section]
\numberwithin{equation}{section}

\newcommand{\p}{{\partial}}

\newcommand{\A}{\mathcal A}
\newcommand{\B}{\mathcal B}

\newcommand{\h}{\mathcal H}
\newcommand{\cp}{\mathcal P}

\newcommand{\bR}{\mathbb R}
\newcommand{\bT}{\mathbb T}
\newcommand{\N}{\mathbb N}

\newcommand\ue{\mathbf{u}^\epsilon}
\newcommand\he{\mathbf{H}^\epsilon}
\newcommand{\ep}{\epsilon}

\newcommand{\ly}{\langle y\rangle}

\begin{document}

\title[ local well-posedness theory for MHD boundary layer]
 {MHD boundary layers theory in Sobolev spaces without monotonicity. \uppercase\expandafter{\romannumeral1}. well-posedness theory}
\author[Cheng-Jie Liu]{Cheng-Jie Liu}
\address{Cheng-Jie Liu
\newline\indent
Department of Mathematics,
City University of Hong Kong,
Tat Chee Avenue, Kowloon, Hong Kong}
\email{cjliusjtu@gmail.com}

\author[Feng Xie]{Feng Xie}
\address{Feng Xie
\newline\indent
School of Mathematical Sciences, and LSC-MOE,
 Shanghai Jiao Tong University,
Shanghai 200240, P. R. China}
\email{tzxief@sjtu.edu.cn}

\author[Tong Yang]{Tong Yang}
\address{Tong Yang
\newline\indent
Department of Mathematics,
City University of Hong Kong,
Tat Chee Avenue, Kowloon, Hong Kong
\newline\indent
Department of Mathematics,
Jinan University, Guangzhou 510632, P. R. China
}
\email{matyang@cityu.edu.hk}

\begin{abstract}
We study the well-posedness theory for the MHD boundary layer. 
The boundary layer equations are governed by the Prandtl type equations
that  are derived from the incompressible
MHD system with non-slip boundary condition on the velocity and perfectly conducting condition on the magnetic field. Under the assumption that the initial tangential magnetic field is not zero, we  establish the local-in-time existence, uniqueness of solution for the nonlinear MHD boundary layer equations. Compared with the well-posedness theory of the classical Prandtl equations
for which the monotonicity condition of the tangential velocity plays a crucial role, this monotonicity condition is not needed for MHD boundary layer. 
This justifies  the physical understanding that the magnetic field has a stabilizing effect on MHD boundary layer in rigorous mathematics.
\end{abstract}



\maketitle

\tableofcontents

%-----------------section one-------------------------------------------------------------
%\renewcommand{\theequation}{\thesection.\arabic{equation}}
%\setcounter{equation}{0}
\section{Introduction and Main Result} \label{S1}
One important problem about Magnetohydrodynamics(MHD) is to understand the 
high Reynolds numbers limits in a domain with boundary. 
In this paper, we consider the following initial boundary value problem for the 
two dimensional (2D) viscous MHD equations (cf. \cite{Cow, davidson, D-L, S-T}) in a periodic domain $\{(t,x,y):~t\in[0,T], x\in\bT, y\in\bR_+\}:$
\begin{align}\label{eq_mhd}
\left\{
\begin{array}{ll}
\p_t\ue+(\ue\cdot\nabla)\ue-(\he\cdot\nabla)\he%(\nabla\times H)\times H
+\nabla p^\ep=\mu\epsilon\triangle\ue,\\
%\nabla\cdot\ue=0,\\
\p_t\he%+\ue\cdot\nabla H-H\cdot\nabla\ue
-\nabla\times(\ue\times \he)
=\kappa\epsilon\triangle \he,\\
\nabla\cdot\ue=0,\quad\nabla\cdot \he=0.
\end{array}
\right.
\end{align}
Here, we assume the viscosity and resistivity coefficients have the same order of a small parameter $\epsilon$. $\ue=(u^\ep_1,u^\ep_2)$ denotes the velocity vector, $\he=(h^\ep_1,h^\ep_2)$ denotes the magnetic field, and $p^\epsilon=\tilde p^\epsilon+\frac{|\he|^2}{2}$ denotes the total pressure with $\tilde p^\epsilon$ the pressure of the fluid.  %with the tangential variable $x\in\mathbb{T}$ and the normal variable $y\in\mathbb{R}_+$.
% Here $\nabla\cdot(f_1,f_2)=\p_xf_1+\p_yf_2$ and $\nabla\times(f_1,f_2)=\p_yf_1-\p_xf_2$ denote the \textit{div} and \textit{curl} operators respectively. 
On the boundary,  the non-slip boundary condition is imposed on velocity field
\begin{align}
 \label{bc_u}
 \ue|_{y=0}=\bf{0},
 \end{align}
and the perfectly conducting boundary condition on magnetic field
 \begin{align} \label{bc_h}
 h_2^\ep|_{y=0}=\p_yh_1^\ep|_{y=0}=0.
 \end{align}
 The formal limiting system of \eqref{eq_mhd} yields the ideal MHD equations
when $\epsilon$ tends to zero. However, there is a mismatch in the tangential velocity between the equations \eqref{eq_mhd} and the limiting equations on the boundary $y=0$. This is  why a boundary layer forms in the vanishing viscosity and resistivity limit process. To find out the terms in \eqref{eq_mhd} whose contribution is essential for the boundary layer,
we use the same scaling as  the
one used in \cite{OS},
 \begin{align*}
 t=t,\quad x=x,\quad \tilde y=\epsilon^{-\frac{1}{2}}y,
 \end{align*}
then, set
\begin{align*}
\left\{
\begin{array}{ll}
u_1(t,x,\tilde y)=u_1^\ep(t,x, y),\\
u_2(t,x,\tilde y)=\epsilon^{-\frac{1}{2}}u^\ep_2(t,x, y),
\end{array}
\right.
\qquad
\left\{
\begin{array}{ll}
h_1(t,x,\tilde y)=h_1^\ep(t,x, y),\\
h_2(t,x,\tilde y)=\epsilon^{-\frac{1}{2}}h_2^\ep(t,x, y),
\end{array}
\right.
\end{align*}
and
\[p(t,x,\tilde y)=p^\ep(t,x,y).\]
Then by taking  the leading order, the equations \eqref{eq_mhd} are reduced to  
\begin{align}
\label{eq_bl}
\left\{
\begin{array}{ll}
\partial_tu_1+u_1\partial_xu_1+u_2\partial_yu_1-h_1\partial_xh_1-h_2\partial_yh_1+\p_xp=\mu\partial^2_yu_1,\\
\p_yp=0,\\
\partial_th_1+\partial_y(u_2h_1-u_1h_2)=\kappa\partial_y^2h_1,\\
\partial_th_2-\partial_x(u_2h_1-u_1h_2)=\kappa\partial_y^2h_2,\\
\partial_xu_1+\partial_yu_2=0,\quad \partial_xh_1+\partial_yh_2=0,
\end{array}
\right.
\end{align}
in $\{t>0, x\in\mathbb{T},y\in\mathbb{R}^+\}$, where we have replaced $\tilde y$ by $y$ for simplicity of notations.

The second equation of \eqref{eq_bl} implies that the leading order of boundary layers for the total pressure $p^\epsilon(t,x,y)$ is invariant across the boundary layer, and should be matched to the outflow pressure $P(t,x)$ on  top of boundary layer, that is, 
the trace of pressure of ideal MHD flow. Consequently, we have
\[p(t,x,y)~\equiv~P(t,x).\]
It is worth noting that the pressure $\tilde p^\epsilon$ of the fluid may have the leading order of boundary layers because of the appearance of the boundary layer for magnetic field. It is different from the general fluid in the absence of magnetic field, for which the leading boundary layer for the pressure of the fluid always vanishes.

The tangential component $u_1(t,x,y)$ of velocity field, respectively $h_1(t,x,y)$ of magnetic filed, should match the outflow tangential velocity $U(t,x)$, respectively the outflow tangential magnetic
field  $H(t,x)$, on the top of boundary layer, that is,
 \begin{equation}\label{bc_infty} 
 	u_1(t,x,y)~\rightarrow~U(t,x),\quad h_1(t,x,y)~\rightarrow~H(t,x),\quad{\mbox as}\quad y~\rightarrow~+\infty,
 \end{equation}
 where $U(t,x)$ and $H(t,x)$ are the trace of tangential velocity and magnetic
field respectively. Therefore, we have the following ``matching'' condition:
\begin{align}\label{Brou}
	U_t+UU_x-HH_x+P_x=0,\quad H_t+UH_x-HU_x=0,
	\end{align}
which shows that \eqref{bc_infty} is consistent with the first and third equations of \eqref{eq_bl}. 	
Moreover, on the boundary $\{y=0\}$, the boundary conditions \eqref{bc_u} and \eqref{bc_h} give
\begin{align}
\label{bc_bl}
u_1|_{y=0}=u_2|_{y=0}=\p_yh_1|_{y=0}=h_2|_{y=0}=0.
\end{align}

On the other hand, it is noted that equation $\eqref{eq_bl}_4$ is a direct consequence of equations $\eqref{eq_bl}_3$, $\p_xh_1+\p_yh_2=0$ in $(\ref{eq_bl})_5$ and the boundary condition (\ref{bc_bl}). Hence, we only need
to study the following initial-boundary value problem of the MHD boundary layer equations in $\{t\in[0,T], x\in\mathbb{T},y\in\mathbb{R}^+\}$,
\begin{align}
\label{bl_mhd}
\left\{
\begin{array}{ll}
\partial_tu_1+u_1\partial_xu_1+u_2\partial_yu_1-h_1\partial_xh_1-h_2\partial_yh_1=\mu\partial^2_yu_1-P_x,\\
\partial_th_1+\partial_y(u_2h_1-u_1h_2)=\kappa\partial_y^2h_1,\\
\partial_xu_1+\partial_yu_2=0,\quad \partial_xh_1+\partial_yh_2=0,\\
u_1|_{t=0}=u_{10}(x,y),\quad h_1|_{t=0}=h_{10}(x,y),\\
(u_1,u_2,\partial_yh_1,h_2)|_{y=0}=\textbf{0},\quad
\lim\limits_{y\rightarrow+\infty}(u_1,h_1)=(U,H)(t,x).%\quad \lim\limits_{y\rightarrow+\infty}h_1=H(t,x).
\end{array}
\right.
\end{align}

The aim of this paper is to show the local well-posedness of the system \eqref{bl_mhd} with non-zero tangential component of that magnetic field, that is,
without loss of generality, by assuming
\begin{equation}
	\label{ass_m}
	h_1(t,x,y)>0.
\end{equation}
Let us first introduce some weighted Sobolev spaces for later use. Denote
 \[\Omega~:=~\big\{(x,y):~x\in\bT,~y\in\bR_+\big\}.\]
For any $l\in\bR,$  denote by $L_l^2(\Omega)$ the weighted Lebesgue space with respect to the spatial variables: 
\[L_l^2(\Omega)~:=~\Big\{f(x,y):~\Omega\rightarrow\bR,~
\|f\|_{L^2_l(\Omega)}:=\Big(\int_{\Omega}\ly^{2l}|f(x,y)|^2dxdy\Big)^{\frac{1}{2}}<+\infty\Big\},\qquad \ly~=~1+y,
\] 
and then, for any given $m\in\mathbb{N},$ denote by $H_l^m(\Omega)$ the weighted Sobolev spaces:
$$H_l^m(\Omega)~:=~\Big\{f(x,y):~\Omega\rightarrow\bR,~\|f\|_{H_l^m(\Omega)}:=\Big(\sum_{m_1+m_2\leq m}\|\ly^{l+m_2}\p_x^{m_1}\p_y^{m_2}f\|_{L^2(\Omega)}^2\Big)^{\frac{1}{2}}<+\infty\Big\}.$$

Now, we can state the main result as follows.
\begin{thm}\label{Th1}
Let $m\geq5$ be a integer, and $l\geq0$ a real number. Assume that the outer flow $(U,H,P_x)(t,x)$
%$$(U,H,P_x)(t,x),\qquad (t,x)\in[0,T]\times\bR, \quad\mbox{for some}\quad T>0$$ %is smooth and 
satisfies that for some $T>0,$ %$H(t,x)\geq2\delta_0$ for some positive constant $\delta_0,$ and
\begin{equation}\label{ass_outflow}
	M_0~:=~\sum_{i=0}^{2m+2}\Big(\sup_{0\leq t\leq T}\|\p_t^i(U,H,P)(t,\cdot)\|_{H^{2m+2-i}(\bT_x)}+\|\p_t^i(U,H,P)\|_{L^2(0,T;H^{2m+2-i}(\bT_x))}\Big)<+\infty.
	\end{equation}
%and\begin{equation*} 
%H(t,x)\geq2\delta_0,\qquad (t,x)\in[0,T]\times\bR
%\end{equation*}for some positive constant $\delta_0>0.$
Also, we suppose the initial data $(u_{10},h_{10})(x,y)$ satisfies
 \begin{equation}\label{ass_ini}
 	\Big(u_{10}(x,y)-U(0,x),h_{10}(x,y)-H(0,x)\Big)\in H^{3m+2}_l(\Omega),
 \end{equation} 
 and the compatibility conditions up to $m$-th order. Moreover, there exists a sufficiently small constant $\delta_0>0$ such that 
 \begin{align}\label{ass_bound}
 \big|\ly^{l+1}\p_y^i(u_{10}, h_{10})(x,y)\big|\leq(2\delta_0)^{-1},	\qquad h_{10}(x,y)\geq2\delta_0,\quad\mbox{for}\quad i=1,2,~ (x,y)\in\Omega.
 \end{align}
 %$|b_{0}|\leq\frac{1}{4},$
%\begin{align}\label{3.2}
%\|u_{\epsilon 0}, b_{\epsilon 0}\|_{\tilde{H}^{2m}_l(\mathbb{R}^2_+)}\leq \sigma_0,\end{align}
%for some $\sigma_0>0$ small enough. And 
%and the shear flow $(u_s(t,y)-U)\in H^{m+2}_l([0,\infty)\times\mathbb{R}^2_+)$. 
Then, there exist a postive time $0<T_*\leq T$ and a unique solution $(u_1,u_2, h_1,h_2)$ to the initial boundary value problem (\ref{bl_mhd}), such that %$h_1(t,x,y)\geq\delta_0$ for $(t,x,y)\in[0,T]\times\Omega,$
\begin{align}\label{est_main1}
	(u_1-U,h_1-H)\in\bigcap_{i=0}^mW^{i,\infty}\Big(0,T_*;H_l^{m-i}(\Omega)\Big),%\quad(\p_yu_1,\p_yh_1)\in \bigcap_{i=0}^mH^{i}\Big(0,T;H_l^{m-i}(\Omega)\Big),
\end{align}
and
\begin{align}\label{est_main2}
(u_2+U_xy,h_2+H_xy)&\in\bigcap_{i=0}^{m-1}W^{i,\infty}\Big(0,T_*;H_{-1}^{m-1-i}(\Omega)\Big),\nonumber\\%\qquad\mbox{for}\quad \lambda>\frac{1}{2},\nonumber\\
	(\p_yu_2+U_x,\p_yh_2+H_x)&\in\bigcap_{i=0}^{m-1}W^{i,\infty}\big(0,T_*;H_l^{m-1-i}(\Omega)\big).
	%&(u_2+U_xy,h_2+H_xy)\in\bigcap_{i=0}^{m-1}W^{i,\infty}\Big(0,T;L^\infty\big(\bR_{y,+};H^{m-1-i}(\bT_x)\big)\Big),\nonumber\\
	%&(\p_yu_2+U_x,\p_yh_2+H_x)\in\bigcap_{i=0}^{m-1}W^{i,\infty}\big(0,T;H_l^{m-1-i}(\Omega)\big).
\end{align}
Moreover, if $l>\frac{1}{2},$
\begin{align}\label{est_main3}
&(u_2+U_xy,h_2+H_xy)\in\bigcap_{i=0}^{m-1}W^{i,\infty}\Big(0,T_*;L^\infty\big(\bR_{y,+};H^{m-1-i}(\bT_x)\big)\Big).%\qquad\mbox{for}\quad l>\frac{1}{2},\nonumber\\
	%&(u_2+U_xy,h_2+H_xy)\in\bigcap_{i=0}^{m-1}W^{i,\infty}\Big(0,T;H_{l-1}^{m-1-i}(\Omega)\Big),\qquad\qquad\qquad\mbox{for}\quad 0\leq l<\frac{1}{2},\nonumber\\
\end{align}
%\begin{equation}\label{positive_h}h_1(t,x,y)>0,\quad t\in[0,T],~(x,y)\in\Omega.\end{equation}
\end{thm}

\begin{rem}
	Note that the regularity assumption on the outflow $(U,H,P)$ and the initial data $(u_{10}, h_{10})$ is not optimal. Here, we need the regularity  to simplify the construction of approximate solution, cf. Section 4. One may relax the regularity requirement by using  other approximations.
\end{rem}

\iffalse
\begin{rem}
	The result in Theorem \ref{Th1} can also be extended to the case on the half plane, i.e., $(x,y)\in\bR_+^2$ under some extra assumption on the outflow $(U,H,P)$, such as
	\[(U,H,P)(t,\cdot)\in L^\infty(\bR_x),\quad\p_t^i\p_x^j(U,H,P)(t,\cdot)\in L^2(\bR_x),\quad \mbox{for}\quad i+j\geq1.\]
\end{rem}
\fi

We now review some related works to the problem studied in this paper. First of all, the study on fluid around a rigid body with high Reynolds numbers is 
an important problem in both  physics and mathematics. The classical work
can be traced back to Prandtl in 1904 about the derivation of the
Prandtl equations for boundary layers from the incompressible Navier-Stokes
equations with non-slip boundary condition, cf. \cite{P}. About sixty years after
its derivation, the first systematic work in rigorous mathematics
was achieved by Oleinik, cf. \cite{O}, in which she showed
that under the monotonicity condition on the tangential velocity
field in the normal direction to the boundary, local in time
well-posedness of the Prandtl system can be justified in 2D by using 
the Crocco tranformation. 
This result
together with some  extensions are presented in
Oleinik-Samokhin's classical
book \cite{OS}. Recently, this well-posedness result was proved by using 
simply energy method  in the framework of
 Sobolev spaces in \cite{AWXY} and \cite{MW1} independently by taking care
of the cancellation in the convection terms to overcome
the loss of derivative in the tangential direction. 
%The method developed in \cite{AWXY} is also used in \cite{WXY}) for the Prandtl equations 
%derived from the compressible Naviver-Stokes equations. 
Moreover, by imposing an additional
favorable condition on the pressure,
a global in time weak solution was obtained in
\cite{XZ}. Some three space dimensional cases were studied for both classical and weak solutions in \cite{LWY1,LWY2}.
Since Oleinik's classical work,  the necessity of the monotonicity condition 
on the velocity field for well-posedness remained as
a question until 1980s when Caflisch and Sammartino \cite{SC1, SC2} obtained
the well-posedness in the framework of analytic functions without this
condition, cf. \cite{IV,KV, KMVW, LCS, Mae, ZZ} and the references therein. And recently, the analyticity condition can be further relaxed
to Gevrey regularity, cf. \cite{GM, GMM, LWX, L-Y}. %by G\'evard-Varet-Masmoudi and Li-Yang.

When the monotonicity condition is violated, separation of the boundary
layer is expected and observed for classical fluid. For this, E-Engquist constructed a finite
time blowup solution to the Prandtl
equations in \cite{EE}. Recently, when the background shear flow has a non-degenerate critical point, some interesting ill-posedness (or instability) phenomena of solutions to both
the linear and nonlinear Prandtl equations around
the shear flow are studied, cf.
\cite{GD,GN,G,GN1,LWY, LY} and the references therein. All these results show that the monotonicity
 assumption on the tangential velocity is essential for the well-posedness
except in the framework of analytic functions or Gevrey functions.

On the other hand, for electrically conducting fluid such as plasmas
and  liquid metals, the system of magnetohydrodynamics(denoted by MHD)  is a fundamental
system to describe the movement of  fluid under the influence of
 electro-magnetic field. The study on the MHD was initiated by
Alfv\'en \cite{Alf} who showed that the magnetic field can induce current in 
a moving conductive fluid with a new propagation 
mechanism along the magnetic field, called Alfv\'en waves.

%The physical stability studies of the MHD boundary layer show that the magnetic field has a stabilizing effect. For the case that the magnetic field is parallel to the flow, Rossow in \cite{R} and Arkhipov in \cite{A} found out
%that the minimum critical Reynolds number is more than double for an interaction parameter somewhat less that unity. However, for the corresponding free boundary layer, which is always unstable without the magnetic effect. It was also found in \cite{D} that the flow may be stabilized for a finite wave by sufficiently strong magnetic field. 

For plasma, the boundary layer equations can be derived from the fundamental
MHD system and they are more complicated than the classical Prandtl system because
of the coupling of the magnetic field with velocity field through the
Maxwell equations. On the other hand, in physics, it is believed
that the magnetic field has a stabilizing effect on the boundary layer that
could provide a mechanism for containment of, for example, 
the high temperature gas. If the magnetic field is transversal to the boundary,
there are extensive discussions on the so called Hartmann boundary layer,
cf. \cite{davidson, Har, H-L}. In addition, there are works on the stability of boundary
 layers
with minimum Reynolds number for flow with different structures to reveal
the difference from the classical boundary layers without electro-magnetic
field, cf. \cite{A,D,R}.

In terms of mathematical derivation when the non-slip boundary condition for the velocity is present, the boundary layer systems that capture the
leading order of fluid variables around the boundary depend on three
physical parameters, magnetic Reynolds number, Reynolds number
and their ratio called magnetic Prandtl number. When the Reynolds
number tends to infinity while the magnetic Reynolds number
is fixed, the derived boundary layer system is similar to the Prandtl
system for classical fluid and its well-posedness was discussed in
Oleinik-Samokhin's book \cite{OS}, for which the monotonicity condition on 
the velocity field is needed. When the Reynolds number is fixed while 
the magnetic Reynolds number tends to infinity that corresponds
to infinite magnetic Prandtl number, the boundary layer system is similar to inviscid 
Prandtl system and the monotonicity condition on the velocity field is not
needed for well-posedness. The case with finite magnetic Prandtl number when
both the Reynolds number and magnetic Reynolds number tend to infinity
at the same rate, the boundary layer system is totally different from
the classical Prandtl system, and this is the system to be discussed in
this paper. Note that for this system, there are no any mathematical well-posedness results obtained so far in the Sobolev spaces. Furthermore, we mention that in \cite{XXW}, the authors establish the vanishing viscosity limit for the MHD system in a bounded smooth domain of $\mathbb{R}^d, d=2,3$ with a slip boundary condition, while the leading order of boundary layers for both velocity and magnetic field vanishes because of the slip boundary conditions.

Precisely, in this paper, to capture the stabilizing effect of the magnetic field,
we  establish the well-posedness theory for the  problem (\ref{bl_mhd}) without any monotonicity assumption on the tangential velocity. The only essential condition is that
 the background tangential magnetic field has a lower positive bound. 
Hence, the result in this paper enriches
the classical local well-posedness results of the classical Prandtl equations. 
In the same time, it is  in agreement with the general physical understanding
that the magnetic field stabilizes the boundary layer. 

The rest of the paper is organized as follows. Some preliminaries are given in Section 2. In Section 3, we
establish the a priori energy estimates for the nonlinear problem (\ref{bl_mhd}). The
local-in-time existence and uniqueness of the solution to (\ref{bl_mhd}) in Sobolev space are given in Section 4. In Section 5, we introduce another method for
the study on  the well-posedness theory for (\ref{bl_mhd}) by using
a nonlinear coordinate transform in the spirit of Crocco transformation
for the classical Prandtl system. Finally, some technical proof of a
lemma is given in the Appendix.

%%%%%%%%%%%%%%%%%%%%%%--------Preliminaries

\section{Preliminaries}

Firstly, we introduce some notations. 
Use the tangential derivative operator
 $$\p^\beta_{\tau}=\p^{\beta_1}_t\p^{\beta_2}_x,\quad\mbox{for}\quad\beta=(\beta_1,\beta_2)\in\N^2,\quad|\beta|=\beta_1+\beta_2,$$%\quad \mbox{for a 2-tuple of nonnegative integers}\quad\beta=(\beta_1,\beta_2),~|\beta|=\beta_1+\beta_2,$ 
 and then denote the derivative operator (in both time and space) by
  $$\quad D^\alpha=\p_\tau^\beta\p_y^k, \quad\mbox{for}\quad\alpha=(\beta_1,\beta_2,k)\in\N^3, \quad|\alpha|=|\beta|+k.$$
  Set $e_i\in\N^2, i=1,2$ and $E_j\in\N^3, j=1,2,3$ by
  $$e_1=(1,0)\in\N^2, ~e_2=(0,1)\in\N^2,~E_1=(1,0,0)\in\N^3, ~E_2=(0,1,0)\in\N^3, ~E_3=(0,0,1)\in\N^3,$$
   and denote by $\p_y^{-1}$ the inverse of derivative $\p_y$, i.e., $(\p_y^{-1}f)(y):=\int_0^yf(z)dz.$
Moreover, we use the notation $[\cdot,\cdot]$ to denote the commutator, and denote  a nondecreasing  polynomial function
by $\cp(\cdot)$, which may differ from line to line.%and for $m\in\mathbb{N},$ denote by $\p_\tau^m$ the summation of tangential derivatives $\p_\tau^\beta$ for all $|\beta|\leq m$, and $D^m$ the summation of derivatives $D^\alpha$ for all $|\alpha|\leq m$. 

For $m\in\mathbb{N},$ define the function spaces $\h_l^m$ of measurable functions $f(t,x,y): [0,T]\times\Omega\rightarrow\bR,$ such that for any $t\in[0,T],$ 
\begin{align}\label{def_h}
\|f(t)\|_{\h_l^m}~:=~
	\Big(\sum_{|\alpha|\leq m}\|\ly^{l+k}D^\alpha f(t,\cdot)\|_{L^2(\Omega)}^2\Big)^{\frac{1}{2}}<+\infty.
	\end{align}
%Obviously, it follows that
%\[\A_l^m(T)~=~\bigcap_{i=0}^mH^{i}\big(0,T;H_l^{m-i}(\Omega)\big),\quad \B_l^m(T)~=~\bigcap_{i=0}^mW^{i,\infty}\big(0,T;H_l^{m-i}(\Omega)\big). \]
\iffalse
For $m\in\mathbb{N},$ we define the function spaces $\A_l^m(T)$ and $\B_l^m(T)$ of measurable functions $f(t,x,y): [0,T]\times\Omega\rightarrow\bR,$ such that 
\[\begin{split}
	&\|f\|_{\A_l^m(T)}~:=~%\Big(\int_0^T\|D^mf(t,\cdot)\|^2_{L^2_l}dt\Big)^{\frac{1}{2}}=
	\Big(\sum_{|\alpha|\leq m}\|\ly^{l+k}D^\alpha f\|_{L^2(\Omega)}^2\Big)^{\frac{1}{2}}<+\infty,\\\
	&\|f\|_{\B_l^m(T)}~:=~%\sup_{0\leq t\leq T}\|D^mf(t,\cdot)\|_{L_l^2=
	\Big(\sup_{0\leq t\leq T}\sum_{|\alpha|\leq m}\|\ly^{l+k}D^\alpha f(t,\cdot)\|_{L^2(\Omega)}^2\Big)^{\frac{1}{2}}<+\infty.
	%&\|f\|_{\C_l^m(T)}~:=~\Big(\sum_{|\alpha|\leq m}\|\ly^{l+k}D^\alpha f\|_{L^2_{tx}L_{y}^\infty(\Omega)}^2\Big)^{\frac{1}{2}}<+\infty,\quad \|f\|_{\D_l^m(T)}~:=~\Big(\sup_{0\leq t\leq T}\sum_{|\alpha|\leq m}\|\ly^{l+k}D^\alpha f(t,\cdot)\|_{L^2_{x}L_{y}^\infty(\Omega)}^2\Big)^{\frac{1}{2}}<+\infty.
\end{split}\]
Obviously, it follows that
\[\A_l^m(T)~=~\bigcap_{i=0}^mH^{i}\big(0,T;H_l^{m-i}(\Omega)\big),\quad \B_l^m(T)~=~\bigcap_{i=0}^mW^{i,\infty}\big(0,T;H_l^{m-i}(\Omega)\big). \]
\fi
The following inequalities will be used frequently in this paper.

\begin{lem}\label{lemma_ineq}
	For  proper functions $f,g,h$, the following holds.\\
	
	%\romannumeral1) 
	%For $l>\frac{1}{2}$ and If	$\lim\limits_{y\rightarrow+\infty}f(x,y)=0,$ 
	%\begin{equation}\label{trace0}\big\|f|_{y=0}\big\|_{L^2(\bT_x)}\leq  C\|\p_yf\|_{L_l^2(\Omega)}^{\frac{1}{2}}.\end{equation}	
	
	\romannumeral1) 
	If  $\lim\limits_{y\rightarrow+\infty}(fg)(x,y)=0,$ then
	\begin{equation}\label{trace}
	\Big|\int_{\bT_x}(fg)|_{y=0}dx\Big|\leq  \|\p_yf\|_{L^2(\Omega)}\|g\|_{L^2(\Omega)}+\|f\|_{L^2(\Omega)}\|\p_yg\|_{L^2(\Omega)}.
\end{equation}
In particular,   if $\lim\limits_{y\rightarrow+\infty}f(x,y)=0,$ then
	\begin{equation}\label{trace0}
	\big\|f|_{y=0}\big\|_{L^2(\bT_x)}\leq  \sqrt{2}~\|f\|_{L^2(\Omega)}^{\frac{1}{2}}\|\p_yf\|_{L^2(\Omega)}^{\frac{1}{2}}.
\end{equation}	
	%\begin{equation}\label{trace}\|f(t,x,0)\|_{L^2(\bT_x)}\leq C \|\p_yf\|_{L^2(\Omega)}^{\frac{1}{2}}\|f\|_{L^2(\Omega)}^{\frac{1}{2}}+C\|f\|_{L^2(\Omega)}.\end{equation}
	
	%\romannumeral2) For any $l>\frac{1}{2}$, 
	%\begin{align}\label{normal}
		%\big\|(\p_y^{-1}f)(y)\big\|_{L^\infty_y(\bR_+)}\leq C\|f\|_{L_l^2(\bR_+)}.
	%\end{align}

\romannumeral2) For $l\in\bR$ and an integer $m\geq3, $  any $\alpha=(\beta,k)\in\N^3, \tilde\alpha=(\tilde\beta,\tilde k)\in\N^3$ with $|\alpha|+|\tilde\alpha|\leq m$,
	\begin{align}\label{Morse}
		\big\|\big(D^\alpha f\cdot D^{\tilde\alpha}g\big)(t,\cdot)\big\|_{L^2_{l+k+\tilde k}(\Omega)}\leq C\|f(t)\|_{\h_{l_1}^m}\|g(t)\|_{\h_{l_2}^m},\qquad \forall~l_1,l_2\in\bR,\quad l_1+l_2=l.
	\end{align}
	
\romannumeral3) 
For any $\lambda>\frac{1}{2}, \tilde\lambda>0$, 
	\begin{align}\label{normal}
		 \big\|\ly^{-\lambda}(\p_y^{-1}f)(y)\big\|_{L^2_y(\bR_+)}\leq \frac{2}{2\lambda-1}\big\|\ly^{1-\lambda}f(y)\big\|_{L^2_y(\bR_+)},~ \big\|\ly^{-\tilde\lambda}(\p_y^{-1}f)(y)\big\|_{L^\infty_y(\bR_+)}\leq \frac{1}{\tilde\lambda}\big\|\ly^{1-\tilde\lambda}f(y)\big\|_{L^\infty_y(\bR_+)},
	\end{align}
and then, for $l\in\bR$, an integer $m\geq3, $  and any $\alpha=(\beta,k)\in\N^3, \tilde \beta=(\tilde \beta_1,\tilde\beta_2)\in\N^2$ with $|\alpha|+|\tilde\beta|\leq m$,
\begin{align}\label{normal0}
	\big\|\big(D^\alpha g\cdot\p_\tau^{\tilde\beta}\p_y^{-1}h\big)(t,\cdot)\big\|_{L^2_{l+k}(\Omega)}\leq C\|g(t)\|_{\h_{l+\lambda}^m}\|h(t)\|_{\h_{1-\lambda}^m}.
\end{align}
In particular, for $\lambda=1,$
\begin{align}\label{normal1}
	\big\|\ly^{-1}(\p_y^{-1}f)(y)\big\|_{L^2_y(\bR_+)}\leq 2\big\|f\big\|_{L^2_y(\bR_+)},\quad\big\|\big(D^\alpha g\cdot\p_\tau^{\tilde\beta}\p_y^{-1}h\big)(t,\cdot)\big\|_{L^2_{l+k}(\Omega)}\leq C\|g(t)\|_{\h_{l+1}^m}\|h(t)\|_{\h_{0}^m}.
\end{align}

\romannumeral4) 
For any $\lambda>\frac{1}{2}$, 
	\begin{align}\label{normal2}
		\big\|(\p_y^{-1}f)(y)\big\|_{L^\infty_{y}(\bR_+)}\leq C\|f\|_{L_{y,\lambda}^2(\bR_+)},
	\end{align}
and then, for $l\in\bR$, an integer $m\geq2, $  and any $\alpha=(\beta,k)\in\N^3, \tilde \beta=(\tilde \beta_1,\tilde\beta_2)\in\N^2$ with $|\alpha|+|\tilde\beta|\leq m$,
\begin{align}\label{normal3}
	\big\|\big(D^\alpha f\cdot\p_\tau^{\tilde\beta}\p_y^{-1}g\big)(t,\cdot)\big\|_{L^2_{l+k}(\Omega)}\leq C\|f(t)\|_{\h_l^m}\|g(t)\|_{\h_\lambda^m}.
\end{align}
	\end{lem}

To overcome the technical difficulty originated from the boundary terms at $\{y=+\infty\}$, we introduce an auxiliary function $\phi(y)\in C^\infty(\bR_+)$ satisfying that
\[\phi(y)=\begin{cases}
	y,\quad y\geq 2R_0,\\
	0,\quad 0\leq y\leq R_0
\end{cases}\]
for some constant $R_0>0$. Then, set the new unknowns:
\begin{align}\label{new_quan}
	&u(t,x,y)~:=~u_1(t,x,y)-U(t,x)\phi'(y),\quad v(t,x,y)~:=~u_2(t,x,y)+U_x(t,x)\phi(y),\nonumber\\
	&h(t,x,y)~:=~h_1(t,x,y)-H(t,x)\phi'(y),\quad g(t,x,y)~:=~h_2(t,x,y)+H_x(t,x)\phi(y).
	\end{align}
Choose the above construction for $(u,v,h,g)$ to ensure the divergence free conditions and homogenous boundary conditions, i.e.,
\[\begin{split}
&\p_x u+\p_y v=0,\quad \p_x h+\p_y g=0,\\
&(u,v,\p_yh,g)|_{y=0}=\textbf 0,\quad \lim_{y\rightarrow+\infty}(u,h)=\textbf 0,	
\end{split}\]
which implies that $v=-\p_y^{-1}\p_xu$ and $g=-\p_y^{-1}\p_xh.$
And it is easy to get that
\begin{align*}
	(u, h)(t,x,y)=\big(u_1(t,x,y)-U(t,x), h_1(t,x,y)-H(t,x)\big)+\big(U(t,x)(1-\phi'(y)),H(t,x)(1-\phi'(y))\big),
\end{align*}
which implies that by the construction of $\phi(y)$,
\begin{align}\label{est_axu}
\|(u, h)(t)\|_{\h_l^m}-CM_0\leq\|(u_1-U,h_1-H)(t)\|_{\h_l^m}\leq&\|(u,h)(t)\|_{\h_l^m}+CM_0.%\|(u, h)(t)\|_{\h_l^m}+\sum_{|\alpha|\leq m}\|D^\alpha\big(U(1-\phi'),H(1-\phi')\big)(t)\|_{L_l^2(\bT_x)}\leq\|(u,h)(t)\|_{\h_l^m}+CM_0,\nonumber\\ 
%\|(u, h)(t)\|_{\h_l^m}\leq&\|(u,h)(t)\|_{\h_l^m}+C\sum_{|\beta|\leq m}\|\p_\tau^\beta(U,H)(t)\|_{L^2(\bT_x)}\leq\|(u,h)(t)\|_{\h_l^m}+CM_0.
\end{align}
%Plugging \eqref{new_quan} into the problem \eqref{bl_mhd}, 
By using the new unknowns $(u,v,h,g)$ given by \eqref{new_quan}, we can reformulate the original problem \eqref{bl_mhd} to the following:
\begin{align}\label{bl_main}
	%\left\{\begin{array}{ll}
	\begin{cases}	
\partial_tu+\big[(u+U\phi')\partial_x+(v-U_x\phi)\partial_y\big]u-\big[(h+H\phi')\partial_x+(g-H_x\phi)\partial_y\big]h-\mu\partial^2_yu\\
\qquad\qquad+U_x\phi'u+U\phi''v-H_x\phi'h-H\phi''g=r_1,\\
		\partial_th+\big[(u+U\phi')\partial_x+(v-U_x\phi)\partial_y\big]h-\big[(h+H\phi')\partial_x+(g-H_x\phi)\partial_y\big]u-\kappa\partial_y^2h\\
		\qquad\qquad+H_x\phi'u+H\phi''v-U_x\phi'h-U\phi''g=r_2,\\
		\partial_xu+\partial_yv=0,\quad \partial_xh+\partial_yg=0,\\
		%u_1|_{t=0}=u_{10}(x,y),\quad b_1|_{t=0}=b_{10}(x,y),\\
		(u,v,\partial_yh,g)|_{y=0}=\textbf 0,\\%\quad		\lim\limits_{y\rightarrow+\infty}u_1=U(t,x),\quad \lim\limits_{y\rightarrow+\infty}b_1=B(t,x).
		(u,h)|_{t=0}=\big(u_{10}(x,y)-U(0,x)\phi'(y),h_{10}(x,y)-H(0,x)\phi'(y)\big)\triangleq(u_0,h_0)(x,y),
		\end{cases}
	%\end{array}\right.
\end{align}
where   
\begin{align}\label{def_rhs}
	\begin{cases}
		r_1=&U_t[(\phi')^2-\phi\phi''-\phi']+P_x\big[(\phi')^2-\phi\phi''-1\big]+\mu U\phi^{(3)},\\
		r_2=&H_t[(\phi')^2+\phi\phi''-\phi']+\kappa H\phi^{(3)}.
	\end{cases}
\end{align}
Note that we have used the divergence free conditions in obtaining the equations of $(u,h)$ in \eqref{bl_main}, and the relations \eqref{Brou}  in the calculation of \eqref{def_rhs}. It is worth noting that by substituting \eqref{new_quan} into the second equation of \eqref{bl_mhd} directly, there is another equivalent form for the equation of $h$, which may be convenient for use in some situations:
\begin{align}\label{eq_h}
	\p_t h+\p_y\big[(v-U_x\phi)(h+H\phi')-(u+U\phi')(g-H_x\phi)\big]-\kappa\p_y^2h=-H_t\phi'+\kappa H\phi^{(3)}.
\end{align}

By the choice of $\phi(y)$, it is easy to get that 
\begin{align}\label{property_r}
	r_1(t,x,y),~r_2(t,x,y)~\equiv~0,\qquad &y\geq2R_0,\nonumber\\
	r_1(t,x,y)~\equiv~-P_x(t,x),\quad r_2(t,x,y)~\equiv~0,\qquad &0\leq y\leq R_0,
\end{align}
and then for any $t\in[0,T],\lambda\geq0$ and $|\alpha|\leq m$, by virtue of \eqref{ass_outflow},
\begin{equation}
	\label{est_rhd}
	\|\ly^\lambda D^\alpha r_1(t)\|_{L^2(\Omega)},~\|\ly^\lambda D^\alpha r_2(t)\|_{L^2(\Omega)}\leq C\sum_{|\beta|\leq|\alpha|+1}\|\p_\tau^{\beta}(U,H,P_x)(t)\|_{L^{2}(\bT_x)}\leq CM_0.%\quad \mbox{for}\quad |\alpha|\leq m.
\end{equation}
Furthermore, similar to \eqref{est_axu} we have that for the initial data:
\begin{align}\label{est_ini}
\|(u_0,h_0)\|_{H_l^{2m}(\Omega)}-CM_0\leq&\big\|\big(u_{10}(x,y)-U(0,x),h_{10}-H(0,x)\big)\big\|_{H_l^{2m}(\Omega)}\leq \|(u_0,h_0)\|_{H_l^{2m}(\Omega)}+CM_0.%+C\|(U,H)(0,\cdot)\|_{H^{2m}(\bT_x)}\nonumber\\
%\leq~&\big\|\big(u_{10}(x,y)-U(0,x),h_{10}-H(0,x)\big)\big\|_{H_l^{2m}(\Omega)}+CM_0.
\end{align}

Finally, from the transformation \eqref{new_quan}, and the relations \eqref{est_axu} and \eqref{est_ini}, it is easy to know that Theorem \ref{Th1} is a corollary of the following result.
\begin{thm}\label{thm_main}
Let $m\geq5$ be a integer, $l\geq0$ a real number, and $(U,H,P_x)(t,x)$ satisfies  the hypotheses given in Theorem \ref{Th1}. 
In addition, assume that for the problem \eqref{bl_main}, the initial data 
 	\(\big(u_{0}(x,y),h_{0}(x,y)\big)\in H^{3m+2}_l(\Omega),\) 
 and the compatibility conditions up to $m$-th order. Moreover, there exists a sufficiently small constant $\delta_0>0$, such that 
 \begin{align}\label{ass_bound-modify}
 \big|\ly^{l+1}\p_y^i(u_{0}, h_{0})(x,y)\big|\leq(2\delta_0)^{-1},	\quad h_{0}(x,y)+H(0,x)\phi'(y)\geq2\delta_0,\quad\mbox{for}\quad i=1,2,~ (x,y)\in\Omega.
 \end{align}
Then, there exist a time $0<T_*\leq T$ and a unique solution $(u,v,h,g)$ to the initial boundary value problem (\ref{bl_main}), such that 
\begin{align}\label{result_1}
	(u, h)\in\bigcap_{i=0}^mW^{i,\infty}\Big(0,T_*;H_l^{m-i}(\Omega)\Big),%\quad(\p_yu_1,\p_yh_1)\in \bigcap_{i=0}^mH^{i}\Big(0,T;H_l^{m-i}(\Omega)\Big),
\end{align}
and
\begin{align}\label{result_2}
(v,g)&\in\bigcap_{i=0}^{m-1}W^{i,\infty}\Big(0,T_*;H_{-1}^{m-1-i}(\Omega)\Big),\quad	(\p_yv,\p_yg)\in\bigcap_{i=0}^{m-1}W^{i,\infty}\big(0,T_*;H_l^{m-1-i}(\Omega)\big).
\end{align}
Moreover, if $l>\frac{1}{2},$
\begin{align}\label{result_3}
&(v, g)\in\bigcap_{i=0}^{m-1}W^{i,\infty}\Big(0,T_*; L^\infty\big(\bR_{y,+};H^{m-1-i}(\bT_x)\big)\Big).
\end{align}
\end{thm}
Therefore, our main task is to show the above Theorem \ref{thm_main}, and its proof will be given in the following two sections.

%%%%%%%%%%%%%%%%%%%%%%--------uniform a priori estimates

\section{A priori estimates}

In this section, we will establish a priori estimates for the nonlinear problem (\ref{bl_main}).
\begin{prop}\label{prop_priori}[\textit{Weighted estimates for $D^m(u,h)$}]\\
%Let $m\geq3$ be a integer, and $l>\frac{1}{2}$ a real number. Assume that the outflow $(U,H)(t,x)$ is smooth and satisfies
%\begin{equation}\label{ass_outflow}M_0~:=~\sup_{t}\sum_{i=0}^{m+1}\|\p_t^i(U,H)(t,\cdot)\|_{H^{m+1-i}(\bT_x)}<+\infty,\qquad H(t,x)\geq2\delta,\quad t>0,x\in\bR\end{equation}
%for some positive constant $\delta>0.$Also, we suppose that for the problem \eqref{bl_main}, the initial data $(u_{0},h_{0})(x,y)$ satisfies
 %\begin{equation}\label{ass_ini}\Big(u_{0}(x,y),h_{0}(x,y))\Big)\in H^{2m}_l(\mathbb{R}^2_+),\qquad h_{0}(x,y)+H(0,x)\geq2\delta,\quad (x,y)\in\Omega,\end{equation} 
% and the compatibility conditions up to $m$-th order.  
%Under the assumptions of Theorem \ref{Th1}, 
Let $m\geq5$ be a integer, $l\geq0$ be a real number, and the hypotheses for $(U,H,P_x)(t,x)$ given in Theorem \ref{Th1}  hold. Assume that $(u,v,h,g)$ is a classical solution to the problem \eqref{bl_main} in $[0,T],$ satisfying that 
$(u,h)\in L^\infty\big(0,T; \h_l^m\big),~ (\p_yu,\p_yh)\in L^2\big(0,T; \h_l^m\big),$ 
and for sufficiently small $\delta_0$:
\begin{equation}\label{ass_h}
	h(t,x,y)+H(t,x)\phi'(y)\geq\delta_0,\quad\ly^{l+1}\p_y^i(u, h)(t,x,y)\leq \delta_0^{-1},\quad i=1,2,~(t,x,y)\in [0,T]\times\Omega.
\end{equation}
 Then, it holds that for small time, %there exists a constant $C>0$, depending only on $m, M_0$ and $\phi$, such that 
	\begin{align}\label{est_priori}
	\sup_{0\leq s\leq t}\|(u, h)(s)\|_{\h_l^m}~\leq~&\delta_0^{-4}\Big(\cp\big(M_0+\|(u_0,h_0)\|_{H_l^{2m}(\Omega)}\big)+CM_0^6t\Big)^{\frac{1}{2}}\nonumber\\
	&\cdot\Big\{1-C\delta_0^{-24}\Big(\cp\big(M_0+\|(u_0,h_0)\|_{H_l^{2m}(\Omega)}\big)+CM_0^6t\Big)^2t\Big\}^{-\frac{1}{4}}.
	%\|(u,h)(t)\|_{\h_l^m}^2+\int_0^t\|(\p_yu,\p_yh)(s)\|^2_{\h_l^m}ds\leq&
		%\Big[P\big(\|(u_0,b_0)\|_{H_l^{2m}(\Omega)}\big)+\|D^{m+1}u_s\|_{L_L^2(\bT_x)}^2\Big]
		% C\big\|(u_0,b_0)\big\|_{H_l^{2m}(\Omega)(\Omega)}^{2m+1}\cdot P\big(E_3(t)\big)\exp\{P\big(E_3(T)\big)\cdot t\},
\end{align}
%provided that the quantity within the set of braces on the right-hand side of \eqref{est_priori} is positive.
Also, we have that for $i=1,2,$
\begin{align}\label{upbound_uy}
	\|\ly^{l+1}\p_y^i(u, h)(t)\|_{L^\infty(\Omega)}
	\leq~&\|\ly^{l+1}\p_y^i(u_0, h_0)\|_{L^\infty(\Omega)}+C\delta_0^{-4}t \Big(\cp\big(M_0+\|(u_0,h_0)\|_{H_l^{2m}(\Omega)}\big)+CM_0^6t\Big)^{\frac{1}{2}}\nonumber\\
	&\cdot\Big\{1-C\delta_0^{-24}\Big(\cp\big(M_0+\|(u_0,h_0)\|_{H_l^{2m}(\Omega)}\big)+CM_0^6t\Big)^2t\Big\}^{-\frac{1}{4}},
	\end{align}
and 
\begin{align}\label{h_lowbound}
	h(t,x,y)\geq~&h_0(x,y)-C\delta_0^{-4}t \Big(\cp\big(M_0+\|(u_0,h_0)\|_{H_l^{2m}(\Omega)}\big)+CM_0^6t\Big)^{\frac{1}{2}}\nonumber\\
	&\cdot\Big\{1-C\delta_0^{-24}\Big(\cp\big(M_0+\|(u_0,h_0)\|_{H_l^{2m}(\Omega)}\big)+CM_0^6t\Big)^2t\Big\}^{-\frac{1}{4}}.
	\end{align}
\end{prop}
The proof of Proposition \ref{prop_priori}   will be given in the following two subsections. More precisely, we will obtain the weighted estimates for $D^\alpha(u,h)$ for $\alpha=(\beta,k)=(\beta_1,\beta_2,k)$, satisfying $|\alpha|=|\beta|+k\leq m,~|\beta|\leq m-1$, in the first subsection, and the weighted estimates for $\p_\tau^\beta(u,h)$ for $|\beta|=m$ in the second subsection.
% then to obtain the a priori estimates for $D^m(u,h)$ by combining these two estimates.  %To this end, we first introduce the following energy functional:
%\begin{equation}\label{def_E}E_m(t)~:=~\|(u,h)\|_{\B_l^{m}(\Omega)}%+\|(u,b)\|_{\B_l^{1,m}(\Omega)},\qquad t\in[0,T],~m\geq0,~l>\frac{1}{2},\end{equation}
%and in what follows, we denote $P(\cdot)$ by a nondecreasing, polynomial function, which may differ from line to line.

\subsection{Weighted $H^m_l-$estimates with normal derivatives}% $|\alpha|\leq m,|\beta|\leq m-1$}
\indent\newline
 The weighted estimates on $D^\alpha(u,h)$ with $|\alpha|=|\beta|+k\leq m,~|\beta|\leq m-1$ can be obtained by the standard energy method because
 one order  tangential regularity loss is allowed. That is, we have the following estimates:
\begin{prop}\label{prop_estm}[\textit{Weighted estimates for $D^\alpha(u,h)$ with $|\alpha|\leq m,|\beta|\leq m-1$}]\\
%Under the assumptions of Theorem \ref{Th1}, 
Let $m\geq5$ be a integer, $l\geq0$ be a real number, and the hypotheses for $(U,H,P_x)(t,x)$ given in Theorem \ref{Th1}  hold. Assume that $(u,v,h,g)$ is a classical solution to the problem \eqref{bl_main} in $[0,T],$ and satisfies 
$(u, h)\in L^\infty\big(0,T; \h_l^m\big),~ (\p_yu,\p_yh)\in L^2\big(0,T; \h_l^m\big).$ Then, there exists a positive constant $C$, depending on $m, l$ and $\phi$, such that for any small $0<\delta_1<1,$
	\begin{align}\label{est_prop1}
	&\sum_{\tiny\substack{|\alpha|\leq m\\|\beta|\leq m-1}}\Big(\frac{d}{dt}\|D^\alpha(u,h)(t)\|_{L_{l+k}^2(\Omega)}^2+\mu\|D^\alpha\p_yu(t)\|_{L_{l+k}^2(\Omega)}^2+\kappa\|D^\alpha\p_yh(t)\|_{L_{l+k}^2(\Omega)}^2\Big)\nonumber\\
	\leq~&\delta_1C\|(\p_yu,\p_yh)(t)\|_{\h_0^m}^2+C\delta_1^{-1}\|(u, h)(t)\|_{\h_l^m}^2\big(1+\|(u, h)(t)\|_{\h_l^m}^2\big)+\sum_{\tiny\substack{|\alpha|\leq m\\|\beta|\leq m-1}}\|D^\alpha(r_1, r_2)(t)\|_{L^2_{l+k}(\Omega)}^2\nonumber\\
	&+C\sum_{|\beta|\leq m+2}\|\p_\tau^\beta (U,H,P)(t)\|_{L^2(\bT_x)}^2.
 %+C\delta_1^{-1}\big(M_0^2+\|(u, h)(t)\|^2_{\h_l^m}\big)^2.
\end{align}
\end{prop}

\begin{proof}[\textbf{Proof.}]

Applying the operator $D^{\alpha}=\p_\tau^{\beta}\p_y^{k}$ for $\alpha=(\beta,k)=(\beta_1,\beta_2,k)$, satisfying $|\alpha|=|\beta|+k\leq m,~|\beta|\leq m-1,$ to the first two equations of $(\ref{bl_main})$, it yields that
\begin{align}\label{eq_u}
\begin{cases}
	\p_tD^\alpha u=D^\alpha r_1+\mu \p_y^2D^\alpha u-D^\alpha\Big\{\big[(u+U\phi')\partial_x+(v-U_x\phi)\partial_y\big]u-\big[(h+H\phi')\partial_x+(g-H_x\phi)\partial_y\big]h\\
	\qquad\qquad\qquad+U_x\phi'u+U\phi''v-H_x\phi'h-H\phi''g\Big\},\\
\p_tD^\alpha h=D^\alpha r_2+\kappa \p_y^2D^\alpha h-D^\alpha\Big\{\big[(u+U\phi')\partial_x+(v-U_x\phi)\partial_y\big]h-\big[(h+H\phi')\partial_x+(g-H_x\phi)\partial_y\big]u\\
\qquad\qquad\qquad+H_x\phi'u+H\phi''v-U_x\phi'h-U\phi''g\}.
\end{cases}
%\p_tD^\alpha u-\mu\p_y^2D^\alpha u=-D^\alpha\big[(u_s+u)\p_xu+v\p_y(u_s+u)-(1+b)\p_xb-g\p_yb\big].
\end{align}
Multiplying $(\ref{eq_u})_1$ by $\langle y\rangle^{2l+2k}D^\alpha u$,  $(\ref{eq_u})_2$ by $\langle y\rangle^{2l+2k}D^\alpha h$ respectively, and integrating them over $\Omega$, with respect to the spatial variables $x$ and $y$, 
we obtain that 
\begin{align}\label{est-m0}
\frac{1}{2}\frac{d}{dt}\big\|\ly^{l+k}D^\alpha(u,h)(t)\big\|_{L^2(\Omega)}^2
=&\int_{\Omega}\Big(D^\alpha r_1\cdot\ly^{2l+2k}D^\alpha u+D^\alpha r_2\cdot\ly^{2l+2k}D^\alpha h\Big)dxdy\nonumber\\
&+\mu\int_{\Omega}\big(\p_y^2D^\alpha u\cdot\ly^{2l+2k}D^\alpha u\big)dxdy+\kappa\int_{\Omega}\big(\p_y^2D^\alpha h\cdot\ly^{2l+2k}D^\alpha h\big)dxdy\nonumber\\
&-\int_{\Omega}\Big(I_1\cdot\ly^{2l+2k}D^\alpha u+I_2\cdot\ly^{2l+2k}D^\alpha h\Big)dxdy,
\end{align}
where
\begin{align}\label{def_I}
	\begin{cases}
		I_1=&D^\alpha\Big\{\big[(u+U\phi')\partial_x+(v-U_x\phi)\partial_y\big]u-\big[(h+H\phi')\partial_x+(g-H_x\phi)\partial_y\big]h\\
		&\quad+U_x\phi'u+U\phi''v-H_x\phi'h-H\phi''g\Big\},\\
		I_2=&D^\alpha\Big\{\big[(u+U\phi')\partial_x+(v-U_x\phi)\partial_y\big]h-\big[(h+H\phi')\partial_x+(g-H_x\phi)\partial_y\big]u\\
		&\quad+H_x\phi'u+H\phi''v-U_x\phi'h-U\phi''g\Big\}.
	\end{cases}
\end{align}
First of all, it is easy to get that by virtue of \eqref{est_rhd},
\begin{align}\label{est-remainder}
	&\int_{\Omega}\Big(D^\alpha r_1\cdot\ly^{2l+2k}D^\alpha u+D^\alpha r_2\cdot\ly^{2l+2k}D^\alpha h\Big)dxdy\nonumber\\
	%\leq&\frac{1}{2}\big\|\ly^{l+k}D^\alpha(u,h)(t)\big\|_{L^2(\Omega)}^2+\frac{1}{2}\big\|\ly^{l+k}D^\alpha r_1(t)\big\|_{L^2(\Omega)}^2+\frac{1}{2}\big\|\ly^{l+k}D^\alpha r_2(t)\big\|_{L^2(\Omega)}^2\nonumber\\
	\leq &\frac{1}{2}\|D^\alpha(u, h)(t)\|^2_{L_{l+k}^2(\Omega)}%\big\|\ly^{l+k}D^\alpha(u,h)(t)\big\|_{L^2(\Omega)}^2+CM_0^2.
	+\frac{1}{2}\|D^\alpha (r_1, r_2)(t)\|_{L^2_{l+k}(\Omega)}^2.
\end{align}
Next, we assume that the following two estimates holds, which will be proved later: for any small $0<\delta_1<1,$
\begin{align}
\label{est-duff}
&%\sum_{\tiny\substack{|\alpha|\leq m\\|\beta|\leq m-1}}\Big(
\mu\int_{\Omega}\big(\p_y^2D^\alpha u\cdot\langle y\rangle^{2l+2k}D^\alpha u\big) dxdy+\kappa\int_{\Omega}\big(\p_y^2D^\alpha h\cdot\langle y\rangle^{2l+2k}D^\alpha h\big) dxdy\nonumber\\
\leq&%-\sum_{\tiny\substack{|\alpha|\leq m\\|\beta|\leq m-1}}\Big(
-\frac{\mu}{2}\big\|D^\alpha \p_yu(t)\big\|^2_{L^2_{l+k}(\Omega)}-\frac{\kappa}{2}\big\|D^\alpha \p_yh(t)\big\|^2_{L_{l+k}^2(\Omega)}+\delta_1\big\|(\p_yu,\p_yh)(t)\big\|_{\h_0^m}^2
\nonumber\\
&+C\delta_1^{-1}\|(u,h)(t)\|_{\h_l^m}^2\big(1+\|(u,h)(t)\|_{\h_l^m}^2\big)+C\sum_{|\beta|\leq m-1}\|\p_\tau^\beta P_x(t)\|_{L^2(\bT_x)}^2,%+C\delta_1^{-1}\big(M_0^2+\|(u,h)(t)\|_{\h_l^m}^2\big)^2,%\|(u,h)(t)\|_{\h_0^m}^2+CM_0^4,
\end{align}
and
\begin{align}\label{est-convect}
	&%-\sum_{\tiny\substack{|\alpha|\leq m\\|\beta|\leq m-1}}
	-\int_{\Omega}\Big(I_1\cdot\ly^{2l+2k}D^\alpha u+I_2\cdot\ly^{2l+2k}D^\alpha h\Big)dxdy\nonumber\\
	\leq~&C\big(\sum_{|\beta|\leq m+2}\|\p_\tau^\beta(U,H)(t)\|_{L^2(\bT_x)}+\big\|(u,h)(t)\big\|_{\h_l^m}\big)\big\|(u,h)(t)\big\|_{\h_l^m}^2.
\end{align} 
At the moment, by plugging the above inequalities \eqref{est-remainder}-\eqref{est-convect} into \eqref{est-m0}, and summing over $\alpha$, we obtain that there exists a constant $C_m>0$, depending only on $m,$ such that
\begin{align}\label{est_both}
	&\sum_{\tiny\substack{|\alpha|\leq m\\|\beta|\leq m-1}}\Big(\frac{d}{dt}\big\|D^\alpha(u,h)(t)\big\|_{L^2_{l+k}(\Omega)}^2+\mu\big\|D^\alpha\p_yu(t)\big\|_{L^2_{l+k}(\Omega)}^2+\kappa\big\|D^\alpha\p_yh(t)\big\|_{L^2_{l+k}(\Omega)}^2\Big)\nonumber\\
	\leq~&\delta_1C_m\big\|(\p_yu,\p_yh)(t)\big\|_{\h_0^m}^2+C\delta_1^{-1}\|(u, h)(t)\|_{\h_l^m}^2\big(1+\|(u, h)(t)\|_{\h_l^m}^2\big)+\sum_{\tiny\substack{|\alpha|\leq m\\|\beta|\leq m-1}}\|D^\alpha(r_1, r_2)(t)\|_{L^2_{l+k}(\Omega)}^2\nonumber\\
	&+C\sum_{|\beta|\leq m+2}\|\p_\tau^\beta (U,H,P)(t)\|_{L^2(\bT_x)}^2,%+C\delta_1^{-1}\big(M_0^2+\|(u, h)(t)\|_{\h_l^m}^2\big)^2,
\end{align}
which implies the estimate \eqref{est_prop1} immediately.

Now, it remains to show the estimates \eqref{est-duff} and \eqref{est-convect}
that will be given as follows.
\indent\newline
\textbf{\textit{Proof of \eqref{est-duff}.}}
In this part, we will first handle the term $\mu\int_{\Omega}\big(\p_y^2D^\alpha u\cdot\langle y\rangle^{2l+2k}D^\alpha u\big) dxdy$, and the term $\kappa\int_{\Omega}\big(\p_y^2D^\alpha h\cdot\langle y\rangle^{2l+2k}D^\alpha h\big) dxdy$ can be 
estimated similarly.
%\indent\newline
%\textbf{\textit{Step 1. Estimate of $\mu\int_{\Omega}\big(\p_y^2D^\alpha u\cdot\ly^{2l+2k}D^\alpha u\big)dxdy.$}} In this step, we will estimate the term $\mu\int_{\Omega}\big(\p_y^2D^\alpha u\cdot\ly^{2l}D^\alpha u\big)dxdy.$
By integration by parts, we have
%\begin{align}\label{2.3}
%\int_0_{\mathbb{R}^2_+}\p_tD^\alpha u\langle y\rangle^{2l}D^\alpha udxdy=\frac12\frac{d}{dt}\|D^\alpha u\|^2_{L^2_l(\mathbb{R}^2_+)},
%\end{align}and
\begin{align}
\label{ex_duff}
\mu\int_{\Omega}\big(\p_y^2D^\alpha u\cdot\langle y\rangle^{2l+2k}D^\alpha u\big)dxdy
=&-\mu\big\|\ly^{l+k}\p_yD^\alpha u(t)\big\|^2_{L^2(\Omega)}+2(l+k)\mu\int_{\Omega}\big(\ly^{2l+2k-1}\p_y D^\alpha u\cdot  D^\alpha u \big)dxdy\nonumber\\
&+\mu\int_{\bT_x}(\p_y D^\alpha u\cdot D^\alpha u)\big|_{y=0} dx.
\end{align}
By Cauchy-Schwarz inequality,
\begin{align}\label{est_duff0}
&2(l+k)\mu\int_{\Omega}\big(\ly^{2l+2k-1}\p_yD^\alpha u\cdot D^\alpha u\big)dxdy\nonumber\\
\leq ~&\frac{\mu}{14}\big\|\ly^{l+k}\p_y D^\alpha u(t)\big\|^2_{L^2(\Omega)}+14\mu(l+k)^2\|\ly^{l+k}D^\alpha u(t)\|^2_{L^2(\Omega)},
\end{align}
which implies that by plugging \eqref{est_duff0} into \eqref{ex_duff},
\begin{align}\label{est_duff1}
&%\mu\sum_{\tiny\substack{|\alpha|\leq m\\|\beta|\leq m-1}}
\mu\int_{\Omega}\big(\p_y^2D^\alpha u\cdot\langle y\rangle^{2l+2k}D^\alpha u\big)dxdy\nonumber\\
\leq&%\sum_{\tiny\substack{|\alpha|\leq m\\|\beta|\leq m-1}}
-\frac{13\mu}{14}\big\|\ly^{l+k}D^\alpha\p_y u(t)\big\|^2_{L^2(\Omega)}+C\|u(t)\|_{\h_l^m}^2+%\mu\sum_{\tiny\substack{|\alpha|\leq m\\|\beta|\leq m-1}}
\mu\int_{\bT_x}(\p_y D^\alpha u\cdot D^\alpha u)\big|_{y=0} dx.
\end{align}
The last term in (\ref{est_duff1}), that is,  the boundary integral $\mu\int_{\bT_x}(\p_yD^\alpha u\cdot D^\alpha u)\big|_{y=0}dx$ is treated in the following
two cases.

{\bf Case 1: $|\alpha|\leq m-1.$} By the inequality \eqref{trace}, we obtain that for any small $0<\delta_1<1,$
\begin{align}
\label{est_duff2}
%&\mu\sum_{|\alpha|\leq m-1}
\mu\Big|\int_{\bT_x}(\p_yD^\alpha u\cdot D^\alpha u)\big|_{y=0}dx\Big|
\leq~ &%\mu\sum_{|\alpha|\leq m-1}\Big(
\mu\big\|\p_y^2D^\alpha u(t)\big\|_{L^2(\Omega)}\big\| D^\alpha u(t)\big\|_{L^2(\Omega)}+\mu\big\|\p_y D^\alpha u(t)\big\|_{L^2(\Omega)}^2\nonumber\\
\leq~&%\sum_{|\alpha|\leq m-1}\Big(
\delta_1\big\|\p_y^2D^\alpha u(t)\big\|^2_{L^{2}(\Omega)}+\frac{\mu^2}{4\delta_1}\| D^\alpha u(t)\|^2_{L^{2}(\Omega)}+\mu\big\|\p_y D^\alpha u(t)\big\|_{L^2(\Omega)}^2\nonumber\\
\leq~&%\delta_1\sum_{|\alpha|\leq m-1}\big\|\p_y^2D^\alpha u(t)\big\|^2_{L^{2}(\Omega)}
\delta_1\|\p_yu(t)\|_{\h_0^m}+C\delta_1^{-1}\|u(t)\|^2_{\h_0^{m}}.
\end{align}
%where the small constant $0<\delta_1<1$ will be chosen later.

{\bf Case 2: $|\alpha|=|\beta|+k=m$.} 
It implies that $k\geq1$ from $|\beta|\leq m-1.$ Then, denote by $\gamma\triangleq\alpha-E_3=(\beta,k-1)$ with $|\gamma|=|\beta|+k-1= m-1$, the first equation in $(\ref{bl_main})$ reads
\begin{align*}
\mu\p_yD^\alpha u=\mu \p^2_yD^\gamma u
=&D^\gamma\Big\{\p_t u+\big[(u+U\phi')\partial_x+(v-U_x\phi)\partial_y\big]u-\big[(h+H\phi')\partial_x+(g-H_x\phi)\partial_y\big]h\nonumber\\
&\qquad+U_x\phi'u+U\phi''v-H_x\phi'h-H\phi''g-r_1\Big\}.
%=&\frac{1}{\mu}D^\gamma\Big\{\p_t u+\p_x\big[(2u_s+u)u-(1+b)b\big]+\p_y\big[(u_s+u)v-bg\big]\Big\}.
\end{align*}
Then, combining \eqref{property_r} with the fact $\phi\equiv0$ for $y\leq R_0$, it yields that at $y=0,$
\begin{align}\label{bc_um}
	\mu\p_yD^\alpha u=~&D^\gamma\Big[\p_t u+\big(u\partial_x+v\partial_y\big)u-\big(h\partial_x+g\partial_y\big)h+P_x\Big]\nonumber\\
	=~&D^\gamma P_x+D^{\gamma+E_1}u+D^\gamma\big(u\p_xu-h\p_xh\Big)+D^\gamma\big(v\p_yu-g\p_yh\Big).
\end{align}
It is easy to get that by \eqref{trace0},
\begin{align}\label{est_bc1}
	\Big|\int_{\bT_x}\big(D^\gamma P_x\cdot D^\alpha u\big)\big|_{y=0}dx\Big|\leq~&\big\|D^\gamma P_x(t)\big\|_{L^2(\bT_x)}\big\|D^\alpha u(t)|_{y=0}\big\|_{L^2(\bT_x)} \nonumber\\
	\leq~&\sqrt{2}\big\|D^\gamma P_x(t)\big\|_{L^2(\bT_x)}\big\|D^\alpha u(t)\big\|_{L^2(\Omega)}^{\frac{1}{2}}\big\|D^\alpha \p_yu(t)\big\|_{L^2(\Omega)}^{\frac{1}{2}}\nonumber\\
	\leq~&\frac{\mu}{14}\|D^\alpha\p_yu(t)\|_{L^2(\Omega)}^2+C\|u(t)\|_{\h_0^m}^2+C\big\|D^\gamma P_x(t)\big\|_{L^2(\bT_x)}^2,%C\delta_1^{-1}\big(M_0^2+\|u(t)\|_{\h_0^m}^2\big),
\end{align}
provided $|\alpha|=m.$ Also, by \eqref{trace} and $|\gamma+E_1|=m,$
\begin{align}\label{est_bc2}
	\Big|\int_{\bT_x}\big(D^{\gamma+E_1}u\cdot D^\alpha u\big)\big|_{y=0}dx\Big|\leq~&\big\|D^{\gamma+E_1}\p_yu(t)\big\|_{L^2(\Omega)}\big\|D^{\alpha}u(t)\big\|_{L^2(\Omega)}+\big\|D^{\gamma+E_1}u(t)\big\|_{L^2(\Omega)}\big\|D^{\alpha}\p_yu(t)\big\|_{L^2(\Omega)}\nonumber\\
	\leq~&\frac{\delta_1}{3}\|D^{\gamma+E_1}\p_yu(t)\|_{L^2(\Omega)}^2+\frac{\mu}{14}\|D^\alpha\p_yu(t)\|_{L^2(\Omega)}^2+C\delta_1^{-1}\|u(t)\|_{\h_0^m}^2.
\end{align}
%provided that $l>\frac{1}{2}$ and $|\alpha|=|\gamma+E_1|=m.$

Hence, as we know $D^\gamma(u\p_xu\big)=\sum_{\tilde\gamma\leq\gamma}\left(\begin{array}{ccc}\gamma \\ \tilde\gamma \end{array}\right)\Big(D^{\tilde\gamma}u\cdot D^{\gamma-\tilde\gamma+E_2}u\Big),$ it follows that
\begin{align}\label{est_bc3-0}
	\Big|\int_{\bT_x}\big(D^{\gamma}(u\p_xu)\cdot D^\alpha u\big)\big|_{y=0}dx\Big|\leq~&C\sum_{\tilde\gamma\leq\gamma}\Big\{\big\|\p_y\big(D^{\tilde\gamma}u\cdot D^{\gamma-\tilde\gamma+E_2}u\big)\big\|_{L^2(\Omega)}\big\|D^\alpha u\big\|_{L^2(\Omega)}\nonumber\\
	&\qquad\quad+\big\|D^{\tilde\gamma}u\cdot D^{\gamma-\tilde\gamma+E_2}u\big\|_{L^2(\Omega)}\big\|D^\alpha \p_yu\big\|_{L^2(\Omega)}\Big\}.
\end{align}
Then, by using \eqref{Morse} and note that $|\gamma|=m-1\geq3$, we have
\begin{align*}
	\big\|\p_y\big(D^{\tilde\gamma}u\cdot D^{\gamma-\tilde\gamma+E_2}u\big)\big\|_{L^2(\Omega)}\leq~&\big\| D^{\tilde\gamma}\p_yu\cdot D^{\gamma-\tilde\gamma+E_2}u\big\|_{L^2(\Omega)}+\big\|D^{\tilde\gamma}u\cdot D^{\gamma-\tilde\gamma+E_2}\p_yu\big\|_{L^2(\Omega)}\\
	\leq~&C\|\p_yu(t)\|_{\h_0^{m-1}}\|\p_xu(t)\|_{\h_0^{m-1}}+C\|u(t)\|_{\h_0^{m-1}}\|\p_{xy}^2u(t)\|_{\h_0^{m-1}}\\
	\leq~&C\|u(t)\|_{\h_0^m}\|\p_yu(t)\|_{\h_0^m}+C\|u(t)\|_{\h_0^m}^2,
\end{align*}
and
\begin{align*}
	\big\|D^{\tilde\gamma}u\cdot D^{\gamma-\tilde\gamma+E_2}u\big\|_{L^2(\Omega)}
	\leq~&C\|u(t)\|_{\h_0^{m}}\|u(t)\|_{\h_0^{m}}
	\leq~C\|u(t)\|_{\h_0^m}^2.
\end{align*}
Substituting the above two inequalities into \eqref{est_bc3-0} gives
\begin{align}\label{est_bc3}
	&\Big|\int_{\bT_x}\big(D^{\gamma}(u\p_xu)\cdot D^\alpha u\big)\big|_{y=0}dx\Big|\nonumber\\
	\leq~&C\sum_{\tilde\gamma\leq\gamma}\Big(\big(\|u(t)\|_{\h_0^m}\|\p_yu(t)\|_{\h_0^m}+\|u(t)\|_{\h_0^m}^2\big)\big\|D^\alpha u\big\|_{L^2(\Omega)}+\|u(t)\|_{\h_0^m}^2\big\|\p_yD^\alpha u\big\|_{L^2(\Omega)}\Big)\nonumber\\
	\leq~&\frac{\delta_1}{3}\|\p_yu(t)\|_{\h_0^m}^2+\frac{\mu}{14}\|D^\alpha\p_yu(t)\|_{L^2(\Omega)}^2+C\delta_1^{-1}\|u(t)\|_{\h_0^m}^4+C\|u(t)\|_{\h_0^m}^2.
\end{align}
Similarly, we have
\begin{align}\label{est_bc4}
	&\Big|\int_{\bT_x}\big(D^{\gamma}(h\p_x h)\cdot D^\alpha u\big)\big|_{y=0}dx\Big|\nonumber\\
	\leq~&C\big(\|h(t)\|_{\h_0^m}\|\p_yh(t)\|_{\h_0^m}+C\|h(t)\|_{\h_0^m}^2\big)\big\|D^\alpha u\big\|_{L^2(\Omega)}+C\|h(t)\|_{\h_0^m}^2\big\|\p_yD^\alpha u\big\|_{L^2(\Omega)}\nonumber\\
	\leq~&\frac{\delta_1}{3}\|(\p_yu,\p_yh)(t)\|_{\h_0^m}^2+\frac{\mu}{14}\|D^\alpha\p_yu(t)\|_{L^2(\Omega)}^2+C\delta_1^{-1}\|(u,h)(t)\|_{\h_0^m}^4+C\|(u,h)(t)\|_{\h_0^m}^2.
\end{align}

We now turn to control the integral $\Big|\int_{\bT_x}\big(D^{\gamma}(v\p_yu)\cdot D^\alpha u\big)\big|_{y=0}dx\Big|$. Recall that $D^\gamma=\p_\tau^\beta\p_y^{k-1}$,   by the boundary condition $v|_{y=0}=0$ and divergence free condition $u_x+v_y=0,$ we obtain that on $\{y=0\},$
\[\begin{split}
	D^\gamma(v\p_yu)=~&\p_\tau^\beta\Big(v\p_y^ku+\sum_{i=1}^{k-1}\left(\begin{array}{ccc}
		 k-1 \\ i
	\end{array}\right)\p_y^iv\cdot\p_y^{k-i}u\Big)=~\sum_{j=0}^{k-2}\left(\begin{array}{ccc}
		 k-1 \\ j+1 \end{array}\right)\p_\tau^\beta\Big[-\p_y^j\p_xu\cdot\p_y^{k-j-1}u\Big]\\
	=~&
	-\sum_{\tiny\substack{\tilde\beta\leq\beta \\ 0\leq j\leq k-2}}\left(\begin{array}{ccc} k-1 \\ j+1	\end{array}\right)\left(\begin{array}{ccc}\beta \\ \tilde \beta \end{array}\right)\Big(\p_\tau^{\tilde\beta+e_2}\p_y^j u\cdot\p_\tau^{\beta-\tilde\beta}\p_y^{k-j-1}u\Big),
\end{split}\]
where we denote $\left(\begin{array}{ccc}j \\ i	\end{array}\right)=0$ for $i>j$. Note that the right-hand side of the above equality vanishes when $k=1$, and we only need to consider the case $k\geq2$. Thus, from the above expression for $D^\gamma(v\p_yu)$ at $y=0,$ we obtain that by \eqref{trace},
\begin{align}\label{est_bc5-0}
	\Big|\int_{\bT_x}\big(D^{\gamma}(v\p_yu)\cdot D^\alpha u\big)\big|_{y=0}dx\Big|\leq~C\sum_{\tiny\substack{\tilde\beta\leq\beta \\ 0\leq j\leq k-2}}\Big\{&\big\|\p_y\big(\p_\tau^{\tilde\beta+e_2}\p_y^ju\cdot \p_\tau^{\beta-\tilde\beta}\p_y^{k-j-1}u\big)\big\|_{L^2(\Omega)}\big\|D^\alpha u\big\|_{L^2(\Omega)}\nonumber\\
	&+\big\|\p_\tau^{\tilde\beta+e_2}\p_y^ju\cdot \p_\tau^{\beta-\tilde\beta}\p_y^{k-j-1}u\big\|_{L^2(\Omega)}\big\|D^\alpha \p_yu\big\|_{L^2(\Omega)}\Big\}.
\end{align} 
 As $0\leq j\leq k-2$, it follows that by \eqref{Morse}, 
\begin{align*}
	\big\|\p_y\big(\p_\tau^{\tilde\beta+e_2}\p_y^ju\cdot \p_\tau^{\beta-\tilde\beta}\p_y^{k-j-1}u\big)\big\|_{L^2(\Omega)}\leq~&\big\|\p_\tau^{\tilde\beta+e_2}\p_y^{j+1}u\cdot \p_\tau^{\beta-\tilde\beta}\p_y^{k-j-1}u\big\|_{L^2(\Omega)}+\big\|\p_\tau^{\tilde\beta+e_2}\p_y^{j}u\cdot \p_\tau^{\beta-\tilde\beta}\p_y^{k-j}u\big\|_{L^2(\Omega)}\\
	\leq~&C\|\p_yu(t)\|_{\h_0^{m-1}}\|\p_yu(t)\|_{\h_0^{m-1}}+C\|\p_xu(t)\|_{\h_0^{m-1}}\|\p_{y}u(t)\|_{\h_0^{m-1}}\\
	\leq~&C\|u(t)\|_{\h_0^m}^2,
\end{align*}
and
\begin{align*}
	\big\|\p_\tau^{\tilde\beta+e_2}\p_y^ju\cdot \p_\tau^{\beta-\tilde\beta}\p_y^{k-j-1}u\big\|_{L^2(\Omega)}
	\leq~&C\|u(t)\|_{\h_0^{m}}\|u(t)\|_{\h_0^{m}}
	\leq~C\|u(t)\|_{\h_0^m}^2,
\end{align*}
provided that $|\beta|+k=|\alpha|=m.$
Substituting the above two inequalities into \eqref{est_bc5-0} gives
\begin{align}\label{est_bc5}
	\Big|\int_{\bT_x}\big(D^{\gamma}(v\p_yu)\cdot D^\alpha u\big)\big|_{y=0}dx\Big|\leq&C\sum_{\tiny\substack{\tilde\beta\leq\beta \\ 0\leq j\leq k-2}}\Big\{\|u(t)\|_{\h_0^m}^2\big\|D^\alpha u\big\|_{L^2(\Omega)}+\|u(t)\|_{\h_0^m}^2\big\|\p_yD^\alpha u\big\|_{L^2(\Omega)}\Big\}\nonumber\\
	\leq~&%\frac{\delta_1}{6}\|\p_yu(t)\|_{\h_0^m}^2
	\frac{\mu}{14}\|D^\alpha\p_yu(t)\|_{L^2(\Omega)}^2+C\|u(t)\|_{\h_0^m}^4+C\|u(t)\|_{\h_0^m}^2.
\end{align}
Similarly, we can obtain
\begin{align}\label{est_bc6}
	\Big|\int_{\bT_x}\big(D^{\gamma}(g\p_yh)\cdot D^\alpha u\big)\big|_{y=0}dx\Big|\leq&C\|h(t)\|_{\h_0^m}^2\big\|D^\alpha u\big\|_{L^2(\Omega)}+C\|h(t)\|_{\h_0^m}^2\big\|\p_yD^\alpha u\big\|_{L^2(\Omega)}\nonumber\\
	\leq~&%\frac{\delta_1}{6}\|\p_yu(t)\|_{\h_0^m}^2
	\frac{\mu}{14}\|D^\alpha\p_yu(t)\|_{L^2(\Omega)}^2+C\|(u, h)(t)\|_{\h_0^m}^4+C\|(u, h)(t)\|_{\h_0^m}^2.
\end{align}
Therefore, from \eqref{bc_um} and combining the estimates \eqref{est_bc1}, \eqref{est_bc2}, \eqref{est_bc3}, \eqref{est_bc4}, \eqref{est_bc5} and \eqref{est_bc6}, we have that when $|\alpha|=|\beta|+k=m$ with $|\beta|\leq m-1$,
\begin{align}\label{est_duff3}
\Big|\int_{\bT_x}(\mu\p_yD^\alpha u\cdot D^\alpha u)\big|_{y=0}dx\Big|
\leq~&\delta_1\big\|(\p_yu,\p_yh)(t)\big\|^2_{\h_0^{m}}+\frac{3\mu}{7}\|D^\alpha\p_yu(t)\|_{L^2(\Omega)}^2+C\delta_1^{-1}\|(u, h)(t)\|_{\h_0^{m}}^4\nonumber\\
&+C\delta_1^{-1}\|(u, h)(t)\|^2_{\h_0^{m}}+C\|D^\gamma P_x(t)\|_{L^2(\bT_x)}^2.%+C\|(u,h)(t)\|^3_{\h_0^{m}}.+C\delta_1^{-1}M_0^2.
\end{align}

Combining \eqref{est_duff2} with \eqref{est_duff3}, it implies that for $|\alpha|=|\beta|+k\leq m, |\beta|\leq m-1$,
\begin{align}\label{est_duff4}
&\Big|\int_{\bT_x}(\mu\p_yD^\alpha u\cdot D^\alpha u)\big|_{y=0}dx\Big|\nonumber\\
\leq~&\delta_1\big\|(\p_yu,\p_yh)(t)\big\|^2_{\h_0^{m}}+\frac{3\mu}{7}\|D^\alpha\p_yu(t)\|_{L^2(\Omega)}^2+C\delta_1^{-1}\|(u, h)(t)\|_{\h_0^{m}}^2\big(1+\|(u, h)(t)\|_{\h_0^{m}}^2\big)%+C\|(u, h)(t)\|^2_{\h_0^{m}}
\nonumber\\
&+C\sum_{|\beta|\leq m-1}\|\p_\tau^\beta P_x(t)\|_{L^2(\bT_x)}^2.%+C\delta_1^{-1}\big(M_0+\|(u,h)(t)\|^2_{\h_0^{m}}\big)^2.%+C\|(u,h)(t)\|^3_{\h_0^{m}}.
\end{align}
Then, plugging the above estimate \eqref{est_duff4} into \eqref{est_duff1} we have
\begin{align}\label{est_duff5}
&\mu\int_{\Omega}\big(\p_y^2D^\alpha u\cdot\langle y\rangle^{2l+2k}D^\alpha u\big)dxdy\nonumber\\
\leq&-\frac{\mu}{2}\big\|D^\alpha \p_yu(t)\big\|^2_{L_{l+k}^2(\Omega)}+\delta_1\big\|(\p_yu,\p_yh)(t)\big\|^2_{\h_0^{m}}+C\delta_1^{-1}\|(u, h)(t)\|_{\h_0^{m}}^2\big(1+\|(u, h)(t)\|_{\h_0^{m}}^2\big)\nonumber\\
&+C\sum_{|\beta|\leq m-1}\|\p_\tau^\beta P_x(t)\|_{L^2(\bT_x)}^2.%+C\|D^\gamma P_x(t)\|_{L^2(\bT_x)}^2.+C\delta_1^{-1}\big(M_0^2+\|(u, h)(t)\|^2_{\h_l^{m}}\big)^2.%+C\|(u,h)(t)\|^3_{\h_0^{m}}.
\end{align}

On the other hand, one can get the similar estimation on the term $\kappa\int_{\Omega}\big(\p_y^2D^\alpha h\cdot\langle y\rangle^{2l+2k}D^\alpha h\big) dxdy$: %, we can obtain 
\begin{align}\label{est_duff6}
\kappa\int_{\Omega}\big(\p_y^2D^\alpha h\cdot\langle y\rangle^{2l+2k}D^\alpha h\big)dxdy
\leq&-\frac{\kappa}{2}\big\|D^\alpha \p_yh(t)\big\|^2_{L_{l+k}^2(\Omega)}+\delta_1\big\|(\p_yu,\p_yh)(t)\big\|^2_{\h_0^{m}}\nonumber\\
&+C\delta_1^{-1}\|(u, h)(t)\|_{\h_0^{m}}^2\big(1+\|(u, h)(t)\|_{\h_0^{m}}^2\big).%+C\delta_1^{-1}\big(M_0^2+\|(u, h)(t)\|^2_{\h_l^{m}}\big)^2,
\end{align}
Thus, we prove \eqref{est-duff} by combining \eqref{est_duff5} with \eqref{est_duff6}.

\indent\newline
\textit{\textbf{Proof of \eqref{est-convect}.}}
%Step 2. Estimate of $-\int_{\Omega}\Big(I_1\cdot\ly^{2l+2k}D^\alpha u+I_2\cdot\ly^{2l+2k}D^\alpha h\Big)dxdy.$}}
From the definition \eqref{def_I} of $I_1$ and $I_2$, we have
\begin{align*}
	I_1~=~&\big[(u+U\phi')\p_x+(v-U_x\phi)\p_y\big] D^\alpha u-\big[(h+H\phi')\p_x+(g-H_x\phi)\p_y\big]D^\alpha h\\
	&+\big[D^\alpha, (u+U\phi')\p_x+(v-U_x\phi)\p_y\big]u-\big[D^\alpha, (h+H\phi')\p_x+(g-H_x\phi)\p_y\big]h\\
	&+D^\alpha\big[U_x\phi'u+U\phi''v-H_x\phi'h-H\phi''g\big]\\
	\triangleq~&I_1^1+I_1^2+I_1^3,
\end{align*}
and
\begin{align*}
	I_2~=~&\big[(u+U\phi')\p_x+(v-U_x\phi)\p_y\big] D^\alpha h-\big[(h+H\phi')\p_x+(g-H_x\phi)\p_y\big]D^\alpha u\\
	&+\big[D^\alpha, (u+U\phi')\p_x+(v-U_x\phi)\p_y\big]h-\big[D^\alpha, (h+H\phi')\p_x+(g-H_x\phi)\p_y\big]u\\
	&+D^\alpha\big[H_x\phi'u+H\phi''v-U_x\phi'h-U\phi''g\big]\\
	\triangleq~&I_2^1+I_2^2+I_2^3.
\end{align*}
Thus, we divide the term $-\int_{\Omega}\Big(I_1\cdot\ly^{2l+2k}D^\alpha u+I_2\cdot\ly^{2l+2k}D^\alpha h\Big)dxdy$ into three parts:
\begin{align}\label{divide}
	&-\int_{\Omega}\Big(I_1\cdot\ly^{2l+2k}D^\alpha u+I_2\cdot\ly^{2l+2k}D^\alpha h\Big)dxdy\nonumber\\
	=~&-\sum_{i=1}^3\int_{\Omega}\Big(I_1^i\cdot\ly^{2l+2k}D^\alpha u+I_2^i\cdot\ly^{2l+2k}D^\alpha h\Big)dxdy\nonumber\\
	:=~&G_1+G_2+G_3,
\end{align}
and estimate each $G_i, i=1,2,3$ in the following. 
Firstly, note that
\begin{align*}%\label{phi}
\phi(y)~\equiv~y,\quad \phi'(y)~\equiv~1,\quad \phi^{(i)}(y)~\equiv~0,\qquad{\mbox for}~y\geq2R_0,~i\geq2,\end{align*}
and then, there exists some positive constant $C$ such that
\begin{align}\label{phi_y}
	\|\ly^{i-1}\phi^{(i)}(y)\|_{L^\infty(\bR_+)},~\|\ly^{\lambda}\phi^{(j)}(y)\|_{L^\infty(\bR_+)}~\leq~C,\quad\mbox{for}\quad i=0,1,~j\geq2,~\lambda\in\bR,.%\|\ly^{i-1}\phi^{(i)}(y)\|_{L^\infty(\bR_+)}~\leq~C,\qquad \forall~i\geq0.
\end{align}
\indent\newline
\textbf{\textit{Estimate for $G_1$:}}  Note that 
\[\p_x(u+U\phi')+\p_y(v-U_x\phi)=0,\quad \p_x(h+H\phi')+\p_y(g-H_x\phi)=0,\]
and the boundary conditions $(v-U_x\phi)|_{y=0}=(g-H_x\phi)|_{y=0}=0,$ we obtain that by integration by parts,
\begin{align*}
	G_1~=~&-\frac{1}{2}\int_{\Omega}\Big\{\ly^{2l+2k}\big[(u+U\phi')\p_x+(v-U_x\phi)\p_y\big]\big(|D^\alpha u|^2+|D^\alpha h|^2\big)\Big\}dxdy\nonumber\\
	&+\int_{\Omega}\Big\{\ly^{2l+2k}\big[(h+H\phi')\p_x+(g-H_x\phi)\p_y\big]\big(D^\alpha u\cdot D^\alpha h\big)\Big\}dxdy\nonumber\\
	=~&(l+k)\int_{\Omega}\Big\{\ly^{2l+2k-1}(v-U_x\phi)\cdot\big(|D^\alpha u|^2+|D^\alpha h|^2\big)\Big\}dxdy\nonumber\\
	&-2(l+k)\int_{\Omega}\Big\{\ly^{2l+2k-1}(g-H_x\phi)\cdot\big(D^\alpha u\cdot D^\alpha h\big)\Big\}dxdy.
\end{align*}
Then, by using that $v=-\p_y^{-1}\p_xu,g=-\p_y^{-1}\p_xh$ and \eqref{phi_y} for $i=0,$ %the fact $\phi(y)\leq Cy$ for some positive constant $C$, 
we get that by virtue of \eqref{normal} and Sobolev embedding inequality,
\begin{align}\label{est_G1}
	G_1~\leq~ &(l+k)\Big(\Big\|\frac{v-U_x\phi}{1+y}\Big\|_{L^\infty(\Omega)}+\Big\|\frac{g-H_x\phi}{1+y}\Big\|_{L^\infty(\Omega)}\Big)\cdot\big\|\ly^{l+k}D^\alpha(u,h)(t)\big\|_{L^2(\Omega)}^2\nonumber\\
	\leq~&C\big(\|u_x(t)\|_{L^\infty(\Omega)}+\|h_x(t)\|_{L^\infty(\Omega)}+\|(U_x,H_x)(t)\|_{L^\infty(\bT_x)} %\|H_x\|_{L^\infty(D_t)}
	\big)\cdot\big\|\ly^{l+k}D^\alpha(u,h)(t)\big\|_{L^2(\Omega)}^2\nonumber\\
	\leq~&C\big(\|(u, h)(t)\|_{\h_0^3}+\|(U_x,H_x)(t)\|_{L^\infty(\bT_x)}\big)\|(u, h)(t)\|_{\h_l^m}^2. %\cdot\big\|\ly^{l+k}D^\alpha(u,h)\big\|_{L^2(\Omega)}^2~\leq~C\big(M_0+\|(u,h)\|_{\A_l^m(t)}\big)^4,
\end{align}
%provided that $m\geq3.$

\indent\newline
\textbf{\textit{Estimate for $G_2$:}}
For $G_2$, note that
\begin{align}\label{G2}
	G_2~\leq~\|I_1^2(t)\|_{L^2_{l+k}(\Omega)}\|D^\alpha u(t)\|_{L^2_{l+k}(\Omega)}+\|I_2^2(t)\|_{L^2_{l+k}(\Omega)}\|D^\alpha h(t)\|_{L^2_{l+k}(\Omega)}.
\end{align}
Thus, we need to obtain $\|I_1^2(t)\|_{L^2_{l+k}(\Omega)}$ and $\|I_2^2(t)\|_{L^2_{l+k}(\Omega)}$. To this end, we are going to estimate only the $L^2_{l+k}$ of
$I_1^2$, %$\|I_1^2\|_{L^2_{l+k}(\Omega)}$, 
because the $L^2_{l+k}-$estimate on $I_2^2$ %of $\|I_2^2\|_{L^2_{l+k}(\Omega)}$ 
can be obtained similarly. 

Rewrite the quantity $I_1^2$ as: 
\begin{align}\label{I12}
	I_1^2~=~&\big[D^\alpha, u\p_x+v\p_y\big]u-\big[D^\alpha, h\p_x+g\p_y\big]h\nonumber\\
	&+\big[D^\alpha, U\phi'\p_x-U_x\phi\p_y\big]u-\big[D^\alpha, H\phi'\p_x-H_x\phi\p_y\big]h\nonumber\\
:=~&I_{1,1}^2+I_{1,2}^2.
\end{align}
In the following, we will estimate $\|I_{1,1}^2\|_{L^2_{l+k}(\Omega)}$ and $\|I_{1,2}^2\|_{L_{l+k}^2(\Omega)}$ respectively.
\indent\newline
\underline{\textit{$L^2_{l+k}-$estimate on $I_{1,1}^2$:}}
The quantity $I_{1,1}^2$ can be expressed as:
\begin{align}\label{I112}
	I_{1,1}^2~=~&\sum_{0<\tilde\alpha\leq\alpha}\left(\begin{array}{ccc}
		 \alpha \\ \tilde\alpha 
	\end{array}\right)\Big\{\Big(D^{\tilde\alpha}u~\p_x+D^{\tilde\alpha}v~\p_y\Big)(D^{\alpha-\tilde\alpha}u)-\Big(D^{\tilde\alpha}h~\p_x+D^{\tilde\alpha}g~\p_y\Big)(D^{\alpha-\tilde\alpha}h)\Big\}.
\end{align}
Let $\tilde\alpha\triangleq(\tilde\beta,\tilde k)$, then we will study
the terms in \eqref{I112} through the following two cases
corresponding to $\tilde k=0$ and $\tilde k\geq1$ respectively.

\indent\newline
\textit{Case 1: $\tilde k=0.$} Firstly, $D^{\tilde\alpha}=\p_\tau^{\tilde\beta}$ and $\tilde\beta\geq e_i, i=1$ or 2 since $|\tilde\alpha|>0$. Then, we obtain that by \eqref{Morse},
\begin{align*}
\big\|D^{\tilde\alpha}u\cdot\p_x D^{\alpha-\tilde\alpha}u\big\|_{L^2_{l+k}(\Omega)}=&\big\|\p_\tau^{\tilde\beta-e_i}(\p_\tau^{e_i}u)\cdot D^{\alpha-\tilde\alpha}(\p_xu)	\big\|_{L^2_{l+k}(\Omega)}\\
\leq~&C\|\p_\tau^{e_i}u(t)\|_{\h_0^{m-1}}\|\p_xu(t)\|_{\h_l^{m-1}}~\leq~C\|u(t)\|_{\h_l^{m}}^2,
\end{align*}
provided that $m-1\geq3$. Similarly, it also holds
\begin{align*}
\big\|D^{\tilde\alpha}h\cdot\p_x D^{\alpha-\tilde\alpha}h\big\|_{L^2_{l+k}(\Omega)}\leq~C\|h(t)\|_{\h_l^{m}}^2.
\end{align*}

On the other hand, by using $v=-\p_y^{-1}\p_xu,$ we have
\begin{align*}
D^{\tilde\alpha}v\cdot\p_y D^{\alpha-\tilde\alpha}u~=~&-\p_\tau^{\tilde\beta}\p_y^{-1}(\p_xu)\cdot \p_\tau^{\beta-\tilde\beta}\p_y^{k+1}u.
\end{align*}
Then, when $|\alpha|=|\beta|+k\leq m-1$, applying \eqref{normal1} to the right-hand side of the above equality yields
\begin{align*}
\big\|D^{\tilde\alpha}v\cdot\p_y D^{\alpha-\tilde\alpha}u\big\|_{L^2_{l+k}(\Omega)}=~&\big\|\p_\tau^{\tilde\beta}\p_y^{-1}(\p_xu)\cdot \p_\tau^{\beta-\tilde\beta}\p_y^k(\p_yu)	\big\|_{L^2_{l+k}(\Omega)}\\
\leq~&C\|\p_xu(t)\|_{\h_0^{m-1}}\|\p_yu(t)\|_{\h_{l+1}^{m-1}}~\leq~C\|u(t)\|_{\h_l^{m}}^2,
\end{align*}
provided that $m-1\geq3$. When $|\alpha|=|\beta|+k=m$,  it implies that $k\geq1$ since $|\beta|\leq m-1,$ and consequently, we get that by \eqref{normal1},
\begin{align*}
\big\|D^{\tilde\alpha}v\cdot\p_y D^{\alpha-\tilde\alpha}u\big\|_{L^2_{l+k}(\Omega)}=~&\big\|\p_\tau^{\tilde\beta-e_i}\p_y^{-1}(\p_\tau^{e_i+e_2}u)\cdot \p_\tau^{\beta-\tilde\beta}\p_y^{k-1}(\p_y^2u)	\big\|_{L^2_{l+1+(k-1)}(\Omega)}\\
\leq~&C\|\p_\tau^{e_i+e_2}u(t)\|_{\h_0^{m-2}}\|\p_y^2u(t)\|_{\h_{l+2}^{m-2}}~\leq~C\|u(t)\|_{\h_l^{m}}^2,
\end{align*}
provided that $m-2\geq3$. Therefore, it holds that for $|\alpha|=|\beta|+k\leq m, |\beta|\leq m-1$,
\begin{align*}
\big\|D^{\tilde\alpha}v\cdot\p_y D^{\alpha-\tilde\alpha}u\big\|_{L^2_{l+k}(\Omega)}\leq~C\|u(t)\|_{\h_l^{m}}^2.
\end{align*}
Similarly,  one can obtain 
\begin{align*}
\big\|D^{\tilde\alpha}g\cdot\p_y D^{\alpha-\tilde\alpha}h\big\|_{L^2_{l+k}(\Omega)}\leq~C\|h(t)\|_{\h_l^{m}}^2.
\end{align*}
Thus, we conclude that for $\tilde k=0$ with $\tilde\alpha=(\tilde\beta, \tilde k)$,
\begin{align}\label{est_I112-1}
	\big\|\big(D^{\tilde\alpha}u~\p_x+D^{\tilde\alpha}v~\p_y\big)(D^{\alpha-\tilde\alpha}u)-\big(D^{\tilde\alpha}h~\p_x+D^{\tilde\alpha}g~\p_y\big)(D^{\alpha-\tilde\alpha}h)\big\|_{L^2_{l+k}(\Omega)}\leq~&C\|(u,h)(t)\|_{\h_l^{m}}^2.
\end{align}

\indent\newline
\textit{Case 2: $\tilde k\geq1.$} It follows that $\tilde\alpha\geq E_3,$ and then, the  right-hand side of \eqref{I112} becomes:
\begin{align*}
	&\big(D^{\tilde\alpha}u~\p_x+D^{\tilde\alpha}v~\p_y\big)(D^{\alpha-\tilde\alpha}u)-\big(D^{\tilde\alpha}h~\p_x+D^{\tilde\alpha}g~\p_y\big)(D^{\alpha-\tilde\alpha}h)\\
	=~&\big(D^{\tilde\alpha}u~\p_x-D^{\tilde\alpha-E_3}(\p_xu)~\p_y\big)(D^{\alpha-\tilde\alpha}u)-\big(D^{\tilde\alpha}h~\p_x-D^{\tilde\alpha-E_3}(\p_xh)~\p_y\big)(D^{\alpha-\tilde\alpha}h).
\end{align*}
By applying \eqref{Morse} to the terms on the right-hand side of the above quality, we get
\begin{align*}
\big\|D^{\tilde\alpha}u\cdot\p_x D^{\alpha-\tilde\alpha}u\big\|_{L^2_{l+k}(\Omega)}=~&\big\|D^{\tilde\alpha-E_3}(\p_yu)\cdot D^{\alpha-\tilde\alpha}(\p_xu)	\big\|_{L^2_{l+1+(k-1)}(\Omega)}\\
\leq~&C\|\p_yu(t)\|_{\h_{l+1}^{m-1}}\|\p_xu(t)\|_{\h_0^{m-1}}~\leq~C\|u(t)\|_{\h_l^{m}}^2,\\
\big\|D^{\tilde\alpha-E_3}(\p_xu)\cdot\p_y D^{\alpha-\tilde\alpha}u\big\|_{L^2_{l+k}(\Omega)}=~&\big\|D^{\tilde\alpha-E_3}(\p_xu)\cdot D^{\alpha-\tilde\alpha}(\p_yu)	\big\|_{L^2_{l+1+(k-1)}(\Omega)}\\
\leq~&C\|\p_xu(t)\|_{\h_{0}^{m-1}}\|\p_yu(t)\|_{\h_{l+1}^{m-1}}~\leq~C\|u(t)\|_{\h_l^{m}}^2,\\
\big\|D^{\tilde\alpha}h\cdot\p_x D^{\alpha-\tilde\alpha}h\big\|_{L^2_{l+k}(\Omega)}=~&\big\|D^{\tilde\alpha-E_3}(\p_yh)\cdot D^{\alpha-\tilde\alpha}(\p_xh)	\big\|_{L^2_{l+1+(k-1)}(\Omega)}\\
\leq~&C\|\p_yh(t)\|_{\h_{l+1}^{m-1}}\|\p_xh(t)\|_{\h_0^{m-1}}~\leq~C\|h(t)\|_{\h_l^{m}}^2,\\
\big\|D^{\tilde\alpha-E_3}(\p_xh)\cdot\p_y D^{\alpha-\tilde\alpha}h\big\|_{L^2_{l+k}(\Omega)}=~&\big\|D^{\tilde\alpha-E_3}(\p_xh)\cdot D^{\alpha-\tilde\alpha}(\p_yh)	\big\|_{L^2_{l+1+(k-1)}(\Omega)}\\
\leq~&C\|\p_xh(t)\|_{\h_{0}^{m-1}}\|\p_yh(t)\|_{\h_{l+1}^{m-1}}~\leq~C\|h(t)\|_{\h_l^{m}}^2.
\end{align*}
Consequently, we actually conclude that for $\tilde k\geq1$ with $\tilde\alpha=(\tilde\beta, \tilde k)$,
\begin{align}\label{est_I112-2}
	\big\|\big(D^{\tilde\alpha}u~\p_x+D^{\tilde\alpha}v~\p_y\big)(D^{\alpha-\tilde\alpha}u)-\big(D^{\tilde\alpha}h~\p_x+D^{\tilde\alpha}g~\p_y\big)(D^{\alpha-\tilde\alpha}h)\big\|_{L^2_{l+k}(\Omega)}\leq~&C\|(u,h)(t)\|_{\h_l^{m}}^2.
\end{align}

Finally, based on the results obtained in the above two cases, it holds that by using \eqref{est_I112-1} and \eqref{est_I112-2} in \eqref{I112},
\begin{align}\label{est_I112}
	\|I_{1,1}^2(t)\|_{L^2_{l+k}(\Omega)}\leq ~C\|(u,h)(t)\|_{\h_l^m}^2.
\end{align}

\indent\newline
\underline{\textit{$L^2_{l+k}-$estimate on $I_{1,2}^2$:}} 
 Write
\begin{align*}%\label{I122}
	I_{1,2}^2~=~\sum_{0<\tilde\alpha\leq\alpha}\left(\begin{array}{ccc}
		 \alpha \\ \tilde\alpha 
	\end{array}\right)\Big\{&\Big(D^{\tilde\alpha}(U\phi')~\p_x-D^{\tilde\alpha}(U_x\phi)~\p_y\Big)(D^{\alpha-\tilde\alpha}u)-\Big(D^{\tilde\alpha}(H\phi')~\p_x-D^{\tilde\alpha}(H_x\phi)~\p_y\Big)(D^{\alpha-\tilde\alpha}h)\Big\}.
\end{align*}
Let $\tilde\alpha\triangleq(\tilde\beta,\tilde k)$ and note that $|\alpha-\tilde\alpha|\leq |\alpha|-1\leq m-1.$ By using \eqref{phi_y},
we estimate each term on the right hand side of the above equility as follows:
\begin{align*}
	\big\|D^{\tilde\alpha}(U\phi')\cdot\p_x D^{\alpha-\tilde\alpha}u\big\|_{L^2_{l+k}(\Omega)}\leq~&\big\|\ly^{\tilde k}D^{\tilde\alpha}(U\phi')(t)\big\|_{L^\infty(\Omega)}\big\|\ly^{l+k-\tilde k}\p_x D^{\alpha-\tilde\alpha}u(t)\big\|_{L^2(\Omega)}\\\leq~&C\big\|\p_\tau^{\tilde\beta}U(t)\big\|_{L^\infty(\bT_x)}\|u(t)\|_{\h_l^m},\\
	%\end{align*}and similarly,\begin{align*}
	\big\|D^{\tilde\alpha}(U_x\phi)\cdot\p_y D^{\alpha-\tilde\alpha}u\big\|_{L^2_{l+k}(\Omega)}\leq~&\big\|\ly^{\tilde k-1}D^{\tilde\alpha}(U_x\phi)(t)\big\|_{L^\infty(\Omega)}\big\|\ly^{l+k-\tilde k+1}\p_y D^{\alpha-\tilde\alpha}u(t)\big\|_{L^2(\Omega)}\\
	\leq~&C\big\|\p_\tau^{\tilde\beta}U_x(t)\big\|_{L^\infty(\bT_x)}\|u(t)\|_{\h_l^m},
	\end{align*}
	and similarly,
	\begin{align*}
	\big\|D^{\tilde\alpha}(H\phi')\cdot\p_x D^{\alpha-\tilde\alpha}h\big\|_{L^2_{l+k}(\Omega)}%\leq~&\big\|\ly^{\tilde k}D^{\tilde\alpha}(H\phi')(t)\big\|_{L^\infty(\Omega)}\big\|\ly^{l+k-\tilde k}\p_x D^{\alpha-\tilde\alpha}h(t)\big\|_{L^2(\Omega)}
	\leq~C\big\|\p_\tau^{\tilde\beta}H(t)\big\|_{L^\infty(\bT_x)}\|h(t)\|_{\h_l^m},\\
	\big\|D^{\tilde\alpha}(H_x\phi)\cdot\p_y D^{\alpha-\tilde\alpha}h\big\|_{L^2_{l+k}(\Omega)}%\leq~&\big\|\ly^{\tilde k-1}D^{\tilde\alpha}(H_x\phi)(t)\big\|_{L^\infty(\Omega)}\big\|\ly^{l+k-\tilde k+1}\p_y D^{\alpha-\tilde\alpha}h(t)\big\|_{L^2(\Omega)}
	\leq~C\big\|\p_\tau^{\tilde\beta}H_x(t)\big\|_{L^\infty(\bT_x)}\|h(t)\|_{\h_l^m}.
\end{align*}
Therefore, it follows 
\begin{align}\label{est_I122}
\|I_{1,2}^2(t)\|_{L^2_{l+k}(\Omega)}~\leq~C\|(u,h)(t)\|_{\h_l^m}\cdot\big(\sum_{|\beta|\leq m+1}\|\p_\tau^\beta(U,H)(t)\|_{L^\infty(\bT_x)}\big).
\end{align}

Now, we can obtain the estimate of $\|I_1^2\|_{L_{l+k}^2(\Omega)}$. Indeed, plugging \eqref{est_I112} and \eqref{est_I122} into \eqref{I12} yields
\begin{align}\label{est_I12}
\|I_{1}^2(t)\|_{L^2_{l+k}(\Omega)}~\leq~C\big(\sum_{|\beta|\leq m+1}\|\p_\tau^\beta(U,H)(t)\|_{L^\infty(\bT_x)}+\|(u, h)(t)\|_{\h_l^m}\big)~\|(u,h)(t)\|_{\h_l^m}.
\end{align}
Similarly, one can also get
\begin{align}\label{est_I22}
\|I_{2}^2(t)\|_{L^2_{l+k}(\Omega)}~\leq~C\big(\sum_{|\beta|\leq m+1}\|\p_\tau^\beta(U,H)(t)\|_{L^\infty(\bT_x)}+\|(u, h)(t)\|_{\h_l^m}\big)~\|(u,h)(t)\|_{\h_l^m},
\end{align}
then, substituting \eqref{est_I12} and \eqref{est_I22} into \eqref{G2} gives
\begin{align}\label{est_G2}
	G_2~\leq~&C\Big(\sum_{|\beta|\leq m+1}\|\p_\tau^\beta(U,H)(x)\|_{L^\infty(\bT_x)}+\|(u, h)(t)\|_{\h_l^m}\Big)~\|(u, h)(t)\|_{\h_l^m}\|D^\alpha (u, h)(t)\|_{L^2_{l+k}(\Omega)}\nonumber\\
	\leq~&C\Big(\sum_{|\beta|\leq m+2}\|\p_\tau^\beta(U,H)(t)\|_{L^2(\bT_x)}+\|(u, h)(t)\|_{\h_l^m}\Big)~\|(u, h)(t)\|_{\h_l^m}^2.
\end{align}

\indent\newline
\textbf{\textit{Estimate on $G_3$:}}
For $G_3,$ the Cauchy-Schwarz inequality implies
\begin{align}\label{G3}
	G_3~\leq~\|I_1^3(t)\|_{L^2_{l+k}(\Omega)}\|D^\alpha u(t)\|_{L^2_{l+k}(\Omega)}+\|I_2^3(t)\|_{L^2_{l+k}(\Omega)}\|D^\alpha h(t)\|_{L^2_{l+k}(\Omega)}.
\end{align}
Then, it remains to estimate $\|I_1^3(t)\|_{L^2_{l+k}(\Omega)}$ and $\|I_2^3(t)\|_{L^2_{l+k}(\Omega)}$. In the following, we are going to establish the weighted estimate on $I_1^3$, %$\|I_1^3(t)\|_{L^2_{l+k}(\Omega)}$, 
for example, and  the weighed estimate on $I_2^3$ %$\|I_2^3\|_{L^2_{l+k}(\Omega)}$  
can be obtained in a similar way. 

Recall that $D^\alpha=\p_\tau^\beta\p_y^k, $ we have
\begin{align}\label{I13}
I_1^3~=~\sum_{\tilde\alpha\leq\alpha}\left(\begin{array}{ccc}\alpha \\ \tilde\alpha\end{array}\right)\Big[D^{\tilde\alpha}u\cdot D^{\alpha-\tilde\alpha}(U_x\phi')+D^{\tilde\alpha}v\cdot D^{\alpha-\tilde\alpha}(U\phi'')-D^{\tilde\alpha}h\cdot D^{\alpha-\tilde\alpha}(H_x\phi')-D^{\tilde\alpha}g\cdot D^{\alpha-\tilde\alpha}(H\phi'')\Big].	
\end{align}
Then, let $\tilde\alpha\triangleq(\tilde\beta,\tilde k)$, and we estimate each term in  \eqref{I13} as follows. Firstly, by using \eqref{phi_y} we have
\begin{align*}
\big\|D^{\tilde\alpha}u\cdot D^{\alpha-\tilde\alpha}(U_x\phi')\big\|_{L^2_{l+k}(\Omega)}\leq~&\big\|\ly^{l+\tilde k}D^{\tilde\alpha}u(t)\big\|_{L^2(\Omega)}\big\|\ly^{k-\tilde k}D^{\alpha-\tilde\alpha}(U_x\phi')(t)\big\|_{L^\infty(\Omega)}\\\leq~&C\|u(t)\|_{\h_l^m}\big\|\p_\tau^{\beta-\tilde\beta}U_x(t)\big\|_{L^\infty(\bT_x)},
\end{align*}
and similarly,
\begin{align*}
\big\|D^{\tilde\alpha}h\cdot D^{\alpha-\tilde\alpha}(H_x\phi')\big\|_{L^2_{l+k}(\Omega)}
%\leq~&\big\|\ly^{l+\tilde k}D^{\tilde\alpha}h(t)\big\|_{L^2(\Omega)}\big\|\ly^{k-\tilde k}D^{\alpha-\tilde\alpha}(H_x\phi')(t)\big\|_{L^\infty(\Omega)}\\
\leq~&C\|h(t)\|_{\h_l^m}\big\|\p_\tau^{\beta-\tilde\beta}H_x(t)\big\|_{L^\infty(\bT_x)}.
\end{align*}

Secondly, as $v=-\p_y^{-1}\p_xu$, it reads
\begin{align*}
	D^{\tilde\alpha}v\cdot D^{\alpha-\tilde\alpha}(U\phi'')=-D^{\tilde\alpha+E_2}\p_y^{-1}u\cdot D^{\alpha-\tilde\alpha}(U\phi'').
\end{align*}
Therefore, if $\tilde k\geq1$, it follows that by \eqref{phi_y},
\begin{align*}
\big\|D^{\tilde\alpha}v\cdot D^{\alpha-\tilde\alpha}(U\phi'')\big\|_{L^2_{l+k}(\Omega)}=~&\big\|\p_\tau^{\tilde\beta+e_2}\p_y^{\tilde k-1}u\cdot \p_\tau^{\beta-\tilde\beta}\p_y^{k-\tilde k}(U\phi'')\big\|_{L^2_{l+k}(\Omega)}\\
\leq~&\big\|\ly^{\tilde k-1}\p_\tau^{\tilde\beta+e_2}\p_y^{\tilde k-1}u(t)\big\|_{L^2(\Omega)}\big\|\ly^{l+k-\tilde k+1}%D^{\alpha-\tilde\alpha}
\p_\tau^{\beta-\tilde\beta}\p_y^{k-\tilde k}(U\phi'')(t)\big\|_{L^\infty(\Omega)}\\
\leq~&C\|u(t)\|_{\h_0^m}\big\|\p_\tau^{\beta-\tilde\beta}U(t)\big\|_{L^\infty(\bT_x)};
\end{align*}
if $\tilde k=0$, we obtain that by \eqref{normal1} and \eqref{phi_y}, %and $l>\frac{1}{2}$,
\begin{align*}
\big\|D^{\tilde\alpha}v\cdot D^{\alpha-\tilde\alpha}(U\phi'')\big\|_{L^2_{l+k}(\Omega)}=~&\Big\|\frac{\p_\tau^{\tilde\beta+e_2}\p_y^{-1}u}{1+y}\cdot \p_\tau^{\beta-\tilde\beta}\p_y^k(U\phi'')\Big\|_{L^2_{l+k+1}(\Omega)}\\
\leq~&\Big\|\frac{\p_\tau^{\tilde\beta+e_2}\p_y^{-1}u(t)}{1+y}\Big\|_{L^2(\Omega)}\big\|\ly^{l+k+1}%D^{\alpha-\tilde\alpha}
\p_\tau^{\beta-\tilde\beta}\p_y^k(U\phi'')(t)\big\|_{L^\infty (\Omega)}\\
\leq~&C\|\p_\tau^{\tilde\beta+e_2}u(t)\|_{L^2(\Omega)}\big\|\p_\tau^{\beta-\tilde\beta}U(t)\big\|_{L^\infty(\bT_x)}\\
\leq~ &C\|u(t)\|_{\h_0^m}\big\|\p_\tau^{\beta-\tilde\beta}U(t)\big\|_{L^\infty(\bT_x)},
\end{align*} 
provided that $|\tilde\beta|\leq|\beta|\leq m-1.$
Combining the above two inequalities yields that
\begin{align*}
	\big\|D^{\tilde\alpha}v\cdot D^{\alpha-\tilde\alpha}(U\phi'')\big\|_{L^2_{l+k}(\Omega)}
\leq~&C\|u(t)\|_{\h_0^m}\cdot\big(\sum_{|\beta|\leq m+1}\|\p_\tau^\beta(U,H)(t)\|_{L^\infty(\bT_x)}\big).
\end{align*}
Similarly, we have
\begin{align*}
	\big\|D^{\tilde\alpha}g\cdot D^{\alpha-\tilde\alpha}(H\phi'')\big\|_{L^2_{l+k}(\Omega)}
\leq~&C\|h(t)\|_{\h_0^m}\cdot\big(\sum_{|\beta|\leq m+1}\|\p_\tau^\beta(U,H)(t)\|_{L^\infty(\bT_x)}\big).
\end{align*}

We take into account the above arguments, to conclude that
\begin{align}\label{est_I13}
	\|I_1^3(t)\|_{L^2_{l+k}(\Omega)}~\leq~C\|(u, h)(t)\|_{\h_l^m}\cdot\big(\sum_{|\beta|\leq m+1}\|\p_\tau^\beta(U,H)(t)\|_{L^\infty(\bT_x)}\big).
\end{align}
Then, one can obtain a similar estimate of $I_2^3$: %as for $I_1^3$ that
\begin{align}\label{est_I23}
	\|I_2^3(t)\|_{L^2_{l+k}(\Omega)}~\leq~C\|(u, h)(t)\|_{\h_l^m}\cdot\big(\sum_{|\beta|\leq m+1}\|\p_\tau^\beta(U,H)(t)\|_{L^\infty(\bT_x)}\big),
\end{align}
which implies that by plugging \eqref{est_I13} and \eqref{est_I23} into \eqref{G3},
\begin{align}\label{est_G3}
	G_3~\leq~&C\|D^\alpha (u, h)(t)\|_{L^2_{l+k}(\Omega)} \|(u, h)(t)\|_{\h_l^m}\cdot\big(\sum_{|\beta|\leq m+1}\|\p_\tau^\beta(U,H)(x)\|_{L^\infty(\bT_x)}\big)\nonumber\\
	\leq~&C\|(u, h)(t)\|_{\h_l^m}^2\cdot\big(\sum_{|\beta|\leq m+2}\|\p_\tau^\beta(U,H)(t)\|_{L^2(\bT_x)}\big).
\end{align}
 
 Now, as we have completed the estimates on $G_i, i=1,2,3$ given by \eqref{est_G1}, \eqref{est_G2} and \eqref{est_G3} respectively, from \eqref{divide} the conclusion of this step follows immediately:
\begin{align*}%\label{est_step2}
	&-\int_{\Omega}\Big(I_1\cdot\ly^{2l+2k}D^\alpha u+I_2\cdot\ly^{2l+2k}D^\alpha h\Big)dxdy\nonumber\\
	\leq~&C\big(\sum_{|\beta|\leq m+2}\|\p_\tau^\beta(U,H)(t)\|_{L^2(\bT_x)}+\|(u, h)(t)\|_{\h_l^m}\big)~\|(u, h)(t)\|_{\h_l^m}^2,
\end{align*}
and we complete the proof of \eqref{est-convect}. 
\end{proof}

%%%%%%%%%%%%%%%%%%%--------high order tangential derivatives

\subsection{Weighted $H^m_l-$estimates only in tangential variables}
\indent\newline
%\renewcommand{\theequation}{\thesubsection. \arabic{equation}}
%\setcounter{equation}{0}
%To estimate the  $m$-th order tangential derivatives, we will introduce a new quantity, which is equivalent to $m$-th order tangential derivatives. \\
Similar to  the classical Prandtl equations, an essential difficulty for solving the problem \eqref{bl_main} is the loss of one derivative in the 
tangential variable $x$ in the terms $v\p_yu-g\p_yh$ and $v\p_yh-g\p_yu$.
In other words, $v=-\p_y^{-1}\p_xu$ and $g=-\p_y^{-1}\p_xh$, by the divergence free conditions, create a loss of $x-$derivative that prevents us to apply the standard energy estimates.  Precisely,  consider the following equations of $\p_\tau^\beta(u,h)$ with $|\beta|=m$, by taking the $m-$th order tangential derivatives on the first two equations of \eqref{bl_main}
\begin{equation}
	\label{eq_xm}\begin{cases}
		\p_t\p_\tau^\beta u+\big[(u+U\phi')\partial_x+(v-U_x\phi)\partial_y\big]\p_\tau^\beta u-\big[(h+H\phi')\partial_x+(g-H_x\phi)\partial_y\big]\p_\tau^\beta h-\mu\p_y^2\p_\tau^\beta u\\
\qquad\qquad+(\p_yu+U\phi'')\p_\tau^\beta v
-(\p_yh+H\phi'')\p_\tau^\beta g=\p_\tau^\beta r_1+R_u^\beta,\\
	\p_t\p_\tau^\beta h+\big[(u+U\phi')\partial_x+(v-U_x\phi)\partial_y\big]\p_\tau^\beta h-\big[(h+H\phi')\partial_x+(g-H_x\phi)\partial_y\big]\p_\tau^\beta u-\kappa\p_y^2\p_\tau^\beta h\\
\qquad\qquad+(\p_yh+H\phi'')\p_\tau^\beta v
-(\p_yu+U\phi'')\p_\tau^\beta g=\p_\tau^\beta r_2+R_h^\beta,
	\end{cases}
\end{equation}
where 
\begin{equation}\label{def_R}
	\begin{cases}
		R_u^\beta~=&\p_\tau^\beta\big(-U_x\phi'u+H_x\phi'h\big)-[\p_\tau^\beta, U\phi'']v+[\p_\tau^\beta, H\phi'']g-[\p_\tau^\beta,(u+U\phi')\p_x-U_x\phi\p_y]u\\
		&+[\p_\tau^\beta,(h+H\phi')\p_x-H_x\phi\p_y]h-\sum\limits_{0<\tilde\beta<\beta}\left(\begin{array}{ccc}
			\beta\\ \tilde\beta
		\end{array}\right)\Big(\p_\tau^{\tilde\beta} v\cdot\p_\tau^{\beta-\tilde\beta}\p_yu-\p_\tau^{\tilde\beta} g\cdot\p_\tau^{\beta-\tilde\beta}\p_yh\Big),\\
		R_h^\beta~=&\p_\tau^\beta\big(-H_x\phi'u+U_x\phi'h\big)-[\p_\tau^\beta, H\phi'']v+[\p_\tau^\beta, U\phi'']g-[\p_\tau^\beta,(u+U\phi')\p_x-U_x\phi\p_y]h\\
		&+[\p_\tau^\beta,(h+H\phi')\p_x-H_x\phi\p_y]u-\sum\limits_{0<\tilde\beta<\beta}\left(\begin{array}{ccc}
			\beta\\ \tilde\beta
		\end{array}\right)\Big(\p_\tau^{\tilde\beta} v\cdot\p_\tau^{\beta-\tilde\beta}\p_yh-\p_\tau^{\tilde\beta} g\cdot\p_\tau^{\beta-\tilde\beta}\p_yu\Big).
	\end{cases}
\end{equation}
From the expression \eqref{def_R} and  by using the inequalities \eqref{Morse}-\eqref{normal1}, we can control the $L_l^2(\Omega)-$estimates of  each term given in \eqref{def_R}, and then obtain the estimates of $\|R_u^\beta(t)\|_{L_l^2(\Omega)}$ and $\|R_h^\beta(t)\|_{L_l^2(\Omega)}$. For example, %it is easy to get $\|\p_\tau^\beta r_i\|_{L^2(\Omega)}\leq CM_0$ for $i=1,2$; 
for $\tilde\beta>0$, which implies that $\tilde\beta\geq e_i, i=1$ or 2, by virtue of \eqref{Morse}, 
\begin{align*}
	&\big\|\big[\p_\tau^{\tilde\beta}(u+U\phi')\p_x-\p_\tau^{\tilde\beta}(U_x\phi)\p_y\big](\p_\tau^{\beta-\tilde\beta}u)\big\|_{L_l^2(\Omega)}\\
	\leq~&\big\|\big[\p_\tau^{\tilde\beta-e_i}(\p_\tau^{e_i}u)\cdot\p_\tau^{\beta-\tilde\beta}(\p_xu)\big\|_{L_l^2(\Omega)}+\big\|\p_\tau^{\tilde\beta}(U\phi')(t)\big\|_{L^\infty(\Omega)}\|\p_x\p_\tau^{\beta-\tilde\beta}u(t)\|_{L_l^2(\Omega)}\\
	&+\Big\|\frac{\p_\tau^{\tilde\beta}(U_x\phi)(t)}{1+y}\Big\|_{L^\infty(\Omega)}\|\p_y\p_\tau^{\beta-\tilde\beta}u(t)\|_{L_{l+1}^2(\Omega)}\\
	\leq~&C\|\p_\tau^{e_i}u(t)\|_{\h_0^{m-1}}\|\p_xu(t)\|_{\h_l^{m-1}}+C\|\p_\tau^{\tilde\beta}(U, U_x)(t)\|_{L^\infty(\bT_x)}\|u(t)\|_{\h^m_l}\\
	\leq ~&C\big(\|\p_\tau^{\tilde\beta}(U, U_x)(t)\|_{L^\infty(\bT_x)}+\|u(t)\|_{\h^m_l}\big)\|u(t)\|_{\h^m_l},
\end{align*}
provided $m-1\geq3$ and $|\beta-\tilde\beta|\leq m-1$;
\eqref{normal1} gives that for $\tilde\beta<\beta$
\begin{align*}
	\big\|\p_\tau^{\tilde\beta}v\cdot\p_\tau^{\beta-\tilde\beta}(U\phi'')\big\|_{L_l^2(\Omega)}\leq~&\Big\|\frac{\p_\tau^{\tilde\beta+e_2}\p_y^{-1}u(t)}{1+y}\Big\|_{L^2(\Omega)}\big\|\ly^{l+1}\p_\tau^{\beta-\tilde\beta}(U\phi'')(t)\big\|_{L^\infty(\Omega)}\\
	\leq~& C\|\p_\tau^{\beta-\tilde\beta}U(t)\|_{L^\infty(\bT_x)}\|u(t)\|_{\h_0^m};
\end{align*}
moreover, for $0<\tilde\beta<\beta$ which implies that $\tilde\beta\geq e_i, \beta-\tilde\beta\geq e_j, i, j=1$ or 2, \eqref{normal1} yields that 
\begin{align*}
	\big\|\p_\tau^{\tilde\beta}v\cdot\p_\tau^{\beta-\tilde\beta}(\p_y u)\big\|_{L_l^2(\Omega)}=~&\big\|\p_\tau^{\tilde\beta-e_i}\p_y^{-1}(\p_\tau^{e_i+e_2}u)\cdot\p_\tau^{\beta-\tilde\beta-e_j}(\p_\tau^{e_j}\p_y u)\big\|_{L_l^2(\Omega)}\\
	\leq~&C\|\p_\tau^{e_i+e_2}u(t)\|_{\h_0^{m-2}}\|\p_y\p_\tau^{e_j}u(t)\|_{\h_{l+1}^{m-2}}\leq~C\|u(t)\|_{\h_l^m}^2
\end{align*}
provided $m-2\geq3.$ The other terms in $R_u^\beta$ and $R_h^\beta$ can be estimated similarly so that
\begin{align}\label{est_error-m}
	\|R_u^\beta(t)\|_{L_l^2(\Omega)},~\|R_h^\beta(t)\|_{L_l^2(\Omega)}\leq~&%\|\p_\tau^{\beta}(r_1,r_2)(t)\|_{L_l^2(\Omega)}+
	C\big(\sum_{|\beta|\leq m+2}\|\p_\tau^{\beta}(U, H)(t)\|_{L^2(\bT_x)}+\|(u,h)(t)\|_{\h_l^m}\big)\|(u,h)(t)\|_{\h_l^m}.
\end{align}

On the other hand, consider the equations \eqref{eq_xm}, the main 
difficulty comes from the terms 
$$(\p_yu+U\phi'')\p_\tau^\beta v-(\p_yh+H\phi'')\p_\tau^\beta g=-(\p_yu+U\phi'')\cdot(\p_y^{-1}\p_\tau^{\beta+e_2}u)+(\p_yh+H\phi'')\cdot(\p_y^{-1}\p_\tau^{\beta+e_2}h),$$ 
and 
$$(\p_yh+H\phi'')\p_\tau^\beta v-(\p_yu+U\phi'')\p_\tau^\beta g=-(\p_yh+H\phi'')\cdot(\p_y^{-1}\p_\tau^{\beta+e_2}u)+(\p_yu+U\phi'')\cdot(\p_y^{-1}\p_\tau^{\beta+e_2}h),$$ 
that contain the $m+1-$th order tangential derivatives which
can not  be controlled by the standard energy method. To overcome this difficulty, we rely on  the following two key observations. One is that from the equation \eqref{eq_h},  $\p_y^{-1}h$ satisfies the following equation (see also the equation \eqref{eq-psi} for $\psi$)
\[\p_t (\p_y^{-1}h)+(v-U_x\phi)(h+H\phi')-(g-H_x\phi)(u+U\phi')-\kappa\p_yh=-H_t\phi+\kappa H\phi'',\]%(u+U\phi')\p_x\p_y^{-1}h+v(h+H\phi')-U_x\phi h+H_x\phi u+H_t\phi(1-\phi')-\kappa H\phi^{(3)}=\kappa\p_y^2\p_y^{-1}b,\]
or
\[\p_t (\p_y^{-1}h)+(h+H\phi')v+(u+U\phi')\p_x(\p_y^{-1}h)-U_x\phi h+H_x\phi u-\kappa\p_yh=H_t\phi(\phi'-1)+\kappa H\phi'',\]
by using $g=-\p_x\p_y^{-1}h$ and the second relation of \eqref{Brou}. 
This inspires us  in the case of $h+H\phi'>0,$ to introduce the following two quantities 
\begin{equation}\label{new_qu}
 	 u_\beta:=\p_\tau^\beta u-\frac{\p_yu+U\phi''}{h+H\phi'}\p_\tau^\beta\p_y^{-1}h,\qquad  h_\beta:=\p_\tau^\beta h-\frac{\p_yh+H\phi''}{h+H\phi'}\p_\tau^\beta\p_y^{-1}h,
 \end{equation}
 to eliminate the terms involving $\p_\tau^\beta v$, then to avoid the loss of $x-$derivative on $v$. Note that the new quantities $(u_\beta, h_\beta)$ are almost equivalent to $\p_\tau^\beta(u,h)$ in $L^2_l$-norm, that is,
 \begin{equation}\label{equ}
 	\|\p_\tau^\beta(u,h)\|_{L^2_l(\Omega)}~\lesssim~\|(u_\beta,h_\beta)\|_{L^2_l(\Omega)}~\lesssim~\|\p_\tau^\beta(u,h)\|_{L^2_l(\Omega)},
 \end{equation}
 that will be proved at the end of this subsection.

Another observation is that by using the above two new 
unknowns $(u_\beta, h_\beta)$ in \eqref{new_qu}, the regularity loss generated by $g=-\p_y^{-1}\p_xh$, can be cancelled by using the convection terms $-(h+H\phi')\p_xh$ and $-(h+H\phi')\p_xu$, more precisely,  
\[\begin{split}
&-(h+H\phi')\p_x\p_\tau^\beta h-(\p_yh+H\phi'')\p_\tau^\beta g\\
=&-(h+H\phi')\p_x\Big(h_\beta+\frac{\p_yh+H\phi''}{h+H\phi'}\p_\tau^\beta\p_y^{-1}h\Big)+(\p_yh+H\phi'')\cdot(\p_y^{-1}\p_\tau^{\beta+e_2}h)\\
=&-(h+H\phi')\p_x h_\beta-(h+H\phi')\p_x\Big(\frac{\p_yh+H\phi''}{h+H\phi'}\Big)\cdot\p_\tau^\beta\p_y^{-1} h,	
\end{split}\]
and
\[\begin{split}
&-(h+H\phi')\p_x\p_\tau^\beta u-(\p_yu+U\phi'')\p_\tau^\beta g\\
=&-(h+H\phi')\p_x\Big(u_\beta+\frac{\p_yu+U\phi''}{h+H\phi'}\p_\tau^\beta\p_y^{-1}h\Big)+(\p_yu+U\phi'')\cdot(\p_y^{-1}\p_\tau^{\beta+e_2}h)\\
=&-(h+H\phi')\p_x u_\beta-(h+H\phi')\p_x\Big(\frac{\p_yu+U\phi''}{h+H\phi'}\Big)\cdot\p_\tau^\beta \p_y^{-1}h.
\end{split}\]
 This cancellation 
mechanism reveals the stabilizing effect of the magnetic field on the boundary layer.
Note that in the above expressions, the  convection terms can be handled by the symmetric structure of the system.
% and expect them there are only $m-$th order tangential derivatives, which allows us to achieve the desired estimates by the standard energy method.
Based on the above discussion, we will carry out the estimation as follows. First of all, we always assume that there exists a positive constant $\delta_0\leq1,$ such that
\begin{equation}
	\label{priori_ass}
	%|b(t,x,y)|\leq\frac{1}{2},\quad\mbox{and then,}\quad \frac{1}{2}\leq 1+b(t,x,y)\leq \frac{2}{3},
	h(t,x,y)+H(t,x)\phi'(y)\geq\delta_0,\qquad \mbox{for}\quad (t,x,y)\in [0,T]\times\Omega.%\quad 0<\delta_0\leq1.
\end{equation}
%and apply the above arguments in the following.

Firstly, from the divergence free condition $\p_xh+\p_yg=0$,
%\begin{align}\label{div_bg}\p_xh+\p_yg=0,\end{align}
there exists a stream function $\psi$, such that
\begin{align}\label{psi}
h=\p_y\psi,\quad g=-\p_x\psi,\quad \psi|_{y=0}=0.
\end{align}
Then, the equation \eqref{eq_h} for $h$ reads
\begin{align}\label{eq_psi}
&\p_t \p_y\psi+\p_y\big[(v-U_x\phi)(\p_y\psi+H\phi')+(\p_x\psi+H_x\phi)(u+U\phi')\big]-\kappa\p_y^3\psi=-H_t\phi'+\kappa H\phi^{(3)}.
\end{align}
By virtue of %another divergence free condition $\p_xu+\p_yv=0$,\begin{align}\label{div_uv}\p_xu+\p_yv=0\end{align}and 
the boundary conditions:
\begin{align*}
\p_t\psi|_{y=0}=\p_x\psi|_{y=0}=\p_y^2\psi|_{y=0}=v|_{y=0}=0,
\end{align*}
and $\phi(y)\equiv0$ for $y\in[0,R_0]$,
we integrate the equation (\ref{eq_psi}) with respect to the variable $y$ over $[0,y]$, to obtain
\begin{align}\label{eq-psi}
\p_t \psi+\big[(u+U\phi')\p_x+(v-U_x\phi)\p_y\big]\psi+H_x\phi u+H\phi'v-\kappa\p_y^2\psi=r_3,
\end{align}
with
\begin{align}\label{r_3}
	r_3~=~H_t\phi(\phi'-1)+\kappa H\phi^{(3)}.
\end{align}

Next, applying the $m$-th order tangential derivatives operator on (\ref{eq-psi}) and by virtue of $\p_y\psi=h$, it yields that
\begin{align}\label{psi-m}
\p_t \p_\tau^\beta\psi+\big[(u+U\phi')\p_x+(v-U_x\phi)\p_y\big]\p_\tau^\beta\psi+(h+H\phi')\p_\tau^\beta v-\kappa\p_y^2\p_\tau^\beta\psi=\p_\tau^\beta r_3+R_\psi^\beta,
\end{align}
where $R_\psi^\beta$ is defined as follows:
\begin{align}\label{r0}
R_\psi^\beta=~&-\p_\tau^\beta\big(H_x\phi u\big)-[\p_\tau^\beta,H\phi']v-[\p_\tau^\beta,(u+U\phi')\p_x-U_x\phi\p_y]\psi-\sum\limits_{0<\tilde\beta<\beta}\left(\begin{array}{ccc}
			\beta\\ \tilde\beta
		\end{array}\right)\big(\p_\tau^{\tilde\beta} v\cdot\p_\tau^{\beta-\tilde\beta}\p_y\psi\big).
		\end{align}
By $\psi=\p_y^{-1}h$ and $v=-\p_x\p_y^{-1}u$, it gives
\begin{align*}
		R_\psi^\beta=~&-\p_\tau^\beta\big(H_x\phi u\big)+[\p_\tau^\beta,H\phi']\p_x\p_y^{-1}u-[\p_\tau^\beta,(u+U\phi')]\p_x\p_y^{-1}h+[\p_\tau^\beta, U_x\phi]h\\
		&+\sum\limits_{0<\tilde\beta<\beta}\left(\begin{array}{ccc}
			\beta\\ \tilde\beta
		\end{array}\right)\big(\p_\tau^{\tilde\beta+e_2} \p_y^{-1}u\cdot\p_\tau^{\beta-\tilde\beta}h\big)\\
		=~&-\sum\limits_{\tilde\beta\leq\beta}\left(\begin{array}{ccc}
			\beta\\ \tilde\beta
		\end{array}\right)\big[\p_\tau^{\tilde\beta}(H_x\phi)\cdot\p_\tau^{\beta-\tilde\beta}u\big]+\sum\limits_{0<\tilde\beta\leq\beta}\left(\begin{array}{ccc}
			\beta\\ \tilde\beta
		\end{array}\right)\Big[\p_\tau^{\tilde\beta}(H\phi')\cdot\p_\tau^{\beta-\tilde\beta+e_2}\p_y^{-1}u\\
		&\qquad-\p_\tau^{\tilde\beta}(u+U\phi')\cdot\p_\tau^{\beta-\tilde\beta+e_2}\p_y^{-1}h+\p_\tau^{\tilde\beta}(U_x\phi)\cdot\p_\tau^{\beta-\tilde\beta}h \Big]+\sum\limits_{0<\tilde\beta<\beta}\left(\begin{array}{ccc}
			\beta\\ \tilde\beta
		\end{array}\right)\big(\p_\tau^{\tilde\beta+e_2} \p_y^{-1}u\cdot\p_\tau^{\beta-\tilde\beta}h\big),
		\end{align*}
and then, we can estimate $\Big\|\frac{R_\psi^\beta(t)}{1+y}\Big\|_{L^2(\Omega)}$ from the above expression term by term. For example, it is easy to get that 
\begin{align*}
	%\Big\|\frac{\p_\tau^\beta r_3}{1+y}\Big\|_{L^2(\Omega)}\leq CM_0, \quad 
\Big\|\frac{\p_\tau^{\tilde\beta}(H_x\phi)\cdot\p_\tau^{\beta-\tilde\beta}u}{1+y}\Big\|_{L^2(\Omega)}\leq~&\Big\|\frac{\p_\tau^{\tilde\beta}(H_x\phi)(t)}{1+y}\Big\|_{L^\infty(\Omega)}\big\|\p_\tau^{\beta-\tilde\beta}u(t)\big\|_{L^2(\Omega)}\leq~ C\|\p_\tau^{\tilde\beta}H_x(t)\|_{L^\infty(\bT_x)}\|u(t)\|_{\h_0^m},
\end{align*}
%\begin{align*}\Big\|\frac{\p_\tau^{\tilde\beta}(H_x\phi)\cdot\p_\tau^{\beta-\tilde\beta}u}{1+y}\Big\|_{L^2(\Omega)}\leq\Big\|\frac{\p_\tau^{\tilde\beta}(H_x\phi)(t)}{1+y}\Big\|_{L^\infty(\Omega)}\big\|\p_\tau^{\beta-\tilde\beta}u(t)\big\|_{L^2(\Omega)}\leq CM_0\|u(t)\|_{\h_0^m},\\
	%\Big\|\frac{\p_\tau^{\tilde\beta}(H_x\phi)\cdot\p_\tau^{\beta-\tilde\beta}h}{1+y}\Big\|_{L^2(\Omega)}\leq\Big\|\frac{\p_\tau^{\tilde\beta}(H_x\phi)(t)}{1+y}\Big\|_{L^\infty(\Omega)}\big\|\p_\tau^{\beta-\tilde\beta}h(t)\big\|_{L^2(\Omega)}\leq CM_0\|h(t)\|_{\h_0^m},
%\end{align*}
and \eqref{normal1} implies that
\begin{align*}
	\Big\|\frac{\p_\tau^{\tilde\beta}(H\phi')\cdot\p_\tau^{\beta-\tilde\beta+e_2}\p_y^{-1}u}{1+y}\Big\|_{L^2(\Omega)}\leq\big\|\p_\tau^{\tilde\beta}(H\phi')(t)\big\|_{L^\infty(\Omega)}\Big\|\frac{\p_\tau^{\beta-\tilde\beta+e_2}\p_y^{-1}u(t)}{1+y}\Big\|_{L^2(\Omega)}\leq C\|\p_\tau^{\tilde\beta}H(t)\|_{L^\infty(\bT_x)}\|u(t)\|_{\h_0^m},\\
	%\Big\|\frac{\p_\tau^{\tilde\beta}(U\phi')\cdot\p_\tau^{\beta-\tilde\beta+e_2}\p_y^{-1}h}{1+y}\Big\|_{L^2(\Omega)}\leq\big\|\p_\tau^{\tilde\beta}(U\phi')(t)\big\|_{L^\infty(\Omega)}\Big\|\frac{\p_\tau^{\beta-\tilde\beta+e_2}\p_y^{-1}h(t)}{1+y}\Big\|_{L^2(\Omega)}\leq CM_0\|h(t)\|_{\h_0^m},
\end{align*}
provided $|\beta-\tilde\beta|\leq|\beta|-1=m-1$. Also, \eqref{normal1} allows us to get that for $\tilde\beta\geq e_i, i=1$ or 2,
\begin{align*}
	\Big\|\frac{\p_\tau^{\tilde\beta}u\cdot\p_\tau^{\beta-\tilde\beta+e_2}\p_y^{-1}h}{1+y}\Big\|_{L^2(\Omega)}=~&\big\|\p_\tau^{\tilde\beta-e_i}(\p_\tau^{e_i}u)\cdot\p_\tau^{\beta-\tilde\beta}\p_y^{-1}(\p_xh)\big\|_{L_{-1}^2(\Omega)}\\
	\leq~ &C\|\p_\tau^{e_i}u(t)\|_{\h_0^{m-1}}\|\p_xh(t)\|_{\h_0^{m-1}}\leq C\|(u, h)(t)\|_{\h_0^m}^2.
	\end{align*}
The other terms in $R_\psi^\beta$ can be estimated similarly,  and we have
\begin{equation}\label{est_r0}
	\Big\|\frac{R_\psi^\beta(t)}{1+y}\Big\|_{L^2(\Omega)}\leq %\Big\|\frac{\p_\tau^\beta r_3(t)}{1+y}\Big\|_{L^2(\Omega)}+
	C\big(\sum_{|\beta|\leq m+2}\|\p_\tau^{\beta}(U,H)(t)\|_{L^2(\bT_x)}+\|(u, h)\|_{\h_0^m}\big)\|(u, h)\|_{\h_0^m}.
\end{equation}

Now, combining \eqref{new_qu} with \eqref{psi},
we define  new functions:
\begin{align}\label{new}
u_\beta=\p_\tau^\beta u-\frac{\p_yu+U\phi''}{h+H\phi'}\p_\tau^\beta\psi,\quad h_\beta=\p_\tau^\beta h-\frac{\p_yh+H \phi''}{h+H\phi'}\p_\tau^\beta\psi, %\hat{b}_\alpha=\p_\tau^\beta b-\frac{\p_yb}{1+b}\p_\tau^\beta\psi
\end{align}
and denote 
\begin{equation}
	\label{def_eta}
\eta_1~\triangleq~\frac{\p_yu+U\phi''}{h+H\phi'},\quad \eta_2~\triangleq~\frac{\p_yh+H\phi''}{h+H\phi'}.
\end{equation}
Then, by noting that $\p_\tau^\beta g=-\p_x\p_\tau^\beta\psi$  from \eqref{psi}, we compute $(\ref{eq_xm})_1 -(\ref{eq-psi})\times\eta_1$ and $(\ref{eq_xm})_2 -(\ref{eq-psi})\times\eta_2$ respectively, to obtain that 
\begin{align}\label{eq_hu}
\begin{cases}
	\p_tu_\beta +\big[(u+U\phi')\partial_x+(v-U_x\phi)\partial_y\big]u_\beta -\big[(h+H\phi')\partial_x+(g-H_x\phi)\partial_y\big]h_\beta-\mu\p_y^2u_\beta +(\kappa-\mu)\eta_1\p_y h_\beta &=R_1^\beta,\\
\p_th_\beta +\big[(u+U\phi')\partial_x+(v-U_x\phi)\partial_y\big]h_\beta -\big[(h+H\phi')\partial_x+(g-H_x\phi)\partial_y\big]u_\beta-\kappa\p_y^2h_\beta &=R_2^\beta,
\end{cases}
%\p_t\hat{u}_\alpha+(u_s+u)\p_x\hat{u}_\alpha+v\p_y\hat{u}_\alpha-(1+b)\p_x\hat{b}_\alpha-g\p_y\hat{b}_\alpha=\mu\p_y^2\hat{u}_\alpha+\widehat{R}_1,
\end{align}
where
\begin{align}\label{def_newr}
\begin{cases}
R_1^\beta&=\p_\tau^\beta r_1-\eta_1\p_\tau^\beta r_3+%(\mu-\kappa)\eta_1\p_y\hat b_\alpha
R_u^\beta-\eta_1R_\psi^\beta+[2\mu\p_y\eta_1+(g-H_x\phi)\eta_2+(\mu-\kappa)\eta_1\eta_2]\p_\tau^\beta h-\zeta_1\p_\tau^\beta\psi,\\
%:=&~(\mu-\kappa)\eta_1\p_y\hat b_\alpha+\widetilde R_1,\\
R_2^\beta&=\p_\tau^\beta r_1-\eta_2\p_\tau^\beta r_2+%\kappa\p_y\big(\p_y\eta_2\p_x^m\psi\big)+
R_h^\beta-\eta_2R_\psi^\beta+\big[2\kappa\p_y\eta_2+(g-H_x\phi)\eta_1\big]\p_\tau^\beta h-\zeta_2\p_\tau^\beta\psi,
\end{cases}
%(1+b)\p_x(\frac{\p_yb}{1+b}\p_\tau^m\psi)+\p_\tau^mg\p_yb-g\p_y\hat{b}-R_0\frac{\p_y(u_s+u)}{1+b}\\
%&-\p_\tau^m\psi\Big[\p_t(\frac{\p_y(u_s+u)}{1+b})+(u_s+u)\p_x(\frac{\p_y(u_s+u)}{1+b})-k\p_y^2(\frac{\p_y(u_s+u)}{1+b})\Big]\\
%&+2k\p_y(\frac{\p_y(u_s+u)}{1+b})\p_y\p^m_\tau\psi+(\mu-k)\p_y^2\Big[(\frac{\p_y(u_s+u)}{1+b})\p_\tau^m\psi\Big]\\
%\triangleq &\bar{R}^1_1+(\mu-k)\p_y^2\Big[(\frac{\p_y(u_s+u)}{1+b})\p_\tau^m\psi\Big].
\end{align}
with
\begin{align}\label{zeta}
	\zeta_1~&=~\p_t\eta_1+\big[(u+U\phi')\p_x+(v-U_x\phi)\p_y\big]\eta_1-\big[(h+H\phi')\p_x+(g-H_x\phi)\p_y\big]\eta_2-\mu\p_y^2\eta_1+(\kappa-\mu)\eta_1\p_y\eta_2,\nonumber\\
	\zeta_2~&=~\p_t\eta_2+\big[(u+U\phi')\p_x+(v-U_x\phi)\p_y\big]\eta_2-\big[(h+H\phi')\p_x+(g-H_x\phi)\p_y\big]\eta_1-\kappa\p_y^2\eta_2.
\end{align}
Also, direct calculation gives the corresponding initial-boundary values as follows:
\begin{equation}\label{ib_hat}
\begin{cases}
&u_\beta|_{t=0}=\p_\tau^\beta u(0,x,y)-\frac{\p_yu_{0}(x,y)+U(0,x)\phi''(y)}{h_0(x,y)+H(0,x)\phi'(y)}\int_0^y\p_\tau^\beta h(0,x,z)dz\triangleq u_{\beta 0}(x,y),\\
&h_\beta|_{t=0}=\p_\tau^\beta h(0,x,y)-\frac{\p_yh_{0}(x,y)+H(0,x)\phi''(y)}{h_0(x,y)+H(0,x)\phi'(y)}\int_0^y\p_\tau^\beta h(0,x,z)dz\triangleq h_{\beta 0}(x,y),\\
&u_\beta|_{y=0}=0,\quad \p_y h_\beta|_{y=0}=0.	
%&\hat b_\alpha|_{t=0}=\p_\tau^\beta b(0,x,y)-\frac{\p_yb_0(x,y)}{1+b_0(x,y)}\int_0^y\p_\tau^\beta b(0,x,z)dz\triangleq\hat b_{\alpha 0}(x,y),\\
%&\hat{u}_\alpha|_{y=0}=0,\quad \p_y\hat{b}_\alpha|_{y=0}=0.	
\end{cases}\end{equation}
Finally, we obtain the initial-boundary value problem for $(u_\beta, h_\beta)$:
\begin{equation}\label{pr_hat}
\begin{cases}
	%&\p_t\hat{u}_\alpha+(u_s+u)\p_x\hat{u}_\alpha+v\p_y\hat{u}_\alpha-(1+b)\p_x\hat{b}_\alpha-g\p_y\hat{b}_\alpha=\mu\p_y^2\hat{u}_\alpha+\widehat{R}_1,\\
	%&\p_t\hat{b}_\alpha+(u_s+u)\p_x\hat{b}_\alpha+v\p_y\hat{b}_\alpha-(1+b)\p_x\hat{u}_\alpha-g\p_y\hat{u}_\alpha=\kappa\p_y^2\hat{b}_\alpha+\widehat{R}_2,\\
	%&(\hat u_\alpha,\p_y\hat b_\alpha)|_{y=0}=0,\qquad (\hat u_\alpha,\hat b_\alpha)|_{t=0}=(\hat u_{\alpha 0},\hat b_{\alpha0})(x,y)
	\p_tu_\beta +\big[(u+U\phi')\partial_x+(v-U_x\phi)\partial_y\big]u_\beta -\big[(h+H\phi')\partial_x+(g-H_x\phi)\partial_y\big]h_\beta-\mu\p_y^2u_\beta+(\kappa-\mu)\eta_1\p_yh_\beta=R_1^\beta,\\
\p_th_\beta +\big[(u+U\phi')\partial_x+(v-U_x\phi)\partial_y\big]h_\beta -\big[(h+H\phi')\partial_x+(g-H_x\phi)\partial_y\big]u_\beta-\kappa\p_y^2h_\beta =R_2^\beta,\\
	( u_\beta,\p_y h_\beta)|_{y=0}=0,\qquad (u_\beta,h_\beta)|_{t=0}=(u_{\beta 0},h_{\beta 0})(x,y),
\end{cases}\end{equation}
with the initial data $(u_{\beta 0},h_{\beta 0})(x,y)$ given by \eqref{ib_hat}.
Moreover, by combining $\psi=\p_y^{-1}h$ with \eqref{normal1}, %for $\lambda=1$,
\begin{align}\label{est_psi}
	\|\ly^{-1}\p_\tau^\beta\psi(t)\|_{L^2(\Omega)}\leq 2\|\p_\tau^\beta h(t)\|_{L^2(\Omega)}.
\end{align}
From the expression\eqref{def_eta} of $\eta_1$ and $\eta_2$, by \eqref{priori_ass} and Sobolev embedding inequality we have that for $\lambda\in\bR$ and $i=1,2,$
\begin{align}\label{est_eta}
	\|\ly^\lambda\eta_i\|_{L^\infty(\Omega)}\leq C\delta_0^{-1}\big(\|(U,H)(t)\|_{L^\infty(\bT_x)}+\|(u,h)\|_{\h_{\lambda-1}^3}\big),\nonumber\\
	\|\ly^\lambda\p_y\eta_i\|_{L^\infty(\Omega)}\leq C\delta_0^{-2}\big(\|(U,H)(t)\|_{L^\infty(\bT_x)}+\|(u,h)\|_{\h_{\lambda-1}^4}\big)^2,
\end{align}
and
\begin{align}\label{est_zeta}
	\|\ly^\lambda\zeta_i\|_{L^\infty(\Omega)}\leq C\delta_0^{-3}\big(\sum_{|\beta|\leq1}\|\p_\tau^\beta(U,H)(t)\|_{L^\infty(\bT_x)}+\|(u,h)\|_{\h_{\lambda-1}^5}\big)^3,\qquad i=1,2.
\end{align}
Then, %note that $\|\ly^{-1}\p_\tau^\beta\psi\|_{L^2(\Omega)}\leq 2\|\p_\tau^\beta h\|_{L^2(\Omega}$ by \eqref{normal}, 
for the terms $R_1^\beta$ and $ R_2^\beta$ given by \eqref{def_newr}, from the above inequalities \eqref{est_psi}-\eqref{est_zeta}, the estimates \eqref{est_error-m} and \eqref{est_r0} we obtain that for $|\beta|=m\geq5, l\geq0,$
% for $m\geq4,$
\begin{equation}
	\label{est-r1}
	\begin{split}
		\|R_1^\beta(t)\|_{L_l^2(\Omega)}\leq~&\|\p_\tau^\beta r_1-\eta_1\p_\tau^\beta r_3\|_{L_l^2(\Omega)}+\|R_u^\beta\|_{L_l^2(\Omega)}+\|\ly^{l+1}\eta_1\|_{L^{\infty}(\Omega)}\|\ly^{-1}R_\psi^\beta\|_{L^2(\Omega)}\\
		&+\big(\big\|2\mu\p_y\eta_1+(\mu-\kappa)\eta_1\eta_2\big\|_{L^\infty(\Omega)}+\big\|\ly^{-1}(g-H_x\phi)\big\|_{L^\infty(\Omega)}\big\|\ly\eta_2\big\|_{L^\infty(\Omega)}\big)\|\p_\tau^\beta h\big\|_{L^2_l(\Omega)}\\
		&+\|\ly^{l+1}\zeta_1\|_{L^\infty(\Omega)}\|\ly^{-1}\p_\tau^\beta\psi\|_{L^2(\Omega)}\\
		\leq~&\|\p_\tau^\beta r_1-\eta_1\p_\tau^\beta r_3\|_{L_l^2(\Omega)}+C\delta_0^{-3}\big(\sum_{|\beta|\leq m+2}\|\p_\tau^{\beta}(U,H)(t)\|_{L^2(\bT_x)}+\|(u,h)\|_{\h_{l}^m}\big)^3\|(u ,h)(t)\|_{\h_l^m},\\
		%C\delta_0^{-3}\big(M_0+\|(u,h)\|_{\h_{l}^m}\big)^4,\\%\big(1+\|(u,h)\|_{\h_{l}^m}\big),\\%+\|\zeta_1\|_{L_{tx}^\infty L_{y,l}^2(\Omega)}\|\p_\tau^m\psi\|_{L^2_{tx}L_y^\infty(\Omega)}\\
		%&+\|R_1\|_{L_L^2(\Omega)}+\|\eta_1\|_{L^\infty(\Omega)}\|R_0\|_{L_L^2(\Omega)}\\
		%\lesssim& P\big(E_3(t)\big)\Big(\|(u,b)\|_{\A_l^{m,0}(\Omega)}+\|(u,b)\|_{\A_l^{m-1,1}(\Omega)}\Big),\\
		%\Big(1+P\big(E_{m-1}(t)\big)\Big)\Big(1+\|(u,b)(t)\|_{H_l^{m,0}(\Omega)}\Big),\\
		\end{split}\end{equation}
		and
		\begin{equation}\label{est-r2}\begin{split}
		\|R_2^\beta(t)\|_{L_l^2(\Omega)}\leq~&\|\p_\tau^\beta r_2-\eta_2\p_\tau^\beta r_3\|_{L_l^2(\Omega)}+\|R_h^\beta\|_{L_l^2(\Omega)}+\|\ly^{l+1}\eta_2\|_{L^{\infty}(\Omega)}\|\ly^{-1}R_\psi^\beta\|_{L^2(\Omega)}\\
		&+\big(\big\|2\kappa\p_y\eta_2\big\|_{L^\infty(\Omega)}+\big\|\ly^{-1}(g-H_x\phi)\big\|_{L^\infty(\Omega)}\big\|\ly\eta_1\big\|_{L^\infty(\Omega)}\big)\|\p_\tau^\beta h\big\|_{L^2_l(\Omega)}\\
		&+\|\ly^{l+1}\zeta_2\|_{L^\infty(\Omega)}\|\ly^{-1}\p_\tau^\beta\psi\|_{L^2(\Omega)}\\
		\leq~&\|\p_\tau^\beta r_2-\eta_2\p_\tau^\beta r_3\|_{L_l^2(\Omega)}+C\delta_0^{-3}\big(\sum_{|\beta|\leq m+2}\|\p_\tau^{\beta}(U,H)(t)\|_{L^2(\bT_x)}+\|(u,h)\|_{\h_{l}^m}\big)^3\|(u ,h)(t)\|_{\h_l^m}.
		%\lesssim& P\big(E_3(t)\big)\Big(\|(u,b)\|_{\A_l^{m,0}(\Omega)}+\|(u,b)\|_{\A_l^{m-1,1}(\Omega)}\Big),
		%P\big(E_{m-1}(t)\big)\Big(1+\|(u,b)(t)\|_{H_l^{m,0}}\Big),
	\end{split}
\end{equation}
%where we have used the equations of $\p_yu$ and $\p_yb$ to estimate $\|\zeta_1\|_{L_{tx}^\infty L_{y,l}^2(\Omega)}$ and $\|\zeta_2\|_{L_{tx}^\infty L_{y,l}^2(\Omega)}$.

Now, we are going to derive the following $L^2_l$-norms of $(u_\beta,h_\beta)$.
%$ and $\hat{b}$.
\begin{prop}\label{prop_xm}[\textit{$L^2_l-$estimate on $(u_\beta,h_\beta)$}]\\
	Under the hypotheses of Proposition \ref{prop_priori},  we have that for any $t\in[0,T]$ and the quantity $(u_\beta,h_\beta)$ given in \eqref{new},
	\begin{align}\label{est_hat}%\begin{split}
		&\sum_{|\beta|=m}\Big(\frac{d}{dt}\|(u_\beta, h_\beta)(t)\|_{L^2_l(\mathbb{R}^2_+)}^2%-\|(\hat u_\alpha,\hat{b}_\alpha)(0)\|_{L^2_l(\mathbb{R}^2_+)}^2
		+\mu\|\p_y u_\beta(t)\|_{L^2_l(\Omega)}^2+\kappa\|\p_y h_\beta(t)\|_{L^2_l(\Omega)}^2\Big)\nonumber\\
		\leq~&\sum_{|\beta|=m}\Big(\|\p_\tau^\beta r_1-\eta_1\p_\tau^\beta r_3\|_{L_l^2(\Omega)}^2+\|\p_\tau^\beta r_2-\eta_2\p_\tau^\beta r_3\|_{L_l^2(\Omega)}^2\Big)\nonumber\\
	&+C\delta_0^{-2}\big(\sum_{|\beta|\leq m+2}\|\p_\tau^{\beta}(U,H)(t)\|_{L^2(\bT_x)}+\|(u,h)(t)\|_{\h_l^m}\big)^2\Big(\sum_{|\beta|=m}\|(u_\beta, h_\beta)(t)\|_{L^2_l(\Omega)}^2\Big)\nonumber\\
	&+C\delta_0^{-4}\big(\sum_{|\beta|\leq m+2}\|\p_\tau^{\beta}(U,H)(t)\|_{L^2(\bT_x)}+\|(u,h)\|_{\h_{l}^m}\big)^4\|(u ,h)(t)\|_{\h_l^m}^2.
	%\end{split}
	\end{align}
\end{prop}
\begin{proof}[\textbf{Proof.}]
Multiplying $(\ref{pr_hat})_1$ and $\eqref{pr_hat}_2$ by $\langle y\rangle^{2l}u_\beta$ and $\langle y\rangle^{2l}h_\beta$ respectively, and integrating them over $\Omega$ with $t\in[0,T]$, we obtain that by integration by parts,
\iffalse
\begin{align}\label{est_m}
\frac{1}{2}\frac{d}{dt}\big(\|\hat{u}_\alpha\|_{L^2_l(\mathbb{R}^2_+)}+\|\hat{b}_\alpha\|_{L^2_l(\mathbb{R}^2_+)}\big)-l\int_0_{\bR^2_+}\ly^{2l-1}v\big(|\hat u|^2+|\hat b|^2\big)dxdy+2l\int_{\Omega}\ly^{2l-1}g(\hat u \hat b)dxdy\nonumber\\
-\mu\int_0_{\mathbb{R}^2_+}\p_y^2\hat{u}\hat{u}\langle y\rangle^{2l}dxdy=\int_0_{\mathbb{R}^2_+}\bar{R}_1\hat{u}\langle y\rangle^{2l}dxdy.
\end{align}
And, multiplying (\ref{2.24}) by $\hat{b}\langle y\rangle^{2l}$ and integrating the resulting equation over $\mathbb{R}^2_+$ give
\begin{align}
\label{2.27}
\frac{1}{2}\frac{d}{dt}\|\hat{b}\|_{L^2_l(\mathbb{R}^2_+)}+\int_0_{\mathbb{R}^2_+}((u_s+u)\p_x\hat{b}+v\p_y\hat{b}-(1+b)\p_x\hat{u}-g\p_y\hat{u})\hat{b}\langle y\rangle^{2l}dxdy\nonumber\\
-k\int_0_{\mathbb{R}^2_+}\p_y^2\hat{b}\hat{b}\langle y\rangle^{2l}dxdy=\int_0_{\mathbb{R}^2_+}\bar{R}_2\hat{b}\langle y\rangle^{2l}dxdy.
\end{align}
Add (\ref{2.26}) and (\ref{2.27}) together, and then integrate by parts, we obtain
\fi
\begin{align}\label{est_m}
&\frac{1}{2}\frac{d}{dt}\|(u_\beta,h_\beta)(t)\|_{L^2_l(\mathbb{R}^2_+)}^2%-\|(\hat u_{\alpha0},\hat{b}_{\alpha0})\|_{L^2_l(\mathbb{R}^2_+)}^2\Big)
+\mu\|\p_y u_\beta\|_{L^2_l(\Omega)}^2+\kappa\|\p_y h_\beta\|_{L^2_l(\Omega)}^2\nonumber\\
=~&2l\int_{\Omega}\ly^{2l-1}\big[(v-U_x\phi)\frac{u_\beta^2+h_\beta^2}{2}-(g-H_x\phi)u_\beta h_\beta\big]dxdy+(\mu-\kappa)\int_{\Omega}\ly^{2l}\big(\eta_1\p_yh_\beta\cdot u_\beta\big)dxdy\nonumber\\
&+\int_{\Omega}\ly^{2l}\big(u_\beta R^\beta_1+h_\beta R^\beta_2\big)dxdy%2l\int_{\Omega}\ly^{2l-1}\big[(v-U_x\phi)\frac{u_\beta^2+h_\beta^2}{2}-(g-H_x\phi)u_\beta h_\beta\big]dxdy 
-2l\int_{\Omega}\ly^{2l-1}\big(\mu u_\beta\p_y u_\beta+\kappa h_\beta\p_y h_\beta\big)dxdy,
%\nonumber\\&+\int_{\Omega}\ly^{2l}\big(u_\beta R^\beta_1+h_\beta R^\beta_2\big)dxdy,%+\int_{\Omega}\ly^{2l}\hat b_\alpha\widehat{R}_2dxdy,
\end{align}
where we have used the boundary conditions in \eqref{pr_hat} and $(v,g)|_{y=0}=0.$

By  \eqref{normal}, it gives that
\begin{align}\label{est_m0}
&\Big|2l\int_{\Omega}\ly^{2l-1}\big[(v-U_x\phi)\frac{u_\beta^2+h_\beta^2}{2}-(g-H_x\phi)u_\beta h_\beta\big]dxdy\Big|\nonumber\\
\leq~&2l\Big(\Big\|\frac{v-U_x\phi}{1+y}\Big\|_{L^\infty(\Omega)}+\Big\|\frac{g-H_x\phi}{1+y}\Big\|_{L^\infty(\Omega)}\Big)\|(u_\beta, h_\beta)\|_{L_l^2(\Omega)}^2\nonumber\\%\big(\|\hat{u}_\alpha\|^2_{L^2_l(\Omega)}+\|\hat{b}_\alpha\|^2_{L^2_l(\Omega)}\big)\nonumber\\
\leq~&2l\big(\|(U_x,H_x)(t)\|_{L^\infty(\bT_x)}+\|u_x(t)\|_{L^\infty(\Omega)}+\|h_x(t)\|_{L^\infty(\Omega)}\big)\|(u_\beta, h_\beta)(t)\|_{L_l^2(\Omega)}^2\nonumber\\
%\big(\|\hat{u}_m\|_{L^2_l(\mathbb{R}^2_+)}^2+\|\hat{b}_m\|_{L^2_l(\mathbb{R}^2_+)}^2\big)\nonumber\\
\leq~&C\big(\|(U_x,H_x)(t)\|_{L^\infty(\bT_x)}+\|(u, h)(t)\|_{\h_l^m}\big)\|(u_\beta, h_\beta)(t)\|_{L_l^2(\Omega)}^2.%CP\big(E_2(t)\big)\|(\hat{u}_\alpha,\hat b_\alpha)\|_{L^2_l(\Omega)}^2,
%\Big(\|\hat{u}_m(t)\|^2_{L^2_l(\mathbb{R}^2_+)}+\|\hat{b}_m(t)\|^2_{L^2_l(\mathbb{R}^2_+)}\Big),
\end{align}
By integration by parts and the boundary condition $u_\beta|_{y=0}=0$, we obtain that%and by  \eqref{est_eta},
\begin{align}\label{est_m1}
	&(\mu-\kappa)\int_{\Omega}\ly^{2l}\big(\eta_1\p_yh_\beta\cdot u_\beta\big)dxdy\nonumber\\
	=&-\mu\int_{\Omega}h_\beta\p_y\big(\ly^{2l}\eta_1u_\beta\big)dxdy-\kappa \int_{\Omega}\ly^{2l}\big(\eta_1\p_yh_\beta\cdot u_\beta\big)dxdy\nonumber\\
	\leq~&\frac{\mu}{4}\|\p_y u_\beta(t)\|_{L^2_l(\Omega)}^2+\frac{\kappa}{4}\|\p_y h_\beta(t)\|_{L^2_l(\Omega)}^2+C\big(1+\|\eta_1(t)\|_{L^\infty(\Omega)}^2+\|\p_y\eta_1(t)\|_{L^\infty(\Omega)}\big)\|(u_\beta, h_\beta)(t)\|_{L^2_l(\Omega)}^2\nonumber\\
	\leq~&\frac{\mu}{4}\|\p_y u_\beta(t)\|_{L^2_l(\Omega)}^2+\frac{\kappa}{4}\|\p_y h_\beta(t)\|_{L^2_l(\Omega)}^2+C\delta_0^{-2}\big(\|(U,H)(t)\|_{L^\infty(\bT_x)}+\|(u,h)(t)\|_{\h_l^m}\big)^2\|(u_\beta, h_\beta)(t)\|_{L^2_l(\Omega)}^2,
\end{align}
where we have used \eqref{est_eta} in the above second inequality.

Next, it is easy to get that by  \eqref{est-r1} and \eqref{est-r2},
\begin{align}\label{est_m2}
	\int_{\Omega}\ly^{2l}\big(u_\beta R^\beta_1+h_\beta R^\beta_2\big)dxdy
	\leq~&\|u_\beta(t)\|_{L_l^2(\Omega)}\|R_1^\beta(t)\|_{L_l^2(\Omega)}+\|h_\beta(t)\|_{L_l^2(\Omega)}\|R_2^\beta(t)\|_{L_l^2(\Omega)}\nonumber\\
	\leq~&\|\p_\tau^\beta r_1-\eta_1\p_\tau^\beta r_3\|_{L_l^2(\Omega)}^2+\|\p_\tau^\beta r_2-\eta_2\p_\tau^\beta r_3\|_{L_l^2(\Omega)}^2\nonumber\\
	&+C\delta_0^{-2}\big(\sum_{|\beta|\leq m+2}\|\p_\tau^{\beta}(U,H)(t)\|_{L^2(\bT_x)}+\|(u,h)(t)\|_{\h_l^m}\big)^2\|(u_\beta, h_\beta)(t)\|_{L^2_l(\Omega)}^2\nonumber\\
	&+C\delta_0^{-4}\big(\sum_{|\beta|\leq m+2}\|\p_\tau^{\beta}(U,H)(t)\|_{L^2(\bT_x)}+\|(u,h)\|_{\h_{l}^m}\big)^4\|(u ,h)(t)\|_{\h_l^m}^2.
\end{align}
Also,
\begin{align}\label{est_m3}
&\Big|2l\int_{\Omega}\ly^{2l-1}\big(\mu u_\beta\p_y u_\beta+\kappa h_\beta\p_y h_\beta\big)dxdy\Big|\nonumber\\
\leq~&\frac{\mu}{4}\|\p_y u_\beta(t)\|_{L^2_l(\Omega)}^2+\frac{\kappa}{4}\|\p_y h_\beta(t)\|_{L^2_l(\Omega)}^2+C\|(u_\beta, h_\beta)(t)\|_{L^2_l(\Omega)}^2.
%\big(\|\hat{u}_m(t)\|_{L^2_l(\mathbb{R}^2_+)}^2+\|\hat{b}_m(t)\|_{L^2_l(\mathbb{R}^2_+)}^2\big).
\end{align}

Substituting \eqref{est_m0}-\eqref{est_m3} into \eqref{est_m} yields that
\begin{align}\label{est-m}%\begin{split}
		&\frac{d}{dt}\|(u_\beta,h_\beta)(t)\|_{L^2_l(\mathbb{R}^2_+)}^2
+\mu\|\p_y u_\beta\|_{L^2_l(\Omega)}^2+\kappa\|\p_y h_\beta\|_{L^2_l(\Omega)}^2\nonumber\\
\leq~&\|\p_\tau^\beta r_1-\eta_1\p_\tau^\beta r_3\|_{L_l^2(\Omega)}^2+\|\p_\tau^\beta r_2-\eta_2\p_\tau^\beta r_3\|_{L_l^2(\Omega)}^2\nonumber\\
	&+C\delta_0^{-2}\big(\sum_{|\beta|\leq m+2}\|\p_\tau^{\beta}(U,H)(t)\|_{L^2(\bT_x)}+\|(u,h)(t)\|_{\h_l^m}\big)^2\|(u_\beta, h_\beta)(t)\|_{L^2_l(\Omega)}^2\nonumber\\
	&+C\delta_0^{-4}\big(\sum_{|\beta|\leq m+2}\|\p_\tau^{\beta}(U,H)(t)\|_{L^2(\bT_x)}+\|(u,h)\|_{\h_{l}^m}\big)^4\|(u ,h)(t)\|_{\h_l^m}^2,
	%\end{split}
\end{align}
thus we prove \eqref{est_hat} by taking the summation over all $|\beta|=m$ in \eqref{est-m}. 
\end{proof}

Finally, we give the following result, which shows the almost equivalence in $L_l^2-$norm between $\p_\tau^\beta(u,h)$ and the quantities $(u_\beta,h_\beta)$ given by \eqref{new}.

\begin{lem}\label{lem_equ}[\textit{Equivalence between  $\|\p_\tau^\beta(u,h)\|_{L_l^2}$ and $\|(u_\beta,h_\beta)\|_{L_l^2}$}]\\
	%If $(u,h)$ is a smooth solution of problem \eqref{bl_main} with $(u,h)\in L^\infty\big(0,T;H^{3}_l(\Omega)\big)$ for some $l>\frac{1}{2}$, 
	If the smooth function $(u,h)$ satisfies the problem \eqref{bl_main} in $[0,T]$, and \eqref{priori_ass} holds, then for any $t\in[0,T], l\geq0,$ aninteger $m\geq3$ and the quantity $(u_\beta,h_\beta)$ with $|\beta|=m$ defined by \eqref{new}, we have%with $|\alpha|=m$, 
	\begin{equation}\label{equi}
		M(t)^{-1}\|\p_\tau^\beta (u, h)(t)\|_{L_l^2(\Omega)}~\leq~\|(u_\beta,h_\beta)(t)\|_{L_l^2(\Omega)}~\leq~M(t)\|\p_\tau^\beta(u, h)(t)\|_{L_l^2(\Omega)},
	\end{equation}
	and
	\begin{equation}
		\label{equi_y1}
		\big\|\p_y\p_\tau^\beta (u, h)(t)\big\|_{L_l^2(\Omega)}
	\leq\|\p_y (u_\beta,  h_\beta)(t)\|_{L_l^2(\Omega)}+M(t)\|h_\beta(t)\|_{L_l^2(\Omega)},
	\end{equation}
	%and\begin{equation}\label{equi1}
		%M(t)^{-1}\|\p_\tau^\beta (u,b)\|_{L_L^2(\Omega)}~\lesssim~\|(\hat u_\alpha,\hat b_\alpha)\|_{L_L^2(\Omega)}~\lesssim~M(t)\|\p_\tau^\beta(u,b)\|_{L_L^2(\Omega)},\end{equation}
	where
	\begin{equation}\label{def_M}
M(t)~:=~2\delta_0^{-1}\Big(C\|(U,H)(t)\|_{L^\infty(\bT_x)}+\big\|\ly^{l+1}\p_y(u, h)(t)\big\|_{L^\infty(\Omega)}+\big\|\ly^{l+1}\p_y^2(u, h)(t)\big\|_{L^\infty(\Omega)}\Big).%C\delta_0^{-1}\big(M_0+\|(u, h)(t)\|_{\h_l^3}\big).
\end{equation}
\iffalse
	Moreover, it holds
	\begin{equation}\label{equi_y}\begin{split}
\big\|\p_y (u_\beta,  h_\beta)(t)\big\|_{L_l^2(\Omega)}
\leq~&\big\|\p_y\p_\tau^\beta (u, h)(t)\big\|_{L_l^2(\Omega)}+C\delta_0^{-2}\big(M_0+\|(u, h)(t)\|_{\h_l^4}\big)^2\|\p_\tau^\beta h(t)\|_{L_l^2(\Omega)},
\end{split}\end{equation}and
	\begin{equation}
		\label{equi_y1}
		\big\|\p_y\p_\tau^\beta (u, h)(t)\big\|_{L_l^2(\Omega)}
	\leq\|\p_y (u_\beta,  h_\beta)(t)\|_{L_l^2(\Omega)}+\big(M(t)+2\delta_0^{-1}\|\ly^{l+1}\p_y^2(u, h)(t)\|_{L^\infty(\Omega)}\big)\|h_\beta(t)\|_{L_l^2(\Omega)}.%+C\delta_0^{-1}\big(M_0+\|(u, h)(t)\|_{\h_l^4}\big)\|h_\beta(t)\|_{L_l^2(\Omega)}.\end{equation}
	\fi
\end{lem} 
\begin{proof}[\textbf{Proof.}]
%Firstly, consider that $\p_y\psi=b$ and $\psi|_{y=0}=0$ from \eqref{psi}, it follows
%\[\|\p_\tau^\beta\psi(t,x,y)\|_{L_x^2L_y^\infty(\Omega)}=\Big\|\int_0^y\p_\tau^\beta b(t,x,z)dz\Big\|_{L_x^2L_y^\infty(\Omega)}\lesssim\|\p_\tau^\beta b(t)\|_{L_l^2(\Omega)}%\quad l>\frac{1}{2},\]provided $l>\frac{1}{2}$.
Firstly, from the definitions of $u_\beta$ and $h_\beta$ in \eqref{new}, we have by using \eqref{est_psi}, %\eqref{est_eta} and the Sobolev embedding inequality it yields that
\begin{equation*}
	\label{equivalent}\begin{split}
		\|u_\beta(t)\|_{L^2_l(\Omega)}\leq~&\|\p_\tau^\beta u(t)\|_{L^2_l(\Omega)}+
		\|\ly^{l+1}\eta_1(t)\|_{L^\infty(\Omega)}\|\ly^{-1}\p_\tau^\beta\psi(t)\|_{L^2(\Omega)}\\
		\leq~&\|\p_\tau^\beta u(t)\|_{L^2_l(\Omega)}+2\delta_0^{-1}\big(C\|U(t)\|_{L^\infty(\bT_x)}+\|\ly^{l+1}\p_yu(t)\|_{L^\infty(\Omega)}\big)\|\p_\tau^\beta h(t)\|_{L^2(\Omega)},
		%\leq~&\|\p_\tau^\beta u(t)\|_{L^2_l(\Omega)}+C\delta_0^{-1}\big(M_0+\|u(t)\|_{\h_l^{3}}\big)\|\p_\tau^\beta h(t)\|_{L^2(\Omega)},
		\end{split}\end{equation*}
		and
		\begin{align*}
			\|h_\beta(t)\|_{L^2_l(\Omega)}\leq~&\|\p_\tau^\beta h(t)\|_{L^2_l(\Omega)}+
		\|\ly^{l+1}\eta_2(t)\|_{L^\infty(\Omega)}\|\ly^{-1}\p_\tau^\beta\psi(t)\|_{L^2(\Omega)}\\
		\leq~&2\delta_0^{-1}\big(C\|H(t)\|_{L^\infty(\bT_x)}+\|\ly^{l+1}\p_yh(t)\|_{L^\infty(\Omega)}\big)\|\p_\tau^\beta h(t)\|_{L_l^2(\Omega)}.
		%\leq~&\|\p_\tau^\beta h(t)\|_{L^2_l(\Omega)}+C\delta_0^{-1}\big(M_0+\|h(t)\|_{\h_l^{3}}\big)\|\p_\tau^\beta h(t)\|_{L^2(\Omega)}.
		\end{align*}
Thus, %set $M(t)=1+\|\p_yu_s\|_{L^\infty(0,t;L_{l}^2(\bR_+))}+\|(u,b)\|_{L^\infty(0,t;H_l^{1,1}(\Omega))},$ 
we have that by \eqref{def_M},
\begin{equation}
	\label{equ_1}\begin{split}
	\|(u_\beta,h_\beta)(t)\|_{L_l^2(\Omega)}~&\leq~
	%\big(1+\|\p_yu_s(t)\|_{L_{y,l}^2(\bR_+)}+\|(u,b)(t)\|_{H_l^{1,1}(\Omega)}\big)
	M(t)\big\|\p_\tau^\beta (u, h)(t)\big\|_{L_l^2(\Omega)}.%\quad M(t)\triangleq 2\delta_0^{-1}\big(CM_0+\big\|\ly^{l+1}(\p_yu, \p_yh)(t)\big\|_{L^\infty(\Omega)}\big).
	 %\|(\p_y\hat u_m,\p_y\hat b_m)\|_{L_L^2(\Omega)}~&\leq~\|\p_y\p_\tau^m(u,b)\|_{L_L^2(\Omega)}+CP\big(E_2(t)\big)\|\p_\tau^m b\|_{L_L^2(\Omega)}.
	 %\quad\|(\hat u_\alpha,\hat b_\alpha)\|_{L_L^2(\Omega)}~\lesssim~M(t)\|\p_\tau^\beta(u,b)\|_{L_L^2(\Omega)}.
	\end{split}\end{equation}
On other hand, note that from $\p_y\psi=h$ and the expression of $h_\beta$ in \eqref{new},
\[h_\beta~=~\p_\tau^\beta h-\frac{\p_yh+H\phi''}{h+H\phi'}\p_\tau^\beta\psi~=~(h+H\phi')\cdot\p_y\Big(\frac{\p_\tau^\beta\psi}{h+H\phi'}\Big),\]
which implies that by  $\p_\tau^\beta\psi|_{y=0}=0,$
\begin{equation}
	\label{def_psi}
	\p_\tau^\beta\psi(t,x,y)=\big(h(t,x,y)+H(t,x)\phi'(y)\big)\cdot\int_0^y\frac{h_\beta(t,x,z)}{h(t,x,z)+H(t,x)\phi'(z)}dz.
\end{equation}
Therefore, combining the definition \eqref{new} for 
$(u_\beta, h_\beta)$ with \eqref{def_psi}, we have
\begin{equation}\label{for_m}
\begin{cases}
	\p_\tau^\beta u(t,x,y)=u_\beta(t,x,y)+\big(\p_yu(t,x,y)+U(t,x)\phi''(y)\big)\cdot\int_0^y\frac{h_\beta(t,x,z)}{h(t,x,z)+H(t,x)\phi'(z)}dz,\\
	\p_\tau^\beta h(t,x,y)=h_\beta(t,x,y)+\big(\p_yh(t,x,y)+H(t,x)\phi''(y)\big)\cdot\int_0^y\frac{h_\beta(t,x,z)}{h(t,x,z)+H(t,x)\phi'(z)}dz.
	\end{cases}\end{equation}
Then, by using \eqref{normal1},
\[\begin{split}
	\|\p_\tau^\beta u(t)\|_{L^2_l(\Omega)}
	\leq~&\|u_\beta(t)\|_{L_l^2(\Omega)}+\big\|\ly^{l+1}\big(\p_yu+U\phi'')(t)\big\|_{L^\infty(\Omega)}\Big\|\frac{1}{1+y}\int_0^y\frac{h_\beta(t,x,z)}{h(t,x,z)+H(t,x)\phi'(z)}dz\Big\|_{L^2(\Omega)}\\
	\leq~&\|u_\beta(t)\|_{L_l^2(\Omega)}+2\big(C\|U(t)\|_{L^\infty(\bT_x)}+\|\ly^{l+1}\p_yu(t)\|_{L^\infty(\Omega)}\big)\Big\|\frac{h_\beta}{h+H\phi'}\Big\|_{L^2(\Omega)}\\
	\leq~&\|u_\beta(t)\|_{L_l^2(\Omega)}+2\delta_0^{-1}\big(C\|U(t)\|_{L^\infty(\bT_x)}+\|\ly^{l+1}\p_yu(t)\|_{L^\infty(\Omega)}\big)\|h_\beta(t)\|_{L^2(\Omega)},
	\end{split}\]
	and similarly,
	\[\begin{split}
	\|\p_\tau^\beta h(t)\|_{L^2_l(\Omega)}
	%\leq~&\|h_\beta(t)\|_{L_l^2(\Omega)}+\big\|\ly^{l+1}\big(\p_yh+H\phi'')(t)\big\|_{L^\infty(\Omega)}\cdot\Big\|\frac{1}{1+y}\int_0^y\frac{h_\beta(t,x,z)}{h(t,x,z)+H(t,x)\phi'(z)}dz\Big\|_{L^2(\Omega)}\\
	%\leq~&\|h_\beta(t)\|_{L_l^2(\Omega)}+C\big(M_0+\|h(t)\|_{\h_l^3}\big)\Big\|\frac{h_\beta}{h+H\phi'}\Big\|_{L^2(\Omega)}\\
	\leq~&2\delta_0^{-1}\big(C\|H(t)\|_{L^\infty(\bT_x)}+\|\ly^{l+1}\p_yh(t)\|_{L^\infty(\Omega)}\big)\|h_\beta(t)\|_{L^2_l(\Omega)},
	\end{split}\]
	%Similarly, we have\[\begin{split}\|\p_\tau^\beta b(t)\|_{L^2_l(\Omega)}
	%\leq&\|\hat b_\alpha(t)\|_{L_l^2(\Omega)}+\|\p_yb(t)\|_{L_x^\infty L^2_{y,l}(\Omega)}\cdot\Big\|\int_0^y\frac{\hat b_\alpha(t,x,z)}{1+b(t,x,z)}dz\Big\|_{L^2_xL^\infty_y(\Omega)}\\
	%\lesssim&\Big(1+\|b(t)\|_{H_l^{1,1}(\Omega)}\Big)\|\hat b_\alpha(t)\|_{L_l^2(\Omega)},\quad \|\p_\tau^\beta b\|_{L_L^2(\Omega)}\lesssim\Big(1+\|b\|_{L^\infty(0,t;H_l^{1,1}(\Omega))}\Big)\|\hat b_\alpha\|_{L_L^2(\Omega)},\end{split}\]
which implies that, 
\begin{equation}
	\label{equ_2}
	\|\p_\tau^\beta (u, h)(t)\|_{L_l^2(\Omega)}~\leq~%\Big(1+\|\p_yu_s(t)\|_{L_{y,l}^2(\bR_+)}+\|(u,b)(t)\|_{H_l^{1,1}(\Omega)}\Big)
	M(t)\|(u_\beta, h_\beta)(t)\|_{L_l^2(\Omega)},%\quad \|\p_\tau^\beta(u,b)\|_{L_L^2(\Omega)}~\lesssim~M(t)\|(\hat u_\alpha, \hat b_\alpha)\|_{L_L^2(\Omega)}.
\end{equation}
provided that $M(t)$ is given in \eqref{def_M}.
Thus, combining \eqref{equ_1} with \eqref{equ_2} yields \eqref{equi}.

Furthermore,  by taking the derivation of \eqref{for_m} in $y$, we get the following forms of $\p_y\p_\tau^\beta u$ and $\p_y\p_\tau^\beta h$:% (which also can be derived by combining \eqref{def_psi}, \eqref{for_m} and \eqref{new_y}):
\begin{align*}\begin{cases}
	\p_y\p_\tau^\beta u(t,x,y)=&\p_y u_\beta(t,x,y)+\eta_1(t,x,y) h_\beta(t,x,y)+\big(\p_y^2u(t,x,y)+U(t,x)\phi^{(3)}(y)\big)\cdot\int_0^y\frac{h_\beta(t,x,z)}{h(t,x,z)+H(t,x)\phi'(z)}dz,\\
	\p_y\p_\tau^\beta h(t,x,y)=&\p_y h_\beta(t,x,y)+\eta_2(t,x,y) h_\beta(t,x,y)+\big(\p_y^2h(t,x,y)+H(t,x)\phi^{(3)}(y)\big)\cdot\int_0^y\frac{h_\beta(t,x,z)}{h(t,x,z)+H(t,x)\phi'(z)}dz.
	\end{cases}\end{align*}
Then, it follows that by \eqref{normal1} and \eqref{est_eta},
\begin{align*}
	\|\p_y\p_\tau^\beta u(t)\|_{L_l^2(\Omega)}
	\leq~&\|\p_y u_\beta(t)\|_{L_l^2(\Omega)}+\|\eta_1(t)\|_{L^\infty(\Omega)}\|h_\beta(t)\|_{L_l^2(\Omega)}\\
	&+\big\|\ly^{l+1}(\p_y^2u+U\phi^{(3)})(t)\big\|_{L^\infty(\Omega)}\Big\|\frac{1}{1+y}\int_0^y\frac{h_\beta(t,x,z)}{h(t,x,z)+H(t,x)\phi'(z)}dz\Big\|_{L^2(\Omega)}\\
	\leq~&\|\p_yu_\beta(t)\|_{L_l^2(\Omega)}+\delta_0^{-1}\big(C\|U(t)\|_{L^\infty(\bT_x)}+\|\p_yu(t)\|_{L^\infty(\Omega)}\big)\|h_\beta(t)\|_{L_l^2(\Omega)}\\
	&+2\big(C\|U(t)\|_{L^\infty(\bT_x)}+\|\ly^{l+1}\p_y^2u(t)\|_{L^\infty(\Omega)}\big)\Big\|\frac{h_\beta}{h+H\phi'}\Big\|_{L^2(\Omega)}\\
	\leq~&\|\p_yu_\beta(t)\|_{L_l^2(\Omega)}+%2\delta_0^{-1}\big(C\|U(t)\|_{L^\infty(\bT_x)}+\|\p_yu(t)\|_{L^\infty(\Omega)}+\|\ly^{l+1}\p_y^2u(t)\|_{L^\infty(\Omega)}\big)
	M(t)\|h_\beta(t)\|_{L_l^2(\Omega)},
\end{align*}
and similarly,
\begin{align*}
	\|\p_y\p_\tau^\beta h(t)\|_{L_l^2(\Omega)}
	%\\\leq~&\|\p_y h_\beta(t)\|_{L_l^2(\Omega)}+\|\eta_2(t)\|_{L^\infty(\Omega)}\|h_\beta(t)\|_{L_l^2(\Omega)}\\&+\big\|\ly^{l+1}(\p_y^2h+H\phi^{(3)})(t)\big\|_{L^\infty(\Omega)}\Big\|\frac{1}{1+y}\int_0^y\frac{h_\beta(t,x,z)}{h(t,x,z)+H(t,x)\phi'(z)}dz\Big\|_{L^2(\Omega)}\\
	%\leq~&\|\p_yu_\beta(t)\|_{L_l^2(\Omega)}+C\delta_0^{-1}\big(M_0+\|h(t)\|_{\h_0^3}\big)\|h_\beta(t)\|_{L_l^2(\Omega)}+C\big(M_0+\|h(t)\|_{\h_l^4}\big)\Big\|\frac{h_\beta}{h+H\phi'}\Big\|_{L^2(\Omega)}\\
	\leq~&\|\p_y h_\beta(t)\|_{L_l^2(\Omega)}+%2\delta_0^{-1}\big(C\|H(t)\|_{L^\infty(\bT_x)}+\|\p_y h(t)\|_{L^\infty(\Omega)}+\|\ly^{l+1}\p_y^2 h(t)\|_{L^\infty(\Omega)}\big)
	M(t)\|h_\beta(t)\|_{L_l^2(\Omega)}.
\end{align*}
Combining the above two inequalities yields that by  \eqref{def_M},
\begin{align*}
	\|\p_y\p_\tau^\beta (u, h)(t)\|_{L_l^2(\Omega)}
	\leq~&\|\p_y (u_\beta, h_\beta)(t)\|_{L_l^2(\Omega)}+M(t)\|h_\beta(t)\|_{L_l^2(\Omega)}.%C\delta_0^{-1}\big(M_0+\|(u,h)(t)\|_{\h_l^4}\big)\|h_\beta(t)\|_{L_l^2(\Omega)},
\end{align*}
Thus we obtain \eqref{equi_y1} and this completes the proof.
\end{proof}

\subsection{Closeness of the a priori estimates}%Proposition \ref{prop_priori}}
%\subsection{Equivalent Norms in Sobolev Spaces}
\indent\newline
In this subsection, we will prove Proposition \ref{prop_priori}. Before that, we need some preliminaries. First of all, as we know that from \eqref{ass_h},
\begin{align*}
	\big\|\ly^{l+1}\p_y^i(u, h)(t)\big\|_{L^\infty(\Omega)}\leq \delta_0^{-1},\quad\mbox{for}\quad i=1,2,\quad t\in[0,T],
\end{align*}	
combining with the definitions \eqref{def_eta} for $\eta_i, i=1,2$, and \eqref{def_M} for $M(t)$, it implies that for $\delta_0$ sufficiently small,
 \begin{align}\label{bound_eta}
 \|\ly^{l+1}\eta_i(t)\|_{L^\infty(\Omega)}\leq2\delta_0^{-2},\quad M(t)\leq2\delta_0^{-1}\big(C\|(U,H)(t)\|_{L^\infty(\bT_x)}+2\delta_0^{-1}\big)\leq 5\delta^{-2}_0,\quad i=1,2.%+C\delta_0^{-1}\|(U,H)(t)\|_{L^\infty(\bT_x)},\quad t\in[0,T].
 \end{align} 	
Then, recall that $D^\alpha=\p_\tau^\beta\p_y^k$, we obtain that by \eqref{equi} and \eqref{equi_y1} given in Lemma \ref{lem_equ},
\begin{align*}
	\|(u, h)(t)\|_{\h_l^m}^2=&\sum_{\substack{|\alpha|\leq m\\|\beta|\leq m-1}}\|D^\alpha(u, h)(t)\|_{L_l^2(\Omega)}^2+\sum_{|\beta|=m}\|\p_\tau^\beta(u, h)(t)\|_{L_l^2(\Omega)}^2\nonumber\\
	\leq&\sum_{\substack{|\alpha|\leq m\\|\beta|\leq m-1}}\|D^\alpha(u, h)(t)\|_{L_l^2(\Omega)}^2+25\delta_0^{-4}\sum_{|\beta|=m}\|(u_\beta,h_\beta)(t)\|_{L_l^2(\Omega)}^2,%+C\sum_{|\beta|=m}\|(u_\beta,h_\beta)(t)\|_{L_l^2(\Omega)}^2,
\end{align*}
%where $M(t)$ is given in \eqref{def_M}, 
and
\begin{align*}
	\|\p_y (u, h)(t)\|_{\h_l^m}^2=&\sum_{\substack{|\alpha|\leq m\\|\beta|\leq m-1}}\| D^\alpha\p_y (u, h)(t)\|_{L_l^2(\Omega)}^2+\sum_{|\beta|=m}\|\p_y\p_\tau^\beta (u, h)(t)\|_{L_l^2(\Omega)}^2\nonumber\\
	\leq&\sum_{\substack{|\alpha|\leq m\\|\beta|\leq m-1}}\| D^\alpha \p_y(u, h)(t)\|_{L_l^2(\Omega)}^2+2\sum_{|\beta|=m}\|\p_y (u_\beta, h_\beta)(t)\|_{L_l^2(\Omega)}^2+50\delta_0^{-4}\sum_{|\beta|=m}\|h_\beta (t)\|_{L^2_l(\Omega)}^2.
\end{align*}
Consequently,  we have the following 
\begin{cor}
	\label{cor_equi}
	Under the assumptions of Proposition \ref{prop_priori}, for any $t\in[0,T]$ and the quantity $(u_\beta, h_\beta), |\beta|=m$ given by \eqref{new}, it holds that
	\begin{align}\label{equ_m0}
	\|(u, h)(t)\|_{\h_l^m}^2
	\leq&\sum_{\substack{|\alpha|\leq m\\|\beta|\leq m-1}}\|D^\alpha(u, h)(t)\|_{L_l^2(\Omega)}^2+25\delta_0^{-4}\sum_{|\beta|=m}\|(u_\beta, h_\beta)(t)\|_{L_l^2(\Omega)}^2,
	%\|(u,b)\|_{\A_l^m(\Omega)}^2\lesssim&\sum_{\substack{|\alpha|\leq m\\|\beta|\leq m-1}}\|D^\alpha(u,b)\|_{L_L^2(\Omega)}^2+M(t)^2\sum_{|\alpha|=m}\|(\hat u_\alpha,\hat b_\alpha)\|_{L_L^2(\Omega)}^2,\label{equ_m}
	%\|(u, h)(t)\|_{\h_l^m(\Omega)}^2\lesssim&\sum_{\substack{|\alpha|\leq m\\|\beta|\leq m-1}}\|D^\alpha(u,b)(t)\|_{L_l^2(\Omega)}^2+M(t)^2\sum_{|\alpha|=m}\|(\hat u_\alpha,\hat b_\alpha)(t)\|_{L_l^2(\Omega)}^2,\label{equ_m0}\\
	%\|(u,b)\|_{\A_l^m(\Omega)}^2\lesssim&\sum_{\substack{|\alpha|\leq m\\|\beta|\leq m-1}}\|D^\alpha(u,b)\|_{L_L^2(\Omega)}^2+M(t)^2\sum_{|\alpha|=m}\|(\hat u_\alpha,\hat b_\alpha)\|_{L_L^2(\Omega)}^2,\label{equ_m}
\end{align}
%where $M(t)$ is given in \eqref{def_M}, 
and
\begin{align}
	\label{equ_ym}
	\|\p_y(u, h)(t)\|_{\h_l^m}^2	\leq&\sum_{\substack{|\alpha|\leq m\\|\beta|\leq m-1}}\| D^\alpha\p_y(u, h)(t)\|_{L_l^2(\Omega)}^2+2\sum_{|\beta|=m}\|\p_y(u_\beta, h_\beta)(t)\|_{L_l^2(\Omega)}^2+50\delta_0^{-4}\sum_{|\beta|=m}\|h_\beta (t)\|_{L^2_l(\Omega)}^2.%\nonumber\\&+C\delta_0^{-2}\big(M_0+\|(u, h)(t)\|_{\h_l^4}\big)^2\sum_{|\beta|=m}\|h_\beta (t)\|_{L^2_l(\Omega)}^2.
	%\|\p_y(u,b)\|_{\A_l^m(\Omega)}^2	\lesssim&\sum_{\substack{|\alpha|\leq m\\|\beta|\leq m-1}}\|\p_y D^\alpha(u,b)\|_{L_L^2(\Omega)}^2+\sum_{|\alpha|=m}\|\p_y(\hat u_\alpha,\hat b_\alpha)\|_{L_L^2(\Omega)}^2+P\big(E_3(t)\big)\sum_{|\alpha|=m}\|\p_\tau^\beta b\|_{L^2_l(\Omega)}.
\end{align}
\end{cor}

%The goal of this subsection is to derive the weighted Sobolev estimates for $(u, h).$ 
Now, we can derive the desired a priori estimates of $(u,h)$ for the problem \eqref{bl_main}. From Proposition \ref{prop_estm} and \ref{prop_xm}, it follows that for $m\geq5$ and any $t\in[0,T],$
\begin{align}\label{est_all}
	&\frac{d}{dt}\Big(\sum_{\substack{|\alpha|\leq m\\|\beta|\leq m-1}}\big\|D^\alpha(u, h)(t)\big\|^2_{L^2_l(\mathbb{R}^2_+)}+25\delta_0^{-4}\sum_{|\beta|=m}\big\|(u_\beta, h_\beta)(t)\big\|_{L_l^2(\Omega)}^2\Big)\nonumber\\
	&+\mu\Big(\sum_{\substack{|\alpha|\leq m\\|\beta|\leq m-1}}\big\|D^\alpha\p_y u(t)\big\|^2_{L^2_l(\mathbb{R}^2_+)}+25\delta_0^{-4}\sum_{|\beta|=m}\big\|\p_y u_\beta(t)\big\|_{L_l^2(\Omega)}^2\Big)\nonumber\\
	&+\kappa\Big(\sum_{\substack{|\alpha|\leq m\\|\beta|\leq m-1}}\big\|D^\alpha\p_y h(t)\big\|^2_{L^2_l(\mathbb{R}^2_+)}+25\delta_0^{-4}\sum_{|\beta|=m}\big\|\p_y h_\beta(t)\big\|_{L_l^2(\Omega)}^2\Big)\nonumber\\
	%-\|D^\alpha(u, b)(0)\|^2_{L^2_l(\mathbb{R}^2_+)}&+\mu\big\|\p_y D^\alpha(u, h)(t)\big\|^2_{L^2_l(\Omega)}\Big)+\sum_{|\beta|=m}\Big(\frac{d}{dt}\big\|(u_\beta, h_\beta)(t)\big\|_{L_l^2(\Omega)}^2+\big\|\p_y(\hat u_\alpha,\hat b_\alpha)\big\|_{L_L^2(\Omega)}^2\Big)\nonumber\\
	\leq~&%\sum_{\substack{|\alpha|\leq m\\|\beta|\leq m-1}}\big\|D^\alpha(u, b)(0)\big\|^2_{L^2_l(\bR^2_+)}+\sum_{|\alpha|=m}\big\|(\hat u_{\alpha0},\hat b_{\alpha0})\big\|_{L_l^2(\Omega)}^2+
	\delta_1C\|\p_y(u, h)(t)\|_{\h_0^m}^2+C\delta_1^{-1}\|(u, h)(t)\|^2_{\h_l^m}\big(1+\|(u, h)(t)\|^2_{\h_l^m}\big)+\sum_{\tiny\substack{|\alpha|\leq m\\|\beta|\leq m-1}}\|D^\alpha(r_1, r_2)(t)\|_{L^2_{l+k}(\Omega)}^2\nonumber\\
	&+C\sum_{|\beta|\leq m+2}\|\p_\tau^\beta (U,H,P)(t)\|_{L^2(\bT_x)}^2+25\delta_0^{-4}\sum_{|\beta|=m}\Big(\|\p_\tau^\beta r_1-\eta_1\p_\tau^\beta r_3\|_{L_l^2(\Omega)}^2+\|\p_\tau^\beta r_2-\eta_2\p_\tau^\beta r_3\|_{L_l^2(\Omega)}^2\Big)\nonumber\\
	&+C\delta_0^{-6}\big(\sum_{|\beta|\leq m+2}\|\p_\tau^{\beta}(U,H)(t)\|_{L^2(\bT_x)}+\|(u,h)(t)\|_{\h_l^m}\big)^2\Big(\sum_{|\beta|=m}\|(u_\beta, h_\beta)(t)\|_{L^2_l(\Omega)}^2\Big)\nonumber\\
	&+C\delta_0^{-8}\big(\sum_{|\beta|\leq m+2}\|\p_\tau^{\beta}(U,H)(t)\|_{L^2(\bT_x)}+\|(u,h)\|_{\h_{l}^m}\big)^4\|(u ,h)(t)\|_{\h_l^m}^2.
	%+C\delta_0^{-8}\big(M_0^2+\|(u, h)(t)\|_{\h_l^m}^2\big)^3\nonumber\\
	%&+C\delta_0^{-6}\big(M_0^2+\|(u, h)(t)\|_{\h_l^m}^2\big)\sum_{|\beta|=m}\|(u_\beta, h_\beta)(t)\|_{L^2_l(\Omega)}^2.
	%+P\big(E_3(t)\big)\sum_{|\alpha|=m}\big\|(\hat u_\alpha, \hat b_\alpha)\big\|_{L^2_l(\Omega)}^2\nonumber\\&+\delta_1^{-1}P\big(E_3(t)\big)%\|(u\|_{\B_L^2(\Omega)}\Big(\|(u,b)\|^2_{\A_l^m(\Omega)}.%+\|D^{m+1}u_s\|_{L^2_l(D_t)}^2\Big).
\end{align}
Plugging the inequalities \eqref{equ_m0} and \eqref{equ_ym} given in Corollary \ref{cor_equi} into \eqref{est_all}, and choosing $\delta_1$ small enough, we get
\begin{align}\label{est_all1}
&\frac{d}{dt}\Big(\sum_{\substack{|\alpha|\leq m\\|\beta|\leq m-1}}\big\|D^\alpha(u, h)(t)\big\|^2_{L^2_l(\mathbb{R}^2_+)}+25\delta_0^{-4}\sum_{|\beta|=m}\big\|(u_\beta, h_\beta)(t)\big\|_{L_l^2(\Omega)}^2\Big)\nonumber\\
	&+\Big(\sum_{\substack{|\alpha|\leq m\\|\beta|\leq m-1}}\big\|D^\alpha\p_y (u, h)(t)\big\|^2_{L^2_l(\mathbb{R}^2_+)}+25\delta_0^{-4}\sum_{|\beta|=m}\big\|\p_y (u_\beta, h_\beta)(t)\big\|_{L_l^2(\Omega)}^2\Big)\nonumber\\
	\leq~&C\sum_{\tiny\substack{|\alpha|\leq m\\|\beta|\leq m-1}}\|D^\alpha(r_1, r_2)(t)\|_{L^2_{l+k}(\Omega)}^2+C\delta_0^{-4}\sum_{|\beta|=m}\Big(\|\p_\tau^\beta (r_1, r_2)(t)\|_{L_l^2(\Omega)}^2+4\delta_0^{-4}\|\p_\tau^\beta r_3\|_{L_{-1}^2(\Omega)}^2\Big)\nonumber\\
	&+C\delta_0^{-8}\Big(1+\sum_{|\beta|\leq m+2}\|\p_\tau^\beta (U,H,P)(t)\|_{L^2(\bT_x)}^2\Big)^3\nonumber\\
	&+C\delta_0^{-8}\Big(\sum_{\substack{|\alpha|\leq m\\|\beta|\leq m-1}}\|D^\alpha(u, h)(t)\|_{L_l^2(\Omega)}^2+25\delta_0^{-4}\sum_{|\beta|=m}\|(u_\beta, h_\beta)(t)\|_{L_l^2(\Omega)}^2\Big)^3,
	%&\sum_{\substack{|\alpha|\leq m\\|\beta|\leq m-1}}\Big(\big\|D^\alpha(u, b)(t)\big\|^2_{L^2_l(\mathbb{R}^2_+)}
	%-\|D^\alpha(u, b)(0)\|^2_{L^2_l(\mathbb{R}^2_+)}
	%+\big\|\p_y D^\alpha(u, b)\big\|^2_{L^2_l(\Omega)}\Big)
	%+\sum_{|\alpha|=m}\Big(\big\|(\hat u_\alpha,\hat b_\alpha)(t)\big\|_{L_l^2(\Omega)}^2+\big\|\p_y(\hat u_\alpha,\hat b_\alpha)\big\|_{L_L^2(\Omega)}^2\Big)\nonumber\\
	%\leq&
	%\sum_{\substack{|\alpha|\leq m\\|\beta|\leq m-1}}\big\|D^\alpha(u, b)(0)\big\|^2_{L^2_l(\bR^2_+)}+\sum_{|\alpha|=m}\big\|(\hat u_{\alpha0},\hat b_{\alpha0})\big\|_{L_l^2(\Omega)}
	%P\big(\|(u_0,b_0)\|_{H_l^{2m}(\Omega)}\big)+P\big(E_3(t)\big)\|D^{m+1}u_s\|_{L_L^2(\bT_x)}^2\nonumber\\&
	%\tilde C_m\big\|(u_0,b_0)\big\|_{H_l^{2m}(\Omega)(\Omega)}^{2m+1}+P\big(E_3(t)\big)\Big(\sum_{\substack{|\alpha|\leq m\\|\beta|\leq m-1}}\|D^\alpha(u,b)\|_{L_L^2(\Omega)}^2+\sum_{|\alpha|=m}\big\|(\hat u_\alpha, \hat b_\alpha)\big\|_{L^2_l(\Omega)}^2\Big).
	%\nonumber\\&+\delta_0^{-1}P\big(E_3(t)\big)%\|(u\|_{\B_L^2(\Omega)}\Big(\|(u,b)\|^2_{\A_l^m(\Omega)}+\|D^{m+1}u_s\|_{L^2_l(D_t)}^2\Big).
\end{align}
where  we have used the fact that
\[\|\eta_i\p_\tau^\beta r_3\|_{L_l^2(\Omega)}\leq\|\ly^{i+1}\eta_i(t)\|_{L^\infty(\Omega)}\|\ly^{-1}\p_\tau^\beta r_3\|_{L^2(\Omega)}\leq 2\delta_0^{-2}\|\p_\tau^\beta r_3\|_{L_{-1}^2(\Omega)},\quad i=1,2.\]
%provided \eqref{bound_eta}.
Denote by
\begin{align}\label{F_def}
	F_0~:=~\sum_{\substack{|\alpha|\leq m\\|\beta|\leq m-1}}\|D^\alpha(u, h)(0)\|_{L_l^2(\Omega)}^2+25\delta_0^{-4}\sum_{|\beta|=m}\big\|(u_{\beta0}, h_{\beta0})\big\|_{L^2_l(\Omega)}^2,
\end{align}
and
\begin{align}\label{Ft_def}
	F(t)~:=~&C\sum_{\tiny\substack{|\alpha|\leq m\\|\beta|\leq m-1}}\|D^\alpha(r_1, r_2)(t)\|_{L^2_{l+k}(\Omega)}^2+C\delta_0^{-4}\sum_{|\beta|=m}\Big(\|\p_\tau^\beta (r_1, r_2)(t)\|_{L_l^2(\Omega)}^2+4\delta_0^{-4}\|\p_\tau^\beta r_3\|_{L_{-1}^2(\Omega)}^2\Big)\nonumber\\
	&+C\delta_0^{-8}\Big(1+\sum_{|\beta|\leq m+2}\|\p_\tau^\beta (U,H,P)(t)\|_{L^2(\bT_x)}^2\Big)^3.
\end{align}
%Then, by direct calculation, we know that $F_0$ is a polynomial of $\big\|\big(u_{0}, h_{0}\big)\big\|_{H_l^{2m}(\Omega)}$,  and satisfies that by combining with \eqref{est_equib},
%\begin{align}\label{def_F}
	%F_0~\leq~\delta_0^{-6}~\cp\big(M_0+\|(u_0,h_0)\|_{H_l^{2m}(\Omega)}\big).%C\delta_0^{-6}\big(M_0+\|(u_0,h_0)\|_{\H_l^{2m}(\Omega)}\big)^{2^{m+1}},
%\end{align}
By the comparison principle of ordinary differential equations in \eqref{est_all1}, it yields that
\begin{align}
	\label{est_fin0}
	&\sum_{\substack{|\alpha|\leq m\\|\beta|\leq m-1}}\|D^\alpha(u, h)(t)\|_{L_l^2(\Omega)}^2+25\delta_0^{-4}\sum_{|\beta|=m}\big\|(u_\beta, h_\beta)(t)\big\|_{L^2_l(\Omega)}^2\nonumber\\
	&+\int_0^t\Big(\sum_{\substack{|\alpha|\leq m\\|\beta|\leq m-1}}\big\|D^\alpha\p_y (u, h)(s)\big\|^2_{L^2_l(\Omega)}+25\delta_0^{-4}\sum_{|\beta|=m}\big\|\p_y (u_\beta, h_\beta)(s)\big\|_{L_l^2(\Omega)}^2\Big)ds\nonumber\\
	\leq~&\big(F_0+\int_0^tF(s)ds\big)\cdot\Big\{1-2C\delta_0^{-8}\big(F_0+\int_0^tF(s)ds\big)^2t\Big\}^{-\frac{1}{2}}.%-M_0^2.\nonumber\\\leq~&F_0^2\Big[1-2C\delta_0^{-8}\big(M_0^2+F_0^2\big)^2t\Big]^{-\frac{1}{2}}.
\end{align}
\iffalse
which implies that by combining with \eqref{def_F}, 
\begin{align}
	\label{est_fin1}
	&\sum_{\substack{|\alpha|\leq m\\|\beta|\leq m-1}}\|D^\alpha(u, h)(t)\|_{L_l^2(\Omega)}^2+36\delta_0^{-4}\sum_{|\beta|=m}\big\|(u_\beta, h_\beta)(t)\big\|_{L^2_l(\Omega)}^2\nonumber\\
	&+\int_0^t\Big(\sum_{\substack{|\alpha|\leq m\\|\beta|\leq m-1}}\big\|D^\alpha\p_y (u, h)(s)\big\|^2_{L^2_l(\Omega)}+36\delta_0^{-4}\sum_{|\beta|=m}\big\|\p_y (u_\beta, h_\beta)(s)\big\|_{L_l^2(\Omega)}^2\Big)ds\nonumber\\
		\leq~& \delta_0^{-6}~\cp\big(M_0+\|(u_0,h_0)\|_{H_l^{2m}(\Omega)}\big)\cdot\Big[1-\delta_0^{-20}\cp\big(M_0+\|(u_0,h_0)\|_{H_l^{2m}(\Omega)}\big)t\Big]^{-\frac{1}{2}}.%C\delta_0^{-6}\Big(M_0+\|(u_0,h_0)\|_{\H_l^{2m}(\Omega)}\Big)^{2^{m+1}}\cdot\Big[1-C\delta_0^{-20}t\big(M_0+\|(u_0,h_0)\|_{\H_l^{2m}(\Omega)}\big)^{2^{m+2}}\Big]^{-\frac{1}{2}}.
	%\leq~&\tilde C_m\big\|(u_0,b_0)\big\|_{H_l^{2m}(\Omega)(\Omega)}^{2m+1}\Big(\exp\{P\big(E_3(T)\big)\cdot t\}-1\Big),
	%&\sum_{\substack{|\alpha|\leq m\\|\beta|\leq m-1}}\|D^\alpha(u, h)(t)\|_{L_L^2(\Omega)}^2+\sum_{|\alpha|=m}\big\|(\hat u_\alpha, \hat b_\alpha)\big\|_{L^2_l(\Omega)}^2
	%\nonumber\\\leq&\Big[P\big(\|(u_0,b_0)\|_{H_l^{2m}(\Omega)}\big)+P\big(E_3(T)\big)\|D^{m+1}u_s\|_{L_L^2(\bT_x)}^2\Big]
	%\leq\tilde C_m\big\|(u_0,b_0)\big\|_{H_l^{2m}(\Omega)(\Omega)}^{2m+1}\Big(\exp\{P\big(E_3(T)\big)\cdot t\}-1\Big),
\end{align}
\fi
%Now, we can give the proof of Proposition \ref{prop_priori}. 
%\begin{proof}[\textbf{Proof of Proposition \ref{prop_priori}}:]
Then, it implies that by combining \eqref{equ_m0} with \eqref{est_fin0},
\begin{align}
	\label{est_fin2}
	\sup_{0\leq s\leq t}\|(u, h)(s)\|_{\h_l^m}~\leq~\big(F_0+\int_0^tF(s)ds\big)^{\frac{1}{2}}\cdot\Big\{1-2C\delta_0^{-8}\big(F_0+\int_0^tF(s)ds\big)^2t\Big\}^{-\frac{1}{4}}.
	%\delta_0^{-3}\cp\big(M_0+\|(u_0,h_0)\|_{H_l^{2m}(\Omega)}\big)\cdot\Big[1-\delta_0^{-20}\cp\big(M_0+\|(u_0,h_0)\|_{H_l^{2m}(\Omega)}\big)t\Big]^{-\frac{1}{4}},%C\delta_0^{-3}\Big(M_0+\|(u_0,h_0)\|_{\H_l^{2m}(\Omega)}\Big)^{2^{m}}\cdot\Big[1-C\delta_0^{-20}t\big(M_0+\|(u_0,h_0)\|_{\H_l^{2m}(\Omega)}\big)^{2^{m+2}}\Big]^{-\frac{1}{4}}.
	%\sum_{\substack{|\alpha|\leq m\\|\beta|\leq m-1}}\|D^\alpha(u,b)\|_{L_L^2(\Omega)}^2+\sum_{|\alpha|=m}\big\|(\hat u_\alpha, \hat b_\alpha)\big\|_{L^2_l(\Omega)}^2\leq
	%&\sum_{\substack{|\alpha|\leq m\\|\beta|\leq m-1}}\Big(\big\|D^\alpha(u, b)(t)\big\|^2_{L^2_l(\mathbb{R}^2_+)}+\big\|\p_y D^\alpha(u, b)\big\|^2_{L^2_l(\Omega)}\Big)+\sum_{|\alpha|=m}\Big(\big\|(\hat u_\alpha,\hat b_\alpha)(t)\big\|_{L_l^2(\Omega)}^2+\big\|\p_y(\hat u_\alpha,\hat b_\alpha)\big\|_{L_L^2(\Omega)}^2\Big)\nonumber\\
	%\leq&%\Big[P\big(\|(u_0,b_0)\|_{H_l^{2m}(\Omega)}\big)+P\big(E_3(T)\big)\|D^{m+1}u_s\|_{L_L^2(\bT_x)}^2\Big]
	%\tilde C_m\big\|(u_0,b_0)\big\|_{H_l^{2m}(\Omega)(\Omega)}^{2m+1}\cdot \exp\{P\big(E_3(T)\big)\cdot t\}.%\qquad\forall t\in[0,T].
\end{align}
%and we get the estimate \eqref{est_priori}. Secondly, we can establish the estimates of \eqref{upbound_uy} and \eqref{h_lowbound}, that is,  the $L^\infty$ estimates for $\ly^{l+1}\p_y^i(u, h)(t), i=1,2,$ and the lower bound estimate for $h(t,x,y)$. Precisely, 

As we know
\begin{align*}
	&\ly^{l+1}\p_y^i(u, h)(t,x,y)=\ly^{l+1}\p_y^i(u_0, h_0)(x,y)+\int_0^t\ly^{l+1}\p_t\p_y^i(u, h)(s,x,y)ds,\quad i=1,2,\end{align*}
	and
	\begin{align*}
	&h(t,x,y)=h_0(x,y)+\int_0^t\p_th(s,x,y)ds.
\end{align*}
Then, by the Sobolev embedding inequality and \eqref{est_fin2} we have that for $i=1,2,$
\begin{align}\label{bound_uy}
	&\|\ly^{l+1}\p_y^i(u, h)(t)\|_{L^\infty(\Omega)}\nonumber\\
	\leq ~&\|\ly^{l+1}\p_y^i(u_0, h_0)\|_{L^\infty(\Omega)}+\int_0^t\|\ly^{l+1}\p_t\p_y^i(u, h)(s)\|_{L^\infty(\Omega)}ds\nonumber\\
	\leq~&\|\ly^{l+1}\p_y^i(u_0, h_0)\|_{L^\infty(\Omega)}+C\big(\sup_{0\leq s\leq t}\|(u, h)(s)\|_{\h_l^5}\big)\cdot t\nonumber\\
	\leq~&\|\ly^{l+1}\p_y^i(u_0, h_0)\|_{L^\infty(\Omega)}+C t\cdot\big(F_0+\int_0^tF(s)ds\big)^{\frac{1}{2}}\Big\{1-2C\delta_0^{-8}\big(F_0+\int_0^tF(s)ds\big)^2t\Big\}^{-\frac{1}{4}}.
	%\leq~&\|\ly^{l+1}\p_y^i(u_0, h_0)\|_{L^\infty(\Omega)}+\delta_0^{-3}t\cdot \cp\big(M_0+\|(u_0,h_0)\|_{H_l^{2m}(\Omega)}\big)\Big[1-\delta_0^{-20}\cp\big(M_0+\|(u_0,h_0)\|_{H_l^{2m}(\Omega)}\big)t\Big]^{-\frac{1}{4}}.
	\end{align}
Similarly, one can obtain that
\begin{align}\label{bound_h}
	h(t,x,y)\geq ~&h_0(x,y)-\int_0^t\|\p_t h(s)\|_{L^\infty(\Omega)}ds	\geq~h_0(x,y)-C\big(\sup_{0\leq s\leq t}\|h(s)\|_{\h_0^3}\big)\cdot t\nonumber\\
	\geq~&h_0(x,y)-C t\cdot\big(F_0+\int_0^tF(s)ds\big)^{\frac{1}{2}}\Big\{1-2C\delta_0^{-8}\big(F_0+\int_0^tF(s)ds\big)^2t\Big\}^{-\frac{1}{4}}.
	%\geq~&h_0(x,y)-\delta_0^{-3}t\cdot \cp\big(M_0+\|(u_0,h_0)\|_{H_l^{2m}(\Omega)}\big)\Big[1-\delta_0^{-20}\cp\big(M_0+\|(u_0,h_0)\|_{H_l^{2m}(\Omega)}\big)t\Big]^{-\frac{1}{4}}.
	\end{align}
	Therefore, we obtain the following  %which implies Proposition \ref{prop_priori} immediately by using \eqref{def_F}.
	\begin{prop}\label{prop-priori}
	Under the assumptions of Proposition \ref{prop_priori}, there exists a constant $C>0$, depending only on $m, M_0$ and $\phi$, such that 
	\begin{align}\label{est_priori-1}
	\sup_{0\leq s\leq t}\|(u, h)(s)\|_{\h_l^m}~\leq~\big(F_0+\int_0^tF(s)ds\big)^{\frac{1}{2}}\cdot\Big\{1-2C\delta_0^{-8}\big(F_0+\int_0^tF(s)ds\big)^2t\Big\}^{-\frac{1}{4}},%F_0\Big[1-2C\delta_0^{-8}\big(M_0^2+F_0^2\big)^2t\Big]^{-\frac{1}{4}},
\end{align}
for small time, where the quantities $F_0$ and $F(t)$ are defined by \eqref{F_def} and \eqref{Ft_def} respectively. Also, we have that for $i=1,2,$
\begin{align}\label{upbound_uy-1}
	&\|\ly^{l+1}\p_y^i(u, h)(t)\|_{L^\infty(\Omega)}\nonumber\\
	\leq~&\|\ly^{l+1}\p_y^i(u_0, h_0)\|_{L^\infty(\Omega)}+Ct\cdot\big(\sup_{0\leq s\leq t}\|(u, h)(s)\|_{\h_l^5}\big)\nonumber\\
	\leq~&\|\ly^{l+1}\p_y^i(u_0, h_0)\|_{L^\infty(\Omega)}+C t\cdot\big(F_0+\int_0^tF(s)ds\big)^{\frac{1}{2}}\Big\{1-2C\delta_0^{-8}\big(F_0+\int_0^tF(s)ds\big)^2t\Big\}^{-\frac{1}{4}},%CF_0 t\cdot\Big[1-2C\delta_0^{-8}\big(M_0^2+F_0^2\big)^2t\Big]^{-\frac{1}{4}},
	\end{align}
and 
\begin{align}\label{h_lowbound-1}
	h(t,x,y)\geq~&h_0(x,y)-C\big(\sup_{0\leq s\leq t}\|h(s)\|_{\h_0^3}\big)\cdot t\nonumber\\
\geq~&h_0(x,y)-C t\cdot\big(F_0+\int_0^tF(s)ds\big)^{\frac{1}{2}}\Big\{1-2C\delta_0^{-8}\big(F_0+\int_0^tF(s)ds\big)^2t\Big\}^{-\frac{1}{4}}.%CF_0 t\cdot\Big[1-2C\delta_0^{-8}\big(M_0^2+F_0^2\big)^2t\Big]^{-\frac{1}{4}}.
\end{align}
\end{prop}

From the above Proposition \ref{prop-priori}, we 
are ready to prove Proposition \ref{prop_priori}. Indeed, by using \eqref{ass_outflow}, \eqref{est_rhd} and the fact $\|\p_\tau^\beta r_3\|_{L_{-1}^2(\Omega)}\leq CM_0$ from the expression \eqref{r_3}, it follows that from the definition \eqref{Ft_def} for $F(t)$,
\begin{align}\label{est_Ft}
	F(t)~\leq~C\delta_0^{-8}M_0^6.
\end{align}
Next, by direct calculation we know that $D^\alpha (u, h)(0,x,y), |\alpha|\leq m$ can be expressed by  the spatial derivatives of initial data $(u_0, h_0)$ up to order $2m$. Then, combining with \eqref{ib_hat} %in other words, from the definition \eqref{F_def} of $F_0$ and by using \eqref{ib_hat}, 
we get that $F_0$, given by \eqref{F_def}, is a polynomial of $\big\|\big(u_{0}, h_{0}\big)\big\|_{H_l^{2m}(\Omega)}$,  and consequently %satisfies that by combining with \eqref{est_equib},
\begin{align}\label{def_F}
	F_0~\leq~\delta_0^{-8}~\cp\big(M_0+\|(u_0,h_0)\|_{H_l^{2m}(\Omega)}\big).%C\delta_0^{-6}\big(M_0+\|(u_0,h_0)\|_{\h_l^{2m}(\Omega)}\big)^{2^{m+1}},
\end{align}
Plugging \eqref{est_Ft} and \eqref{def_F} into \eqref{est_priori-1}-\eqref{h_lowbound-1}, we derive the estimates \eqref{est_priori}-\eqref{h_lowbound}, and 
then obtain the  proof of Proposition \ref{prop_priori}.

\section{Local-in-time existence and uniqueness}

In this section, we will establish the local-in-time existence and uniqueness of solutions to the nonlinear problem (\ref{bl_main}).

\subsection{Existence}
\indent\newline
For this, we consider a parabolic regularized system for problem \eqref{bl_main}, from which we can obtain the local (in time) existence of solution by using classical energy estimates. Precisely, for a small parameter $0<\epsilon<1,$ we investigate the following problem:
\begin{align}
\label{pr_app}
\left\{
\begin{array}{ll}
\partial_tu^\ep+\big[(u^\ep+U\phi')\partial_x+(v^\ep-U_x\phi)\partial_y\big]u^\ep-\big[(h^\ep+H\phi')\partial_x+(g^\ep-H_x\phi)\partial_y\big]h^\ep+U_x\phi'u^\ep+U\phi''v^\ep\\
\qquad-H_x\phi'h^\ep-H\phi''g^\ep=\ep\p_x^2u^\ep+\mu\partial^2_yu^\ep+r_1^\ep,\\
\partial_t h^\ep+\big[(u^\ep+U\phi')\partial_x+(v^\ep-U_x\phi)\partial_y\big]h^\ep-\big[(h^\ep+H\phi')\partial_x+(g^\ep-H_x\phi)\partial_y\big]u^\ep+H_x\phi'u^\ep+H\phi''v^\ep\\
\qquad-U_x\phi'h^\ep-U\phi''g^\ep=\ep\p_x^2h^\ep+\kappa\partial^2_y h^\ep+r_2^\ep,\\
\partial_xu^\ep+\partial_y v^\ep=0,\quad \partial_xh^\ep+\partial_y g^\ep=0,\\
(u^\ep,h^\ep)|_{t=0}=(u_{0}, h_{0})(x,y),\qquad%+\ep(\zeta_1^\ep,\zeta_2^\ep)(x,y)\triangleq (u_{0}^\ep,h_{0}^\ep)(x,y),\\%\quad h_1|_{t=0}=h_{10}(x,y),\\
(u^\ep,v^\ep,\partial_yh^\ep,g^\ep)|_{y=0}=0,
\end{array}
\right.
\end{align}
\iffalse
\begin{align}
\label{pr_app}
\left\{
\begin{array}{ll}
\partial_tu_1^\ep+(u_1^\ep\partial_x+u_2^\ep\partial_y)u_1^\ep-(h_1^\ep\partial_x+h_2^\ep\partial_y)h_1^\ep=\ep\p_x^2u_1^\ep+\mu\partial^2_yu_1^\ep-P_x^\ep,\\
\partial_th_1^\ep+\partial_y(u_2^\ep h_1^\ep-u_1^\ep h_2^\ep)=\ep\p_x^2h_1^\ep+\kappa\partial_y^2h_1^\ep,\\
\partial_xu_1^\ep+\partial_yu_2^\ep=0,\quad \partial_xh_1^\ep+\partial_yh_2^\ep=0,\\
(u_1^\ep,h_1^\ep)|_{t=0}=(u_{10}, h_{10})(x,y)+\ep(\zeta_1^\ep,\zeta_2^\ep)(x,y)\triangleq (u_{10}^\ep,h_{10}^\ep)(x,y),\\%\quad h_1|_{t=0}=h_{10}(x,y),\\
(u_1^\ep,u_2^\ep,\partial_yh_1^\ep,h_2^\ep)|_{y=0}=0,\quad
\lim\limits_{y\rightarrow+\infty}(u_1^\ep,h_1^\ep)=(U,H^\ep)(t,x).%\quad \lim\limits_{y\rightarrow+\infty}h_1=H(t,x).\\
\end{array}
\right.
\end{align}
\fi
where the source term 
\begin{align}\label{new_source}
	(r_1^\ep,r_2^\ep)(t,x,y)~=~(r_1,r_2)+\ep(\tilde r_1^\ep, \tilde r_2^\ep)(t,x,y).
\end{align}  %$\ep(\zeta_1^\ep,\zeta_2^\ep)$ %$\epsilon\zeta_\epsilon$ and $\epsilon\vartheta_\epsilon$ 
Here, $(r_1,r_2)$ is the source term of the original problem \eqref{bl_main}, and $(\tilde r_1^\ep, \tilde r_2^\ep)$ is constructed to ensure that the initial data $(u_{0}, h_{0})$ %$u_{10}(x,y)+\epsilon\zeta_\epsilon$ and $b_{10}(x,y)+\epsilon\vartheta_\epsilon$ 
also satisfies the compatibility conditions of \eqref{pr_app} up to the order of $m$. Actually, we can use the given functions $\p_t^i(u, h)(0,x,y), 0\leq i\leq m$, which can be derived from the equations and initial data of \eqref{bl_main} by induction with respect to $i$, and it follows that $\p_t^i(u, h)(0,x,y)$ can be expressed as polynomials of the spatial derivatives, up to order $2i$, of the initial data $(u_0,h_0)$.  Then, we may choose the corrector $(\tilde r_1^\ep, \tilde r_2^\ep)$ in the following form:
\begin{align}\label{modify}
	(\tilde r_1^\ep, \tilde r_2^\ep)(t,x,y)~:=~-\sum_{i=0}^m\Big(\frac{t^i}{i!}\p_x^2\p_t^i(u, h)(0,x,y)\Big),
\end{align}
which yields that  by direct calculation,
\[\p_t^i(u^\ep, h^\ep)(0,x,y)~=~\p_t^i(u, h)(0,x,y),\quad 0\leq i\leq m.\] 
Likewise, we can derive that $\psi^\ep:=\p_y^{-1}h^\epsilon$ satisfies
\begin{align*}%\label{eq-psi_ep}
\p_t \psi^\ep+\big[(u^\ep+U\phi')\p_x+(v^\ep-U_x\phi)\p_y\big]\psi^\ep+H_x\phi u^\ep+H\phi'v^\ep-\kappa\p_y^2\psi^\ep=r_3^\ep,
\end{align*}
where
\begin{align}\label{r_3ep}
	r_3^\ep~=~r_3-\ep\sum_{i=0}^m\Big(\frac{t^i}{i!}\int_0^y\p_x^2\p_t^ih(0,x,z)dz\Big)~:=~r_3+\ep\tilde r_3
\end{align}
with $r_3$ given by \eqref{r_3}. Moreover, we have for $\alpha=(\beta,k)=(\beta_1,\beta_2,k)$ with $|\alpha|\leq m$, 
\begin{align}\label{est_rmodify}
\|D^\alpha (\tilde r_1^\ep, \tilde r_2^\ep)(t)\|_{L_{l+k}^2(\Omega)},~\|\p_\tau^\beta \tilde r_3^\ep(t)\|_{L_{-1}^2(\Omega)}\leq \sum_{\beta_1\leq i\leq m}t^{i-\beta_1}\cp\Big(M_0+\|(u_0,h_0)\|_{H_l^{2i+2+\beta_2+k}}\Big).%\quad \alpha=(\beta,	
\end{align}

%Moreover, the outflow $(U,H^\ep,P_x^\ep)$ satisfies
%\begin{align}\label{Brou_ep}
	%U_t+UU_x-H^\ep H_x^\ep+P_x^\ep=\ep\p_x^2 U,\quad H_t^\ep+UH_x^\ep-H^\ep U_x=\ep\p_x^2H^\ep.
	%\end{align}

 Based on the a priori energy estimates established in Proposition \ref{prop-priori}, we can obtain 
\begin{prop}
\label{Th2}
Under the hypotheses of Theorem \ref{thm_main}, %andsuppose that the initial data $(u_{0}^\ep, h_{0}^\ep)$ satisfies the same assumptions as $(u_{0}, h_{0})$ given in Theorem \ref{thm_main}. 
 there exist a time $0<T_*\leq T$, independent of $\epsilon$, and a solution $(u^\ep,v^\ep,h^\ep,g^\ep)$ to the initial boundary value problem (\ref{pr_app}) with $(u^\ep, h^\ep)\in L^\infty\big(0,T_*; \h_l^m\big)$, which satisfies the following uniform estimates in $\epsilon$:
\begin{align}\label{est_modify1}
	\sup_{0\leq t\leq T_*}\big\|\big(u^\ep, h^\ep\big)(t)\big\|_{\h_l^m}\leq~2F^{\frac{1}{2}}_0,
\end{align}
where %$F^\ep_0$ is a polynomial of $\big\|\big(u_{0}^\ep, h_{0}^\ep\big)\big\|_{H_l^{2m}(\Omega)}$, endowed with the following definition (similar as 
$F_0$ is given by \eqref{F_def}).
%\begin{align}\label{def_Fep}
	%F^\ep_0~:=~\sum_{\substack{|\alpha|\leq m\\|\beta|\leq m-1}}\|D^\alpha(u^\ep, h^\ep)(0)\|_{L_l^2(\Omega)}^2+36\delta_0^{-4}\sum_{|\beta|=m}\big\|\big(u_{\beta}^\ep, h_{\beta}^\ep\big)(0)\big\|_{L^2_l(\Omega)}^2,
%\end{align} with \begin{align*}\begin{cases}
	%u_{\beta}^\ep(t,x,y)~:=~\p_\tau^\beta u^\ep(t,x,y)-\frac{\p_yu^\ep(t,x,y)+U(t,x)\phi''(y)}{h^\ep(t,x,y)+H(t,x)\phi'(y)}\int_0^y\p_\tau^\beta h^\ep(t,x,z)dz,\\
	%h_{\beta}^\ep(t,x,y)~:=~\p_\tau^\beta h^\ep(t,x,y)-\frac{\p_yh^\ep(t,x,y)+H(t,x)\phi''(y)}{h^\ep(t,x,y)+H(t,x)\phi'(y)}\int_0^y\p_\tau^\beta h^\ep(t,x,z)dz.\end{cases}\end{align*}
 Moreover,  for  $t\in[0,T_*], (x,y)\in\Omega,$
\begin{align}
\label{est_modify2}
\big\|\ly^{l+1}\p_y^i(u^\ep,h^\ep)(t)\big\|_{L^\infty(\Omega)}\leq~\delta_0^{-1},
\quad h^\epsilon(t,x,y)+H(t,x)\phi'(y)~\geq~\delta_0,\quad \quad i=1,2.
\end{align}
\end{prop}
\iffalse
we have\\
(i)If $\mu=k$, there exists an unique solution $(u_\epsilon,v_\epsilon, b_\epsilon,g_\epsilon)$ to the initial boundary value problem (\ref{pr_app}). Moreover,
\begin{align}
\label{3.3}
\|u_\epsilon,b_\epsilon\|_{H^{m}_l(\mathbb{R}^2_+)}\leq C\|u_{\epsilon0},b_{\epsilon0}\|_{H^{m}_l(\mathbb{R}^2_+)}\leq C\|u_{\epsilon0}, b_{\epsilon0}\|_{\tilde{H}^{2m}_l(\mathbb{R}^2_+)}.
\end{align}
(ii) If $\mu\neq k$, we assume additional conditions that $\|\p_yu_s\|_{L^2_{yl}(L^\infty_x)}$ is enough small, there exists an unique solution $(u_\epsilon,v_\epsilon, b_\epsilon,g_\epsilon)$ to the initial boundary value problem (\ref{pr_app}). And the estimates (\ref{3.3}) still hold true.
Here the constant $l>1/2$.
\end{thm}
\begin{rem}
In general, the lifespan of solutions to the initial-boundary value problem (\ref{pr_app}) depends on the small parameter. However, in Section 2, we have establish the uniform a priori energy estimates, which are independent of $\epsilon$. That means the lifespan of solutions to (\ref{pr_app}) is indeed uniform with respect to small parameter $\epsilon$.
\end{rem}
\fi

\begin{proof}[\textbf{Proof.}]
Since the problem \eqref{pr_app} is a parabolic system,  it is standard to show that \eqref{pr_app} admits a solution in a time interval $[0,T_\epsilon]$ ($T_\epsilon$ may depend on $\epsilon$) satisfying the estimates \eqref{est_modify2}.
%the local in time existence of solution to \eqref{pr_app} in $[0,T_\epsilon]$,  which may depend on $\epsilon$, is standard. 
Indeed, one can establish a priori estimates for \eqref{pr_app}, and then obtain the local existence of solution by the standard iteration and weak convergence methods. %Moreover, we can obtain the solution $(u_\ep, h^\ep)$ satisfies \eqref{est_modify2} in $[0,T_\ep]$. 

On the other hand,  we can derive the similar a priori estimates as in Proposition \ref{prop-priori} for \eqref{pr_app}, so by the standard continuity argument we can obtain the existence of solution in a time interval $[0,T_*], T_*>0$ independent of $\epsilon$. Therefore, we only determine the uniform lifespan $T_*$, and verify the estimates \eqref{est_modify1} and \eqref{est_modify2}. 

According to Proposition \ref{prop-priori}, we can obtain the estimates for $(u^\ep, h^\ep)$ similar as \eqref{est_priori-1}: 
\begin{align}\label{est_priori-ep1}
	\sup_{0\leq s\leq t}\|(u^\ep, h^\ep)(s)\|_{\h_l^m}~\leq~\big(F_0+\int_0^tF^\ep(s)ds\big)^{\frac{1}{2}}\cdot\Big\{1-2C\delta_0^{-8}\big(F_0+\int_0^tF^\ep(s)ds\big)^2t\Big\}^{-\frac{1}{4}},%F_0^\ep\Big[1-2C\delta_0^{-8}\big(M_0^2+(F_0^\ep)^2\big)^2t\Big]^{-\frac{1}{4}},
\end{align}
as long as the quantity in $\{\cdot\}$ on the right-hand side of \eqref{est_priori} is positive, where the quantity $F_0$ is given by \eqref{F_def}, and $F^\epsilon(t)$ is defined as follows (similar as \eqref{Ft_def}):
\begin{align}\label{Ft_def-ep}
	F^\ep(t)~:=~&C\sum_{\tiny\substack{|\alpha|\leq m\\|\beta|\leq m-1}}\|D^\alpha(r_1^\ep, r_2^\ep)(t)\|_{L^2_{l+k}(\Omega)}^2+C\delta_0^{-4}\sum_{|\beta|=m}\Big(\|\p_\tau^\beta (r_1^\ep, r_2^\ep)(t)\|_{L_l^2(\Omega)}^2+4\delta_0^{-4}\|\p_\tau^\beta r_3^\ep\|_{L_{-1}^2(\Omega)}^2\Big)\nonumber\\
	&+C\delta_0^{-8}\Big(1+\sum_{|\beta|\leq m+2}\|\p_\tau^\beta (U,H,P)(t)\|_{L^2(\bT_x)}^2\Big)^3.
\end{align}
Substituting \eqref{new_source}-\eqref{r_3ep} into \eqref{Ft_def-ep} and recalling $F(t)$ defined by \eqref{Ft_def}, it yields that
\begin{align*}
	F^\ep(t)~=~&F(t)+C\ep^2\sum_{\tiny\substack{|\alpha|\leq m\\|\beta|\leq m-1}}\|D^\alpha(\tilde r_1^\ep, \tilde r_2^\ep)(t)\|_{L^2_{l+k}(\Omega)}^2+C\ep^2\delta_0^{-4}\sum_{|\beta|=m}\Big(\|\p_\tau^\beta (\tilde r_1^\ep, \tilde r_2^\ep)(t)\|_{L_l^2(\Omega)}^2+4\delta_0^{-4}\|\p_\tau^\beta \tilde r_3^\ep\|_{L_{-1}^2(\Omega)}^2\Big),
\end{align*}
which implies that
from \eqref{est_Ft} and \eqref{est_rmodify}, %the chosen of $(\tilde r_1,\tilde r_2)$ given by \eqref{modify}, 
\begin{align*}
	F^\ep(t)
	\leq~ &C\delta_0^{-8}M_0^6+\ep^2\delta_0^{-8}\cp\big(M_0+\|(u_0,h_0)\|_{H_l^{3m+2}}\big)\leq~\delta_0^{-8}\cp\big(M_0+\|(u_0,h_0)\|_{H_l^{3m+2}}\big).
\end{align*}
Therefore, by choosing 
\begin{align*}
	T_1~:=~\min\Big\{\frac{\delta_0^8F_0}{\cp\big(M_0+\|(u_0,h_0)\|_{H_l^{3m+2}}\big)}, ~\frac{3\delta_0^8}{32CF_0^2}\Big\}
\end{align*}
in \eqref{est_priori-ep1}, we obtain \eqref{est_modify1} for $T_*\leq T_1$.

On the other hand, similar as the estimates \eqref{upbound_uy-1} and \eqref{h_lowbound-1} given in Proposition \ref{prop-priori}, we have the following bounds for $\ly^{l+1}\p_y^i(u^\ep,h^\ep), i=1,2$ and $h^\ep:$ 
\begin{align}\label{upbound_ep}
	\|\ly^{l+1}\p_y^i(u^\ep, h^\ep)(t)\|_{L^\infty(\Omega)}
	\leq~&\|\ly^{l+1}\p_y^i(u_0, h_0)\|_{L^\infty(\Omega)}+Ct\cdot\big(\sup_{0\leq s\leq t}\|(u^\ep, h^\ep)(s)\|_{\h_l^5}\big),\quad i=1,2, 
	%\big(F_0+\int_0^tF^\ep(s)ds\big)^{\frac{1}{2}}\nonumber\\&\cdot\Big\{1-2C\delta_0^{-8}\big(F_0+\int_0^tF^\ep(s)ds\big)^2t\Big\}^{-\frac{1}{4}},\quad i=1,2,
	\end{align}
and 
\begin{align}\label{lowbound_ep}
	h^\ep(t,x,y)\geq~&h_0(x,y)-Ct\cdot\big(\sup_{0\leq s\leq t}\|(u^\ep, h^\ep)(s)\|_{\h_0^3}\big).%\big(F_0+\int_0^tF^\ep(s)ds\big)^{\frac{1}{2}}\cdot\Big\{1-2C\delta_0^{-8}\big(F_0+\int_0^tF^\ep(s)ds\big)^2t\Big\}^{-\frac{1}{4}}.
\end{align}
Then, from the assumptions \eqref{ass_bound-modify} for the initial data $(u_0,h_0)$, and the chosen of $T_1$ above, we obtain that by \eqref{est_priori-ep1},
\begin{align*}
	&\|\ly^{l+1}\p_y^i(u^\ep, h^\ep)(t)\|_{L^\infty(\Omega)}
	\leq~(2\delta_0)^{-1}+2CF_0^{\frac{1}{2}}t,\quad i=1,2, \\
	&h^\ep(t,x,y)+H(t,x)\phi'(y)\geq 2\delta_0+\big(H(t,x)-H(0,x)\big)\phi'(y)-2CF_0^{\frac{1}{2}}t\geq2\delta_0-C\big(M_0+2F_0^{\frac{1}{2}}\big)t.
\end{align*}
So, let us choose
\begin{align*}
	T_2~:=~\min\Big\{T_1,~\frac{1}{4C\delta_0F_0^{\frac{1}{2}}},~\frac{\delta_0}{C\big(M_0+2F_0^{\frac{1}{2}}\big)}\Big\},
\end{align*}
then, \eqref{est_modify2} holds for $T_*= T_2.$ Therefore, we find the lifespan $T_*=T_2$ and establish the estimates \eqref{est_modify1} and \eqref{est_modify2}, and consequently complete the proof of this proposition. 

\iffalse
choose time $0<T_{\epsilon}'\leq T_\epsilon$ the lifespan such that 
\[
\|(u_\epsilon,b_\epsilon)\|_{\B^{3}_l(\Omega)}\leq 2\|(u_{\epsilon0},b_{\epsilon0})\|_{H^{6}_l(\mathbb{R}^2_+)},\quad \forall t\in[0,T_{\epsilon}'].
\] 	
%where $C_0>0$ is determined later. 
Then, by
\[b_\epsilon(t,x,y)=b_{\epsilon0}(x,y)+\int_0^t \p_tb_\epsilon(s,x,y)ds,\]
we that for $t\leq T_\epsilon',$
\[
|b_\epsilon(t,x,y)|\leq|b_{\epsilon0}(x,y)|+t\cdot\|\p_tb_\epsilon\|_{L^\infty(\Omega_{T_\epsilon'})}\leq |b_{\epsilon0}(x,y)|+Ct\cdot\|b_\epsilon\|_{\B_0^3(\Omega_{T_\epsilon'})}\leq |b_{\epsilon0}(x,y)|+2C\|(u_{\epsilon0},b_{\epsilon0})\|_{H^{6}_l(\mathbb{R}^2_+)}~t,
\]
which implies that
\begin{equation}\label{con_ass}
|b_\epsilon(t,x,y)|\leq\frac{1}{2},\quad\mbox{for}\quad t\in\big[0,T_1\big],
\end{equation}
where $T_1:=\min\Big\{T_\epsilon', \frac{1}{4CC_0\|(u_{\epsilon0},b_{\epsilon0})\|_{H^{6}_l(\mathbb{R}^2_+)}^{7}}\Big\}.$

Next, in view of \eqref{con_ass} and combining with Proposition \ref{prop_m}, we obtain that for $m\geq3,$
\[
\|(u,b)\|_{\B_l^m(\Omega)}^2+\|\p_y(u,b)\|^2_{\A_l^m(\Omega)}\leq
		 C_m\big\|(u_0,b_0)\big\|_{H_l^{2m}(\Omega)(\Omega)}^{2m+1}\cdot\exp\{C_1 t\},
\]
where the positive constant $C_m$ depends on $m$ and $\|(u_{\epsilon0},b_{\epsilon0})\|_{H^{6}_l(\mathbb{R}^2_+)}.$
\fi
\end{proof}

From the above Proposition \ref{Th2}, we obtain the local existence of solutions $(u^\ep,v^\ep,h^\ep,g^\ep)$ to the problem \eqref{pr_app} and their uniform estimates in $\epsilon$. Now, by letting $\ep\rightarrow0$ we will obtain the solution to the original problem \eqref{bl_main} through some compactness arguments. Indeed, from the uniform estimate \eqref{est_modify1},  by the Lions-Aubin lemma and the compact embedding of $H_l^m(\Omega)$ in $H_{loc}^{m'}$ for $m'<m$ (see \cite[Lemma 6.2]{MW1}), we know that there exists $(u, h)\in L^\infty\big(0,T_*;\h_l^m\big)\bigcap\Big(\bigcap_{m'<m-1}C^1\big([0,T_*]; H^{m'}_{loc}(\Omega)\big)\Big)$, such that, up to a subsequence,
\begin{align*}
	\p_t^i(u^\ep,h^\ep)~\stackrel{*}{\rightharpoonup}~\p_t^i(u, h),\qquad &\mbox{in}\quad L^\infty\big(0,T_*; H^{m-i}_l(\Omega)\big),\quad 0\leq i\leq m,\\
	(u^\ep,h^\ep)~\rightarrow~(u, h),\qquad &\mbox{in}\quad C^1\big([0,T_*]; H^{m'}_{loc}(\Omega)\big).%\quad 0\leq i\leq m-1,
\end{align*}
Then, by using the uniform convergence of $(\p_x u^\ep, \p_x h^\ep)$ because of $(\p_x u^\ep, \p_x h^\ep)\in Lip~(\Omega_{T_*})$, we get the pointwise convergence for $(v^\ep, g^\ep)$, i.e., 
\begin{align}\label{vg_limit}
	(v^\ep, g^\ep)=\big(-\int_0^y\p_x u^\ep dz, -\int_0^y\p_x h^\ep dz\big)\rightarrow\big(-\int_0^y\p_x u dz, -\int_0^y\p_x h dz\big):=(v,g).
\end{align}

Now, we can pass  the limit $\ep\rightarrow0$ in the problem \eqref{pr_app}, and obtain that $(u,v,h,g)$, $v$ and $g$ given by \eqref{vg_limit}, solves the original problem \eqref{bl_main}. As $(u, h)\in L^\infty\big(0,T_*;\h_l^m\big)$ it is easy to get that $(u, h)\in\bigcap_{i=0}^mW^{i,\infty}\Big(0,T;H_l^{m-i}(\Omega)\Big),$ and consequently \eqref{result_1} is proven. Moreover, the relation \eqref{result_2}, respectively \eqref{result_3}, follows immediately by combining the divergence free conditions $v=-\p_y^{-1}\p_xu, g=-\p_y^{-1}\p_xh$ with \eqref{normal1}, respectively \eqref{normal2}. %we obtain \eqref{result_2} and \eqref{result_3} respectively follows by \eqref{normal2}. 
Thus, we prove the local existence result of Theorem \ref{thm_main}.

\subsection{Uniqueness}
\indent\newline
We will show the uniqueness of the obtained solution to (\ref{bl_main}). %which shows the uniqueness of solutions. 
Let $(u^1,v^1, h^1,g^1)$ and $(u^2, v^2, h^2, g^2)$ be two solutions in $[0,T_*]$, constructed in the previous subsection, with respect to the initial data  $(u_0^1, h_0^1)$ and $(u_0^2, h_0^2)$ respectively. Set 
\[(\tilde{u}, \tilde v, \tilde{h}, \tilde g)=(u^1-u^2, v^1-v^2, h^1-h^2, g^1-g^2),\]
 then we have
 \begin{align}
\label{pr_diff}
\left\{
\begin{array}{ll}
\p_t \tilde{u}+\big[(u^1+U\phi')\p_x+(v^1-U_x\phi)\p_y\big]\tilde u-\big[(h^1+H\phi')\p_x+(g^1-H_x\phi)\p_y\big]\tilde h-\mu\p_y^2\tilde{u}\\
\qquad+(\p_xu^2+U_x\phi')\tilde u+(\p_yu^2+U\phi'')\tilde v-(\p_xh^2+H_x\phi')\tilde h-(\p_yh^2+H\phi'')\tilde g=0,\\
\p_t \tilde{h}+\big[(u^1+U\phi')\p_x+(v^1-U_x\phi)\p_y\big]\tilde h-\big[(h^1+H\phi')\p_x+(g^1-H_x\phi)\p_y\big]\tilde u-\kappa\p_y^2\tilde{h}\\
\qquad+(\p_xh^2+H_x\phi')\tilde u+(\p_yh^2+H\phi'')\tilde v-(\p_xu^2+U_x\phi')\tilde h-(\p_yu^2+U\phi'')\tilde g=0,\\
\p_x\tilde{u}+\p_y\tilde{v}=0,\quad \p_x\tilde{h}+\p_y\tilde{g}=0,\\
(\tilde{u}, \tilde h)|_{t=0}=(u_0^1-u_0^2,h_0^1-h_0^2),\quad %\tilde{b}_0=b_0^1-b_0^2,\\
(\tilde{u},\tilde v,\p_y\tilde h,\tilde g)|_{y=0}=\textbf{0}.%,\quad \p_y\tilde{h}|_{y=0}=0
\end{array}
\right.
\end{align}

Denote by $\tilde\psi:=\p_y^{-1}\tilde h=\p_y^{-1}(h^1-h^2)$, then from the second equation $\eqref{pr_diff}_2$ of $\tilde h$ and the divergence free conditions, we know that $\tilde\psi$ satisfies the following equation:
\begin{align}\label{eq_tpsi}
	\p_t \tilde{\psi}+\big[(u^1+U\phi')\p_x+(v^1-U_x\phi)\p_y\big]\tilde \psi-\big(g^2-H_x\phi)\tilde u+(h^2+H\phi')\tilde v-\kappa\p_y^2\tilde{\psi}=0.
\end{align}
Similar as \eqref{new_qu}, we introduce the new quantities:
\begin{align}\label{new_tqu}
	\bar u~:=~\tilde u-\frac{\p_y u^2+U\phi''}{h^2+H\phi'}\tilde\psi,\quad \bar h~:=~\tilde h-\frac{\p_y h^2+H\phi''}{h^2+H\phi'}\tilde\psi,
\end{align}
and then,
\begin{align}\label{new_tqu1}
	\bar u~:=~u^1-u^2-\eta_1^2~\p_y^{-1}(h^1-h^2),\quad \bar h~:=~ h^1-h^2-\eta_2^2~\p_y^{-1}(h^1-h^2),
\end{align}
where we denote 
\[\eta_1^2~:=~\frac{\p_y u^2+U\phi''}{h^2+H\phi'},\quad \eta_2^2~:=~\frac{\p_y h^2+H\phi''}{h^2+H\phi'}.\]
Next, we can obtain that through direct calculation,   $(\bar u,\bar h)$ admits the following initial-boundary value problem:
\begin{align}\label{eq_buh}\begin{cases}
	\p_t\bar u+\big[(u^1+U\phi')\p_x+(v^1-U_x\phi)\p_y\big]\bar u-\big[(h^1+H\phi')\p_x+(g^1-H_x\phi)\p_y\big]\bar h-\mu\p_y^2\bar{u}+(\kappa-\mu)\eta_1^2\p_y\bar h\\
	\qquad+a_1\bar u+b_1\bar h+c_1\tilde\psi=0,\\
	\p_t\bar h+\big[(u^1+U\phi')\p_x+(v^1-U_x\phi)\p_y\big]\bar h-\big[(h^1+H\phi')\p_x+(g^1-H_x\phi)\p_y\big]\bar u-\kappa\p_y^2\bar h\\
	\qquad+a_2\bar u+b_2\bar h+c_2\tilde\psi=0,\\
	(\bar u,\p_y\bar h)|_{y=0}=0,\quad
	(\bar u, \bar h)|_{t=0}=\big(u_0^1-u_0^2-\eta_{10}^2~\p_y^{-1}(h_0^1-h_0^2),h_0^1-h_0^2-\eta_{20}^2~\p_y^{-1}(h_0^1-h_0^2)\big),
\end{cases}\end{align}
where 
\begin{align}\label{nota_abc}
	a_1=&\p_xu^2+U_x\phi'+(g^2-H_x\phi)\eta_1^2,\quad b_1=(\kappa-\mu)\eta_1^2\eta_2^2-2\mu\p_y\eta_1^2-(\p_xh^2+H_x\phi')-(g^2-H_x\phi)\eta_2^2,\nonumber\\
	c_1=&\big[\p_t+(u^1+U\phi')\p_x+(v^1-U_x\phi)\p_y-\mu\p_y^2\big]\eta_1^2-\big[(h^1+H\phi')\p_x+(g^1-H_x\phi)\p_y\big]\eta_2^2-2\mu\eta_2^2\p_y\eta_1^2\nonumber\\
	&+(\kappa-\mu)\eta_1^2\big[(\eta_2^2)^2+\p_y\eta_2^2\big]+(g^2-H_x\phi)\big[(\eta_1^2)^2-(\eta_2^2)^2\big]+(\p_xu^2+U_x\phi')\eta_1^2-(\p_xh^2+H_x\phi')\eta_2^2,\nonumber\\
	a_2=&\p_xh^2+H_x\phi'+(g^2-H_x\phi)\eta_2^2,\quad b_2=-2\kappa\p_y\eta_2^2-(\p_xu^2+U_x\phi')-(g^2-H_x\phi)\eta_1^2,\nonumber\\
	c_2=&\big[\p_t+(u^1+U\phi')\p_x+(v^1-U_x\phi)\p_y-\kappa\p_y^2\big]\eta_2^2-\big[(h^1+H\phi')\p_x+(g^1-H_x\phi)\p_y\big]\eta_1^2-2\kappa\eta_2^2\p_y\eta_2^2\nonumber\\
	&+(\p_xh^2+H_x\phi')\eta_1^2-(\p_xu^2+U_x\phi')\eta_2^2,
\end{align}
and
\[\eta_{10}^2(x,y)~:=~\frac{\p_yu_0^2+U(0,x)\phi''(y)}{h^2_0+H(0,x)\phi'(y)},\quad\eta_{20}^2(x,y)~:=~\frac{\p_yh_0^2+H(0,x)\phi''(y)}{h^2_0+H(0,x)\phi'(y)}.\]

Combining \eqref{new_tqu} with the fact $\tilde\psi=\p_y^{-1}\tilde h$, we get that
\[\bar h~=~(h^2+H\phi')\cdot\p_y\Big(\frac{\tilde\psi}{h^2+H\phi'}\Big),\]
and then, by $\tilde\psi|_{y=0}=0,$
\begin{align}\label{tpsi}
%\bar h~=~(h^2+H\phi')\cdot\p_y\Big(\frac{\tilde\psi}{h^2+H\phi'}\Big),~\mbox{or}~
\tilde\psi(t,x,y)=\big(h^2(t,x,y)+H(t,x)\phi'(y)\big)\cdot\int_0^y\frac{\bar h(t,x,z)}{h^2(t,x,z)+H(t,x)\phi'(z)}dz.
\end{align}
Since $h^2+H\phi'\geq \delta_0$,  applying \eqref{normal1} in \eqref{tpsi} gives
\begin{align}\label{est_tpsi}
\Big\|\frac{\tilde\psi(t)}{1+y}\Big\|_{L^2(\Omega)}\leq 2\delta_0^{-1}\big\|h^2+H\phi'\big\|_{L^\infty([0,T_*]\times\Omega)}~\|\bar h(t)\|_{L^2(\Omega)}.
\end{align}
Moreover, through a similar process of getting the estimates \eqref{est_zeta},
we can obtain that there exists a constant $$C=C\Big(T_*,\delta_0, \phi, U, H,%\|(U,U_t,U_x)\|_{L^{\infty}(\bT)},\|(H,H_t,H_x)\|_{L^{\infty}(\bT)},
\|(u^1,h^1)\|_{\h_l^5},\|(u^2,h^2)\|_{\h_l^5}\Big)>0,$$ %depending on $\delta_0$
such that
\begin{align}\label{est_abc}
	\|a_i\|_{L^\infty([0,T_*]\times\Omega)},~\|b_i\|_{L^\infty([0,T_*]\times\Omega)},~\|(1+y)c_i\|_{L^\infty([0,T_*]\times\Omega)}~\leq~C,\quad i=1,2.
\end{align}
Thus, we have from \eqref{est_tpsi} and \eqref{est_abc},
\begin{align}\label{est_c}
	\|(c_i\tilde\psi)(t)\|_{L^2(\Omega)}~\leq~C~\|\bar h(t)\|_{L^2(\Omega)},\quad i=1,2.
\end{align}

\iffalse
with\begin{align*}
&\Gamma_1=-\tilde{u}\p_xu^2-\tilde v\p_y u^2+\tilde{h}\p_xh^2+\tilde g\p_y h^2,\quad\end{align*}and\begin{align*}
\Gamma_2=-\tilde{u}\p_xh^2-\tilde v\p_y h^2+\tilde{h}\p_xu^2+\tilde g\p_y u^2.\end{align*}
\begin{align}
\label{pr_diff}
\left\{
\begin{array}{ll}
\p_t \tilde{u}+(u_s+u^1)\p_x\tilde{u}+\tilde{v}\p_y(u_s+u^1)-(1+b^1)\p_x\tilde{b}-\tilde{g}\p_yb^1=\mu\p_y^2\tilde{u}+\Gamma_1,\\
\p_t \tilde{b}+(u_s+u^1)\p_x\tilde{b}+\tilde{v}\p_yb^1-(1+b^1)\p_x\tilde{u}-\tilde{g}\p_y(u_s+u^1)=\kappa\p_y^2\tilde{b}+\Gamma_2,\\
\p_x\tilde{u}+\p_y\tilde{v}=0,\quad \p_x\tilde{b}+\p_y\tilde{g}=0,\\
\tilde{u}_0=u_0^1-u_0^2,\quad \tilde{b}_0=b_0^1-b_0^2,\\
\tilde{u}|_{y=0}=0,\quad \p_y\tilde{b}|_{y=0}=0
\end{array}
\right.
\end{align}
with
\begin{align*}
\Gamma_1=-\tilde{u}\p_xu^2-v^2\p_y\tilde{u}+\tilde{b}\p_xb^2+g^2\p_y\tilde{b},
\end{align*}
and
\begin{align*}
\Gamma_2=-\tilde{u}\p_xb^2-v^2\p_y\tilde{b}+\tilde{b}\p_xu^2+g^2\p_y\tilde{u}.
\end{align*}
\fi
\begin{prop}
\label{Prop_uni}
Let $(u^1,v^1,h^1,g^1)$ and $(u^2, v^2, h^2,g^2)$ be two solutions of problem \eqref{bl_main} with respect to the initial data  $(u_0^1, h_0^1)$ and $(u_0^2, h_0^2)$ respectively, satisfying that $(u^j, h^j)\in\bigcap_{i=0}^mW^{i,\infty}\Big(0,T;H_l^{m-i}(\Omega)\Big)$ for $m\geq5,~j=1,2$. Then, there exists a positive constant 
$$C=C\Big(T_*,\delta_0, \phi, U, H,\|(u^1,h^1)\|_{\h_l^5},\|(u^2,h^2)\|_{\h_l^5}\Big)>0,$$
 such that for the quantities $(\bar u, \bar h)$ given by \eqref{new_tqu1},
\begin{align}\label{est_unique}
\frac{d}{dt}\|(\bar u, \bar h)(t)\|_{L^{2}(\Omega)}^2+\|(\p_y\bar u, \p_y\bar h)(t)\|_{L^2(\Omega)}^2
\leq ~&C\|(\bar u, \bar h)\|_{L^2(\Omega)}^2.
\end{align}
%&\frac{1}{2}\frac{d}{dt}\|(\tilde u, \tilde h)(t)\|_{L^{2}(\Omega)}^2+\mu\|\p_y\tilde u(t)\|_{L^2(\Omega)}^2+\kappa\|\p_y\tilde h(t)\|_{L^2(\Omega)}^2\nonumber\\
%\leq ~&C\Big(\|(U_x\phi', H_x\phi')(t)\|_{L^\infty(\Omega)}+\|\ly(U\phi'', H\phi'')(t)\|_{L^\infty(\Omega)}\nonumber\\
%&\quad+\|(\p_xu^2, \p_x h^2)(t)\|_{L^\infty(\Omega)}+\|\ly(\p_y u^2, \p_y h^2)(t)\|_{L^\infty(\Omega)}\Big)\|(\tilde u, \tilde h)\|_{L^2(\Omega)}^2.
%\|(\tilde u, \tilde h)\|_{\B_l^{m-1}(\Omega)}^2+\|\p_y(\tilde u,\tilde b)\|_{\A_l^{m-1}(\Omega)}^2~\leq ~\sum_{|\alpha|=m-1}\|D^\alpha(\tilde u,\tilde b)(0)\|_{L_l^2(\Omega)}^2+C\|(\tilde u, \tilde b)\|_{\A_l^{m-1}(\Omega)}^2,
%&\frac{d}{dt}(\|\tilde{u}\|_{H^{m-1}_l(\mathbb{R}^2_+)}+\|\tilde{b}\|_{H^{m-1}_l(\mathbb{R}^2_+)})+(\|\p_y\tilde{u}\|_{H^{m-1}_l(\mathbb{R}^2_+)}+\|\p_y\tilde{b}\|_{H^{m-1}_l(\mathbb{R}^2_+)})\nonumber\\
%\leq& C(\|\tilde{u}\|_{H^{m-1}_l(\mathbb{R}^2_+)}+\|\tilde{b}\|_{H^{m-1}_l(\mathbb{R}^2_+)}).
%\end{align}
%where the constant $C$ depends on the norms of $(u^1, b^1)$ and $(u^2, b^2)$ in $\B_l^m(\Omega)$.%$L^\infty\big(0,T; H^m_l(\mathbb{R}^2_+)\big)$.
\end{prop}
The above Proposition \ref{Prop_uni} can be proved by the standard energy method and the estimates \eqref{est_abc}, \eqref{est_c}, here we omit the proof for brevity of presentation. Then, by virtue of Proposition \ref{Prop_uni} we can prove the uniqueness of solutions to (\ref{bl_main}) as follows.%is similar as the uniform a priori energy estimates showed in Section 2. Here, we omit it for simplicity of presentation. And the uniqueness of solutions to (\ref{eq_main})-\eqref{ib_main} is a direct consequence of Proposition \ref{Prop_uni}.

Firstly, if the initial data satisfying $(u^1,h^1)|_{t=0}=(u^2,h^2)|_{t=0}$, then we know that from \eqref{eq_buh}, $(\bar u, \bar h)$ admits the zero initial data, which implies that $(\bar u, \bar h)\equiv0$ by applying Gronwall's lemma to \eqref{est_unique}. Secondly, it yields that $\tilde\psi\equiv0$ by plugging $\bar h\equiv0$ into \eqref{tpsi}. Then, from \eqref{new_tqu1} we have $(u^1,h^1)\equiv(u^2,h^2)$ immediately through the following calculation:
\begin{align*}
	(u^1,h^1)-(u^2,h^2)~=~(\tilde u, \tilde h)~=~(\bar u, \bar h)+(\eta_1^2, \eta_2^2)~\tilde \psi~\equiv~0.%\quadh^1-h^2=\tilde h=\bar h+\eta_2^2~\tilde \psi\equiv0.
\end{align*}
Finally, we obtain $(v^1,g^1)~\equiv~(v^2,g^2)$ since $v^i=-\p_y^{-1}\p_x u^i$ and $g^i=-\p_y^{-1}\p_x h^i$ for $i=1,2,$ and show the uniqueness of solutions.

\begin{rem}
	We mention that in the independent recent preprint \cite{G-P}, the authors give a systematic derivation of MHD boundary layer models, and consider the linearization for the similar system as \eqref{bl_mhd} around some shear flow. By using the analogous transformation to \eqref{new}, they obtain the linear stability for the system in the Sobolev framework.
\end{rem}

%%%%%%%%%%%%%%%%%%--------further discussion

\section{A coordinate transformation}
In this section, we will introduce another method to study the  initial-boundary value problem considered in this paper:
\begin{align}\label{pr_com}
\left\{
\begin{array}{ll}
\partial_tu_1+u_1\partial_xu_1+u_2\partial_yu_1=h_1\partial_x h_1+h_2\partial_y h_1+\mu\partial^2_yu_1,\\
\partial_th_1+\partial_y(u_2h_1-u_1h_2)=\kappa\partial_y^2h_1,\\
\partial_xu_1+\partial_yu_2=0,\quad \partial_xh_1+\partial_yh_2=0,\\
(u_1,u_2,\partial_yh_1,h_2)|_{y=0}=0,\quad
\lim\limits_{y\rightarrow+\infty}(u_1, h_1)=(U,H).%\quad \lim\limits_{y\rightarrow+\infty}b_1=B.
\end{array}
\right.
\end{align}
As we mentioned in Subsectin 2.3, by the divergence free condition,
\begin{align*}%\label{4.2}
\partial_x h_1+\partial_y h_2=0,
\end{align*}
there exists a stream function $\psi$, such that
\begin{align}
\label{4.3}
h_1=\p_y\psi,\quad h_2=-\p_x\psi,\quad \psi|_{y=0}=0,
\end{align}
\iffalse
Then, we obtain
\begin{align}
\label{4.4}\partial_t\p_y\psi+\partial_y(u_2\p_y\psi+u_1\p_x\psi)=k\partial_y^2\p_y\psi,
\end{align}
and the boundary conditions
\begin{align*}
\psi|_{y=0}=\p_x\psi|_{y=0}=\p_y^2\psi|_{y=0}=u_2|_{y=0}=0.
\end{align*}
Integrating equation (\ref{4.4}) with respect to the variable $y$ over $[0,y]$ yields
\fi
moreover, $\psi$ satisfies
\begin{align}
\label{4.5}
\p_t \psi+u_1\p_x\psi+u_2\p_y\psi=\kappa\p_y^2\psi.
\end{align}
Under the assumptions that
\begin{align}
\label{4.6}
h_1(t,x,y)>0,\quad \hbox{or}\quad \p_y\psi(t,x,y)>0,
\end{align}
we can introduce the following transformation
\begin{align}
\label{4.7}
\tau=t,\ \xi=x,\ \eta=\psi(t,x,y),
\end{align}
and then, \eqref{pr_com} can be written in the new coordinates as follows:
\begin{align}\label{pr_crocco}
\left\{
\begin{array}{ll}
\p_\tau u_1+u_1\p_\xi u_1-h_1\p_\xi h_1+(\kappa-\mu)h_1\p_\eta h_1\p_\eta u_1=\mu h_1^2\p^2_\eta u_1,\\
\p_\tau h_1-h_1\p_\xi u_1+u_1\p_\xi h_1 =\kappa h_1^2 \p^2_\eta h_1,\\
(u_1, h_1\p_\eta h_1)|_{y=0}=0,\quad
\lim\limits_{\eta\rightarrow+\infty}(u_1,h_1)=(U,H).
%\quad \lim\limits_{\eta\rightarrow+\infty}b_1=B.
\end{array}
\right.
\end{align}
\begin{rem}
The equations \eqref{pr_crocco} are quasi-linear equations, and there is no  loss of regularity term in \eqref{pr_crocco}, then we can use the classical Picard iteration scheme to establish the local existence.
However, in order to guarantee the coordinates transformation to be
valid, one needs to assume that $h_1(t,x,y)>0$. Moreover, one can obtain the stability of solutions to \eqref{pr_crocco} in the new coordintes $(\tau, \xi, \eta)$. It is necessary to transfer the well-posedness of solutions to the original equations \eqref{pr_com}. And then, there will be some loss of regularity.  
\end{rem}

\begin{rem}
	Based on the well-posedness result for MHD boundary layer in the Sobolev framework given in this paper, we will show the validity of the
vanishing limit of the  viscous MHD equations \eqref{eq_mhd} as $\epsilon\rightarrow0$ in a future work \cite{LXY}, that is, to show the solution to \eqref{eq_mhd} converges to a solution of ideal MHD equations, corresponding to $\ep=0$ in \eqref{eq_mhd}, outside the boundary layer, and to a boundary layer profile studied in this paper inside  the boundary layer. 
\end{rem}
%%%%%%%%%%%%%%%%%%%%%%%%%%%%%%%%%%%%%%%%%%%%%%%%%%%%%%%%%%%%%%%%%%%%%%%%%%%%%%%%%%%%%%%%%%%%%%%%%%%%%%

\appendix

%\section{Derivation }

\section{Some inequalities}
In this appendix, we will prove the inequalities given
 in Lemma \eqref{lemma_ineq}. Such inequalities can  be found in \cite{MW1} and \cite{X-Z}, here we give a proof for readers' convenience.

\begin{proof}[\textbf{Proof of Lemma \ref{lemma_ineq}.}]
	\romannumeral1)
	 From $\lim\limits_{y\rightarrow+\infty}(fg)(x,y)=0$, it yields
	\begin{align*}
		\Big|\int_{\bT_x}(fg)|_{y=0}dx\Big|~=~&\Big|\int_{\Omega}\p_y(fg)dxdy\Big|\leq \int_{\Omega}|\p_yf\cdot g|dxdy+\int_{\Omega}|f\cdot\p_yg|dxdy\\
		\leq~& \|\p_yf\|_{L^2(\Omega)}\|g\|_{L^2(\Omega)}+\|f\|_{L^2(\Omega)}\|\p_yg\|_{L^2(\Omega)},
	\end{align*}
	and we get \eqref{trace}. \eqref{trace0} follows immediately by letting $g=f$ in \eqref{trace}.
	
	\romannumeral 2)
	 From $m\geq3$ and $|\alpha|+|\tilde\alpha|\leq m$, we know that there must be $|\alpha|\leq m-2$ or $|\alpha|\leq m-2$. Without loss of generality, we assume that $|\alpha|\leq m-2$, then for any $l_1,l_2\geq0$ with $l_1+l_2=l$, we have that by using Sobolev embedding inequality, 
	\begin{align*}
		\big\|\big(D^\alpha f\cdot D^{\tilde\alpha}g\big)(t,\cdot)\big\|_{L^2_{l+k+\tilde k}(\Omega)}\leq~&\big\|\ly^{l_1+k}D^\alpha f(t,\cdot)\big\|_{L^\infty(\Omega)}\cdot \big\|\ly^{l_2+\tilde k}D^{\tilde\alpha}g(t,\cdot)\big\|_{L^2(\Omega)}\\
		\leq~& C\big\|\ly^{l_1+k}D^\alpha f(t,\cdot)\big\|_{H^2(\Omega)}\|g(t)\|_{\h_{l_2}^{|\tilde\alpha|}}\\
		\leq~ &C\big\|f(t)\big\|_{\h_{l_1}^{|\alpha|+2}(\Omega)}\|g(t)\|_{\h_{l_2}^{m}},
	\end{align*}  
	which implies \eqref{Morse} because of $|\alpha|+2\leq m$.
	
	\romannumeral3) 
	For $\lambda>\frac{1}{2},$ it follows that by integration by parts,
	\begin{align*}
		\big\|\ly^{-\lambda}(\p_y^{-1}f)(y)\big\|_{L^2_y(\bR_+)}^2=&\int_0^{+\infty}\frac{\big[(\p_y^{-1}f)(y)\big]^2}{1-2\lambda}d(1+y)^{1-2\lambda}=\frac{2}{2\lambda-1}\int_0^{+\infty}(1+y)^{1-2\lambda}f(y)\cdot(\p_y^{-1}f)(y)dy\\
		\leq~& \frac{2}{2\lambda-1}\big\|\ly^{-\lambda}(\p_y^{-1}f)(y)\big\|_{L^2_y(\bR_+)}\cdot\big\|\ly^{1-\lambda}f(y)\big\|_{L^2_y(\bR_+)},
	\end{align*}
	which implies the first inequality of \eqref{normal}. 
	
	On the other hand,  note that for $\tilde\lambda>0,$
	\begin{align*}
		|(\p_y^{-1}f)(y)|\leq~&\int_0^y|f(z)|dz\leq \|(1+z)^{1-\tilde\lambda}f(z)\|_{L^\infty(0,y)}\cdot\int_0^y(1+z)^{\tilde\lambda-1}dz\\
		\leq~&\frac{(1+y)^{\tilde\lambda}-1}{\tilde\lambda}\|(1+y)^{1-\tilde\lambda}f(y)\|_{L^\infty_y(\bR_+)},
	\end{align*}
	which implies the second inequality of \eqref{normal} immediately.
	%and then,\begin{align*}
		%\big\|\ly^{-\tilde\lambda}(\p_y^{-1}f)(y)\big\|_{L^\infty_y(\bR_+)}\leq \frac{1}{\tilde\lambda}\big\|\ly^{1-\tilde\lambda}f(y)\big\|_{L^\infty_y(\bR_+)}\|\ly^{-\tilde\lambda}\big((1+y)^{\tilde\lambda}-1\big)\|_{L^\infty_y}\end{align*}
	
Next, as $m\geq3$ and $|\alpha|+|\tilde\beta|\leq m$, we also get $|\alpha|\leq m-2$ or $|\tilde\beta|\leq m-2$. If $|\alpha|\leq m-2$, by using Sobolev embedding inequality and the first inequality of \eqref{normal}, we have for any $\lambda>\frac{1}{2}$,
\begin{align*}	
	\big\|\big(D^\alpha g\cdot\p_\tau^{\tilde\beta}\p_y^{-1}h\big)(t,\cdot)\big\|_{L^2_{l+k}(\Omega)}
	\leq~&\big\|\ly^{l+\lambda+k}D^\alpha g(t,\cdot)\big\|_{L^\infty(\Omega)}\cdot \big\|\ly^{-\lambda}\p_\tau^{\tilde\beta}\p_y^{-1}h(t,\cdot)\big\|_{L^2(\Omega)}\\
	\leq~&C\big\|\ly^{l+\lambda+k}D^\alpha g(t,\cdot)\big\|_{H^2(\Omega)}\cdot \big\|\ly^{1-\lambda}\p_\tau^{\tilde\beta}h(t,\cdot)\big\|_{L^2(\Omega)}\\
	\leq~&C\|g(t)\|_{\h_{l+\lambda}^{|\alpha|+2}}\|h(t)\|_{\h_{1-\lambda}^{|\tilde\beta|}}.
\end{align*}
If $|\tilde\beta|\leq m-2$, by Sobolev embedding inequality and the second inequality of \eqref{normal},
\begin{align*}	
	\big\|\big(D^\alpha g\cdot\p_\tau^{\tilde\beta}\p_y^{-1}h\big)(t,\cdot)\big\|_{L^2_{l+k}(\Omega)}
	\leq~&\big\|\ly^{l+\lambda+k}D^\alpha g(t,\cdot)\big\|_{L^2(\Omega)}\cdot \big\|\ly^{-\lambda}\p_\tau^{\tilde\beta}\p_y^{-1}h(t,\cdot)\big\|_{L^\infty(\Omega)}\\
	\leq~&C\|g(t)\|_{\h_{l+\lambda}^{|\alpha|}}\cdot \big\|\ly^{1-\lambda}\p_\tau^{\tilde\beta}h(t,\cdot)\big\|_{H^2(\Omega)}\\
	\leq~&C\|g(t)\|_{\h_{l+\lambda}^{|\alpha|}}\|h(t)\|_{\h_{1-\lambda}^{|\tilde\beta|+2}}.
\end{align*}
Thus, we get the proof of \eqref{normal0}, and then, \eqref{normal1} follows by letting $\lambda=1$ in \eqref{normal0}.

\romannumeral4) 
For any $\lambda>\frac{1}{2}$, 
	\begin{align*}
		\big|(\p_y^{-1}f)(y)\big|\leq\|f(y)\|_{L_y^1(\bR_+^2)}\leq \|\ly^{-\lambda}\|_{L_y^2(\bR_+)}\|\ly^{\lambda}f\|_{L_{y}^2(\bR_+)}\leq C\|\ly^{\lambda}f\|_{L_{y}^2(\bR_+)},
	\end{align*}
and we get \eqref{normal2}. 
	
For $m\geq2$ and $|\alpha|+|\tilde\beta|\leq m$, we get that $|\alpha|\leq m-1$ or $|\tilde\beta|\leq m-1$.
If $|\alpha|\leq m-1$, by using Sobolev embedding inequality and \eqref{normal2}, we have for any $\lambda>\frac{1}{2}$,
\begin{align*}	
	\big\|\big(D^\alpha f\cdot\p_\tau^{\tilde\beta}\p_y^{-1}g\big)(t,\cdot)\big\|_{L^2_{l+k}(\Omega)}
	\leq~&\big\|\ly^{l+k}D^\alpha f(t,\cdot)\big\|_{L_x^\infty L^2_y(\Omega)}\cdot \big\|\p_\tau^{\tilde\beta}\p_y^{-1}g(t,\cdot)\big\|_{L_x^2L^\infty_y(\Omega)}\\
	\leq~&C\big\|\ly^{l+k}D^\alpha f(t,\cdot)\big\|_{H^1(\Omega)}\cdot \big\|\ly^{\lambda}\p_\tau^{\tilde\beta}g(t,\cdot)\big\|_{L^2(\Omega)}\\
	\leq~&C\|f(t)\|_{\h_{l}^{|\alpha|+1}}\|g(t)\|_{\h_{\lambda}^{|\tilde\beta|}}.
\end{align*}
If $|\tilde\beta|\leq m-1$, by Sobolev embedding inequality and \eqref{normal2},
\begin{align*}	
	\big\|\big(D^\alpha f\cdot\p_\tau^{\tilde\beta}\p_y^{-1}g\big)(t,\cdot)\big\|_{L^2_{l+k}(\Omega)}
	\leq~&\big\|\ly^{l+k}D^\alpha f(t,\cdot)\big\|_{L^2(\Omega)}\cdot \big\|\p_\tau^{\tilde\beta}\p_y^{-1}g(t,\cdot)\big\|_{L^\infty(\Omega)}\\
	\leq~&C\|f(t)\|_{\h_{l}^{|\alpha|}}\cdot \big\|\ly^{\lambda}\p_\tau^{\tilde\beta}g(t,\cdot)\big\|_{H_1^xL_y^2(\Omega)}\\
	\leq~&C\|f(t)\|_{\h_{l}^{|\alpha|}}\|g(t)\|_{\h_{\lambda}^{|\tilde\beta|+1}}.
\end{align*}
Thus, we get \eqref{normal3}, and then complete the proof of this lemma. 

\end{proof}

\medskip \noindent
{\bf Acknowledgements:}
The second author is partially supported by NSFC (Grant No.11171213, No.11571231). The third author is  supported by the General Research Fund of Hong Kong, CityU No. 11320016, and he would like to thank Pierre Degond for the
initial discussion on this problem at Imperial College.

\bibliographystyle{plain}

\end{document}